\RequirePackage[l2tabu, orthodox]{nag}
\documentclass[a4paper]{amsart}

\usepackage{amsmath}
\usepackage{amsthm}
\usepackage{amssymb}
\usepackage{subcaption}
\usepackage{graphicx}
\usepackage[format=hang]{caption}

\newtheorem{theorem}{Theorem}[section]
\newtheorem{lemma}[theorem]{Lemma}
\newtheorem{corollary}[theorem]{Corollary}

\theoremstyle{definition}
\newtheorem{definition}[theorem]{Definition}

\theoremstyle{remark}
\newtheorem{remark}[theorem]{Remark}
\newtheorem{assumption}[theorem]{Assumption}

\usepackage{changepage}
\newenvironment{myminipage}{\begin{adjustwidth}{.025\textwidth}{0\textwidth}}{\end{adjustwidth}}

\usepackage[pdfproducer=latex, 
              pdfstartview=FitH, 
              bookmarksopen=true, colorlinks=true, 
              bookmarksnumbered=true  
             ]{hyperref}
\allowdisplaybreaks
\title[Quasi-Optimal Error Estimates for the Stabilized Navier-Stokes Problem]
{Quasi-Optimal Error Estimates for the Incompressible Navier-Stokes Problem Discretized by Finite Element Methods and Pressure-Correction Projection with Velocity Stabilization}
\thanks{The work of the first author was supported by CRC 963 founded by German research council (DFG).}
\thanks{The work of the second author was supported by the RTG 1023 founded by German research council (DFG).}
\makeatletter
\g@addto@macro{\endabstract}{\@setabstract}
\newcommand{\authorfootnotes}{\renewcommand\thefootnote{\@fnsymbol\c@footnote}}%
\makeatother
\begin{document}
\begin{center}
  \LARGE 
  Quasi-Optimal Error Estimates for the Incompressible Navier-Stokes Problem Discretized by Finite Element Methods and Pressure-Correction Projection with Velocity Stabilization \par \bigskip

  \normalsize
  \authorfootnotes
  Daniel Arndt\textsuperscript{1}, 
  Helene Dallmann\textsuperscript{1} and
  Gert Lube\textsuperscript{1} \par \bigskip

  \textsuperscript{1}Institute for Numerical and Applied Mathematics, \\Georg-August University of G\"ottingen, D-37083, 
  Germany\\
  \texttt{d.arndt/h.dallmann/lube@math.uni-goettingen.de}\par \bigskip

\end{center}

\begin{abstract}
We consider error estimates for the fully discretized instationary Navier-Stokes problem.
For the spatial approximation we use conforming inf-sup stable finite element methods in conjunction with grad-div and local projection stabilization 
acting on the streamline derivative. For the temporal discretization a pressure-correction projection algorithm based 
on BDF2 is used. 
We can show quasi-optimal rates of convergence with respect to time and spatial discretization for all considered error measures. Some of the error estimates are quasi-robust with respect to the Reynolds number.
\end{abstract}
%
%
%

\section{Introduction}\label{section:introduction}
We consider the time-dependent Navier-Stokes equations
\begin{align}
\begin{aligned}
\label{NavSto_1}
 \partial_t {\boldsymbol{u}} -\nu\Delta {\boldsymbol{u}} +({\boldsymbol{u}}\cdot\nabla) {\boldsymbol{u}} +\nabla p &= {\boldsymbol{f}} \, \quad
          \mbox{in}~~(t_0,T) \times \Omega, \\
  \nabla \cdot {\boldsymbol{u}} &=0 \, \quad \mbox{in}~~(t_0,T) \times \Omega, \\
 {\boldsymbol{u}} &= {\boldsymbol{0}} \, \quad \mbox{in}~~(t_0,T)\times \partial\Omega, \\
 {\boldsymbol{u}}(0,\cdot) &= {\boldsymbol{u}}_0(\cdot) \, \quad \mbox{in}~~\Omega
\end{aligned}
\end{align}
in a bounded polyhedral domain $\Omega \subset {\mathbb R}^d $,~$d\in \{2,3\}$.
Here ${\boldsymbol{u}} \colon (t_0,T) \times \Omega \to {\mathbb R}^d$ and $p \colon (t_0,T) \times \Omega \to {\mathbb R}$ denote
the unknown velocity and pressure fields for given viscosity $\nu>0$ and external forces
${\boldsymbol{f}}\in [L^2(t_0,T;[L^2(\Omega)]^d) \cap C(t_0,T;[L^2(\Omega)]^d)]^d$.

For the discretization with respect to time we use a splitting method called (standard) incremental pressure-correction projection method which is based on the backward differentiation formula of
second order (BDF2). In the continuous problem $\boldsymbol{u}$ and $p$ are coupled through the incompressibility constraint. The idea for pressure-correction projection methods is to define an
auxiliary variable $\widetilde{\boldsymbol{u}}$ and solve for $\widetilde{\boldsymbol{u}}$ and $p$ in two different steps such that the original velocity can be recovered from these two quantities.
Such an approach was first considered by Chorin \cite{chorin1969} and Temam \cite{temam1969}.
An overview over different projection methods is given in \cite{guermond2006overview}.
Badia and Codina \cite{badia2007convergence} analyzed the incremental pressure-correction algorithm with BDF1 time discretization.
The incremental pressure-correction algorithm with BDF2 time discretization is discussed by Guermond in \cite{guermond1999resultat} for the unstabilized Navier-Stokes equations with $\nu=1$. Shen
considered a different second order time discretization scheme in \cite{shen1996error}. It turns out that this technique suffers from unphysical boundary conditions for the pressure that lead to
reduced rates of convergence. To prevent this Timmermans proposed in \cite{Timmermans1996} the rotational pressure-correction projection method that uses a divergence correction for
the pressure. A thorough analysis for this has first been performed in \cite{guermond2004error} for the Stokes problem.

For the spatial stabilization of the Navier-Stokes equations or related problems there exist many different approaches. 
Residual-based stabilization methods penalize the residual of the differential equation in the strong formulation and are hence consistent. See \cite{roos2008robust} for an overview. 
The bulk of non-symmetric form of the stabilization terms and the occurrence of second order derivatives in the residual are drawbacks regarding the efficiency of this method. 
For the fully discretized case a local projection stabilization (LPS) PSPG-type stabilization for the discrete pressure is combined with LPS for the convective term in \cite{rebollo2015numerical}. Stability and convergence is proven using an 
semi-implicit Euler scheme for the discretization in time.

In \cite{arndt2014local} we considered the semi-discretized time-dependent incompressible Navier-Stokes problem where the spatial discretization had been performed with inf-sup stable finite
element methods with grad-div stabilization and a stream line upwind local 
projection stabilization (LPS SU) of the convective term (see also Section 
\ref{sec:2} of the present paper). Inspired by the ideas in  
\cite{burman2007continuous} for edge stabilized methods with equal order ansatz 
spaces, we were able to prove quasi-robust error estimates in case smooth 
solutions satisfy $\boldsymbol{u} \in [L^\infty (t_0,T;W^{1,\infty}(\Omega))]^d$ 
(which ensures uniqueness of the Navier-Stokes solution). The latter means that 
coefficients
of the right hand side of the error estimate may depend on Sobolev norms of the solution $(\boldsymbol{u},p)$ but not on critical physical parameters like $1/\nu$.

In the present paper, we extend the analysis to the fully discrete incremental pressure-correction algorithm with BDF2 time discretization, see Section~\ref{sec:pressure-correction}. The proposed approach is based
on the paper \cite{guermond1999resultat} by Guermond (and preliminary considerations in \cite{technicalreportAD2015}).
It turns out that grad-div stabilization is again essential for the derivation of quasi-robust error estimates whereas the LPS gives additional control of dissipative terms, 
see Section~\ref{sec:full1} with the main result in Theorem~\ref{thm:ansatz_guermond}. As the result is quasi-optimal in the spatial variables, it is not optimal in time. Therefore, in Section~\ref{sec:temporal-spatial} we modify the analysis in
\cite{guermond1999resultat} to improve the order of the temporal discretization. Although the resulting estimates are not quasi-robust, we nevertheless consider the dependence
of the error estimates on $\nu$ and the choice of appropriate bounds for the stabilization parameters. In all the cases, it turns out that grad-div stabilization is essential while the LPS
can be neglected for deriving quasi-robust error estimates. Therefore, in Section~\ref{sec:numerics} numerical examples consider the confirmation of the analytical results with respect to rates of convergence
as well as the the influence of the SU stabilization as subgrid model.
A critical discussion of the results, can be found in Section~\ref{section:discussion}.

\section{Stabilized Finite Element Discretization for the Navier-Stokes Problem} \label{sec:2}

In this section, we describe the model problem and the spatial semi-discretization based on
inf-sup stable interpolation of velocity and pressure together with local projection
stabilization.

\subsection{Time-Dependent Navier-Stokes Problem}

In the following, we will consider the usual Sobolev spaces $W^{m, p}(\Omega)$ with norm $\| \cdot\|_{W^{m, p}(\Omega)}, 
m \in {\mathbb N}_0, p \ge 1$. In particular, we have $L^p(\Omega)=W^{0,p}(\Omega)$ and denote $H^m(\Omega):=W^{m,2}(\Omega)$. Moreover, the closed subspaces 
 $H_0^1(\Omega):=W_0^{1,2}(\Omega)$, consisting of functions in $W^{1,2}(\Omega)$ with zero trace on $\partial \Omega$, and $L_0^2(\Omega)$, 
consisting of $L^2$-functions with zero mean in $\Omega$, will be used. The inner product in $L^2(D)$ with 
$D \subseteq \Omega$ will be denoted by $(\cdot,\cdot)_D$. In case of $D=\Omega$ 
we omit the index.\\
The variational formulation of problem (\ref{NavSto_1}) reads: 
\begin{center}
\begin{myminipage}
Find $\boldsymbol{\mathcal{U}}=(\boldsymbol{u},p) \in [L^2(t_0,T;\boldsymbol{V}) \cap L^\infty(t_0,T;[L^2(\Omega)]^d)] \times L^2(t_0,T;Q)$ where 
$\boldsymbol{V}\times Q :=
[W^{1,2}_0(\Omega)]^d\times L_0^2(\Omega)$ such that
\begin{align}
\label{Varform_1}
 (\partial_t \boldsymbol{u}, \boldsymbol{v}) + A_G(\boldsymbol{u};\boldsymbol{\mathcal{U}}, \boldsymbol{\mathcal{V}})  & = 
(\boldsymbol{f}, \boldsymbol{v})~~\forall \boldsymbol{\mathcal{V}} =(\boldsymbol{v}, q)\in
\boldsymbol{V}\times Q
\end{align}
with the Galerkin form
\begin{align}
\begin{aligned}
\label{Varform_2} 
A_G(\boldsymbol{w};\boldsymbol{\mathcal{U}}, \boldsymbol{\mathcal{V}}) := &  \underbrace{\nu(\nabla \boldsymbol{u}, \nabla 
\boldsymbol{v})-(p, \nabla\cdot \boldsymbol{v})
                                 +(q, \nabla\cdot \boldsymbol{u})}_{=:a_G(\boldsymbol{\mathcal{U}},\boldsymbol{\mathcal{V}})}  \\&
    + \underbrace{\frac12\big[ ((\boldsymbol{w}\cdot\nabla)\boldsymbol{u}, \boldsymbol{v})- 
            ((\boldsymbol{w}\cdot\nabla)\boldsymbol{v}, \boldsymbol{u}) \big]}_{= 
c(\boldsymbol{w};\boldsymbol{u},\boldsymbol{v})}.
\end{aligned}
\end{align}
\end{myminipage}
\end{center}
The skew-symmetric form of the convective term $c$ is chosen for conservation purposes.
In this paper, we will additionally assume that the velocity field $\boldsymbol{u}$ belongs to $L^\infty
(t_0,T;[W^{1,\infty}(\Omega)]^d)$
which ensures uniqueness of the solution.

\subsection{Finite Element Spaces}

For a simplex $T \in {\mathcal T}_h$ or a quadrilateral/hexahedron $T$ in ${\mathbb R}^d$, let 
$\hat T$ be the reference unit simplex or the unit cube $(-1,1)^d$. The bijective reference mapping 
$F_T \colon \hat T \to T$ is affine for simplices and multi-linear for quadrilaterals/hexahedra. 
Let $\hat {\mathbb P}_l$ and $\hat {\mathbb Q}_l$ with $l \in {\mathbb N}_0$ be the set of 
polynomials of degree $\leq l$ and of polynomials of degree $\leq l$ in each variable separately. 
Moreover, we set 
\[
  {\mathbb R}_l(\hat T) := \left \lbrace \begin{tabular}{ll}
      ${\mathbb P}_l(\hat T)$ & on simplices $\hat T$ \\
      ${\mathbb Q}_l(\hat T)$ & on quadrilaterals/hexahedra $\hat T$. 
   \end{tabular} \right.
\]
Define
\begin{align*}
   Y_{h,-l} & :=  \lbrace \boldsymbol{v}_h \in L^2(\Omega):~\boldsymbol{v}_h|_T \circ F_T \in {\mathbb R}_l(\hat T)~
            \forall T \in {\mathcal T}_h \rbrace, \\
   Y_{h,l} & := Y_{h,-l} \cap W^{1,2}(\Omega).
\end{align*}

For convenience, we write $\boldsymbol{V}\!_h = {\mathbb R}_{k_u}$ instead of $\boldsymbol{V}\!_h = [Y_{h,k_u}]^d \cap \boldsymbol{V}$ 
and $Q_h = {\mathbb R}_{\pm k_p}$ instead of 
$Q_h = Y_{h,\pm k_p} \cap Q$.

\begin{assumption}
 \label{assumption-A.1}
Let $\boldsymbol{V}\!_h$ and $Q_h$ be finite element spaces 
satisfying a discrete inf-sup-condition 
\begin{equation}
\label{Babuska-Brezzi}
  \inf_{q \in Q_h \setminus  \lbrace 0 \rbrace }  \sup_{\boldsymbol{v} \in \boldsymbol{V}\!_h \setminus \lbrace \boldsymbol{0} \rbrace } 
       \frac{(\nabla \cdot {\boldsymbol{v}}, q)}{\|\nabla \boldsymbol{v}\|_0 \|q\|_0 } \ge \beta > 0
\end{equation}
with a constant $\beta$ independent on $h$.
\end{assumption}

\subsection{Stabilization}

For a Galerkin approximation of problem (\ref{Varform_1})-(\ref{Varform_2}) on an admissible partition 
${\mathcal T}_h$ of the polyhedral domain $\Omega$, consider finite dimensional spaces 
$\boldsymbol{V}\!_h\times Q_h\subset \boldsymbol{V}\times Q$. 
Then, the semi-discretized problem reads:\\
 Find $\boldsymbol{\mathcal{U}}_h=(\boldsymbol{u}_h,p_h) \colon (t_0,T) \to \boldsymbol{V}\!_h\times Q_h$ such that for all 
$\boldsymbol{\mathcal{V}}_h = (\boldsymbol{v}_h,q_h)\in \boldsymbol{V}\!_h\times Q_h$:
\begin{align}
\label{Galerkin}
 (\partial_t \boldsymbol{u}_h,\boldsymbol{v}_h) + A_G(\boldsymbol{u}_h;\boldsymbol{\mathcal{U}}_h,\boldsymbol{\mathcal{V}}_h)  = 
(\boldsymbol{f}, \boldsymbol{v}_h). 
\end{align}

The semi-discrete Galerkin solution of problem (\ref{Galerkin}) may suffer from spurious oscillations 
due to poor mass conservation or dominating advection. 
The idea of local projection stabilization (LPS) methods is to separate discrete 
function spaces into small and large scales and to add stabilization terms only on small scales. 

Let $\lbrace {\mathcal M}_h \rbrace$ be a family of shape-regular macro decompositions of $\Omega$ into 
$d$-simplices, quadrilaterals ($d=2$) or hexahedra ($d=3$). In the one-level LPS-approach, one has
${\mathcal M}_h = {\mathcal T}_h$. In the two-level LPS-approach, the 
decomposition ${\mathcal T}_h$ is derived from ${\mathcal M}_h$ by barycentric refinement of $d$-simplices 
or regular (dyadic) refinement of quadrilaterals and hexahedra. We denote by $h_T$ and $h$ the diameter
of cells $T \in {\mathcal T}_h$ and $M \in {\mathcal M}_h$. It holds $h_T \leq h \leq Ch_T$ for all 
$T \subset M$ and $M \in {\mathcal M}_h$.

\begin{assumption}\label{assumption-A.2}
 Let the finite element space $\boldsymbol{V}\!_h$ satisfy the local inverse inequality
\begin{equation}
\label{loc-inv}
    \|\nabla \boldsymbol{v}_h\|_{0,M} \leq C h^{-1} \| \boldsymbol{v}_h\|_{0,M} \quad \forall \boldsymbol{v}_h \in \boldsymbol{V}\!_h,~M \in {\mathcal M}_h.
\end{equation}
\end{assumption}

\begin{assumption} \label{assumption-A.3}
There are quasi-interpolation operators $j_u \colon \boldsymbol{V}  \to \boldsymbol{V}\!_h$ and 
$j_p \colon Q \to Q_h$ such that for all $M \in {\mathcal M}_h$, for all 
${\boldsymbol{w}} \in \boldsymbol{V} \cap [W^{l,2}(\Omega)]^d$ with $1 \leq l \leq k_u+1$:
\begin{equation} \label{approx-1}
  \|{\boldsymbol{w}}-j_u{\boldsymbol{w}}\|_{0,M} + h\| \nabla ({\boldsymbol{w}} - j_u{\boldsymbol{w}})\|_{0,M} 
              \leq C h^l \|{\boldsymbol{w}}\|_{W^{l,2}(\omega_M)}
\end{equation}
and for all $q \in Q \cap H^l(M)$ with $1\leq l \leq k_p+1$:
\begin{equation} \label{approx-2}
 \|q-j_p q\|_{0,M} + h \| \nabla (q-j_p q)\|_{0,M} \leq C h^l \| q \|_{W^{l,2}(\omega_M)}.
\end{equation}
on a suitable patch $\omega_M \supset M$. Moreover, let
\[
  \| {\boldsymbol{v}} - j_u{\boldsymbol{v}} \|_{L^\infty(M)} \leq C h |{\boldsymbol{v}}|_{W^{1,\infty}(M)} 
     \quad \forall {\boldsymbol{v}} \in [W^{1,\infty}(M)]^d.
\]
\end{assumption}

Let $\boldsymbol{D}_M \subset [L^\infty(M)]^d$ denote a  finite element space on $M \in \mathcal M_h$ for ${\boldsymbol{u}}_h$.
For each $M \in {\mathcal M}_h$, let $\pi_M \colon [L^2(M)]^d \to \boldsymbol{D}_M$ be the orthogonal 
$L^2$-projection. Moreover, we denote by $\kappa_M :=id-\pi_M$ the so-called fluctuation operator. 

\begin{assumption}\label{assumption-A.4}
The fluctuation operator $\kappa_M =id-\pi_M$ provides the approximation property
(depending on $\boldsymbol{D}_M$ 
and $ s \in \lbrace 0,\cdots, k_u \rbrace$):
\begin{equation} \label{approx-3}
  \| \kappa_M {\boldsymbol{w}} \|_{0,M} \leq C h^l \|{\boldsymbol{w}}\|_{W^{l,2}(M)},~
      \forall \boldsymbol{w} \in W^{l,2}(M),~M \in {\mathcal M}_h,~l=0, \ldots, s.
\end{equation}
\end{assumption}
~\\
A sufficient condition for Assumption~\ref{assumption-A.4} is $[{\mathbb 
P}_{s-1}]^d \subset \boldsymbol{D}_M$.

In this work, we restrict ourselves to inf-sup stable element pairs.
This means that the space of discretely divergence-free functions is non-empty:
 \begin{align}
  \boldsymbol{V}\!_h^{\,div}:=\{\boldsymbol{v}_h\in\boldsymbol{V}\!_h: (\nabla \cdot \boldsymbol{v}_h,q_h)=0\quad\forall q_h\in Q_h\}\not= \emptyset
  \label{def:Vhdiv}
 \end{align}

\begin{definition} \label{def:streamline-direction}
For each macro element $M \in {\mathcal M}_h$ define the element-wise averaged streamline direction 
${\boldsymbol{u}}_M\in {\mathbb R}^d$ by
\begin{equation} \label{average-u} 
   {\boldsymbol{u}}_M := \frac{1}{|M|} \int_M {\boldsymbol{u}}(x)~dx. 
\end{equation}
This choice gives the estimates
\begin{align}
  |{\boldsymbol{u}}_M|&\leq C \|{\boldsymbol{u}}\|_{L^\infty(M)}, &
  \|{\boldsymbol{u}}-{\boldsymbol{u}}_M\|_{L^\infty(M)}&\leq Ch|{\boldsymbol{u}}|_{W^{1,\infty}(M)}.
\end{align}
\end{definition}
~\\
Now, we can formulate the semi-discrete stabilized approximation:
\begin{center}
 \begin{myminipage}
 Find $\boldsymbol{\mathcal{U}}_h=({\boldsymbol{u}}_h,p_h) \colon (t_0,T) \to \boldsymbol{V}\!_h\times Q_h$, such that for all 
${\boldsymbol{\mathcal{V}}}_h = ({\boldsymbol{v}}_h,q_h)\in \boldsymbol{V}\!_h\times Q_h$:
\begin{equation} \label{LPS}
 (\partial_t {\boldsymbol{u}}_h,{\boldsymbol{v}}_h) + A_G({\boldsymbol{u}}_h;\boldsymbol{\mathcal{U}}_h,{\boldsymbol{\mathcal{V}}}_h) + s_h({\boldsymbol{u}}_h;{\boldsymbol{u}}_h,{\boldsymbol{v}}_h) 
  + t({\boldsymbol{u}}_h;{\boldsymbol{u}}_h,{\boldsymbol{v}}_h)  =   (\boldsymbol{f},{\boldsymbol{v}}_h) 
\end{equation}
with the streamline-upwind-type stabilization $s_h$ and the grad-div 
stabilization $t$ according to
\begin{align}
\label{SUPG}
 s_h({\boldsymbol{w}}_h,{\boldsymbol{u}},{\boldsymbol{y}}_h,{\boldsymbol{v}}) &:= \sum_{M\in\mathcal M_h} \tau_M (\kappa_M(({\boldsymbol{w}}_M\cdot\nabla) {\boldsymbol{u}}), 
                \kappa_M(({\boldsymbol{y}}_M\cdot \nabla) {\boldsymbol{v}}))_M \\
\label{grad-div}
 t({\boldsymbol{u}},{\boldsymbol{v}})  &:= \gamma (\nabla \cdot {\boldsymbol{u}}, \nabla \cdot
{\boldsymbol{v}}).
\end{align}  
 \end{myminipage}
\end{center}
The set of non-negative stabilization parameters $\tau_M$, $\gamma$ 
has to be determined later on. For the grad-div stabilization at least 
$\gamma>\nu$ is assumed.

Occasionally, we consider the error in a norm given by symmetrically testing the stabilized approximation without time derivatives:
\begin{align*}
 \|\boldsymbol{u}\|_{LPS}^2:=\nu\|\nabla\boldsymbol{u}\|_0^2+\gamma\|\nabla\cdot\boldsymbol{u}\|_0^2+\sum_{M\in\mathcal M_h}\tau_M \|\kappa((\boldsymbol{u}_M\cdot\nabla)\boldsymbol{u})\|_0^2
\end{align*}

\section{Pressure-Correction Projection Discretization}\label{sec:pressure-correction}
For the discretization in time on the interval $[t_0,T]$ we consider $N$ equidistant time steps of size $\Delta t=(T-t_0)/N$ yielding the set $M_T=\{t_0, \dots, t_N=T\}$. The scheme that we are using is a pressure-correction projection approach based on BDF2.\\
In order to abbreviate the discrete time derivative we define the operator $D_t$ by
\begin{align}
 D_t \boldsymbol{u}^n:=\frac{3 \boldsymbol{u}^n-4\boldsymbol{u}^{n-1}+\boldsymbol{u}^{n-2}}{2 
\Delta t}.
\end{align}
Defining $\boldsymbol{Y}\!_h:=\boldsymbol{V}\!_h^{\,div}\oplus\nabla Q_h\subset [L^2(\Omega)]^d$
the fully discretized and stabilized scheme reads:
\begin{center}
 \begin{myminipage}
Find $\widetilde{\boldsymbol{u}}_h^n\in\boldsymbol{V}\!_h,\boldsymbol{u}_h^n\in\boldsymbol{Y}\!_h$ and $p_h^n \in Q_h$
\begin{align}
\begin{aligned}
\left(\frac{3 \widetilde{\boldsymbol{u}}_h^n-4\boldsymbol{u}^{n-1}_h+\boldsymbol{u}^{n-2}_h}{2 \Delta 
t}, \boldsymbol{v}_h\right) 
+\nu(\nabla \widetilde{\boldsymbol{u}}_h^n, \nabla \boldsymbol{v}_h) &+ c(\widetilde{\boldsymbol{u}}_h^n, \widetilde{\boldsymbol{u}}_h^n, \boldsymbol{v}_h)\\
+ \gamma (\nabla\cdot \widetilde{\boldsymbol{u}}_h^n, \nabla\cdot \boldsymbol{v}_h) 
+ s_h(\widetilde{\boldsymbol{u}}_h^n, \widetilde{\boldsymbol{u}}_h^n, \widetilde{\boldsymbol{u}}_h^n, \boldsymbol{v}_h)
=& (\boldsymbol{f}^n, \boldsymbol{v}_h) +( p^{n-1}_h,\nabla\cdot \boldsymbol{v}_h)\\
\widetilde{\boldsymbol{u}}_h^n|_{\partial\Omega}=&0
\label{eqn:full1}
\end{aligned}
\end{align}
\begin{align}
\begin{aligned}
\left(\frac{3 \boldsymbol{u}_h^n-3\widetilde{\boldsymbol{u}}_h^n}{2 \Delta t}
+\nabla (p_h^{n}-p_h^{n-1}), \boldsymbol{y}_h \right)&=0\\
(\nabla\cdot \boldsymbol{u}_h^n, q_h)&=0\\
\boldsymbol{u}_h^n|_{\partial\Omega} &=0
\label{eqn:full2}
\end{aligned}
\end{align}
holds for all $\boldsymbol{v}_h\in \boldsymbol{V}\!_h,~\boldsymbol{y}_h\in \boldsymbol{Y}\!_h $  and $q_h\in Q_h$. 
\end{myminipage}
\end{center}
~
From here on, we assume $Q_h\subset H^1(\Omega)$.
We call (\ref{eqn:full1}) the convection-diffusion and (\ref{eqn:full2}) the projection step.

\begin{remark}
By testing equation (\ref{eqn:full2}) with $\boldsymbol{w}_h\in\boldsymbol{V}\!_h^{\,div}$ we derive
\begin{align}
(\boldsymbol{u}_h^n-\widetilde{\boldsymbol{u}}_h^n, \boldsymbol{w}_h)=0\quad \forall \boldsymbol{w}_h\in\boldsymbol{V}\!_h^{\,div}.
\end{align}
Hence, $\boldsymbol{u}_h^n$ is the $L^2(\Omega)$ projection of $\widetilde{\boldsymbol{u}}_h^n$ onto $\boldsymbol{V}\!_h^{\,div}$ and $\|\boldsymbol{u}_h^n\|\leq\|\widetilde{\boldsymbol{u}}_h^n\|$.
\label{rem:l2proj}
\end{remark}

Choosing a slightly bigger ansatz space $\boldsymbol{Y}\!_h^1:=\boldsymbol{V}\!_h+\nabla Q_h$ instead of $\boldsymbol{Y}\!_h$ we can eliminate the weakly solenoidal field $\widetilde{\boldsymbol{u}}_h^n$
and replace (\ref{eqn:full1}) by the equation  
\begin{align}
 \begin{aligned}
  &(D_t \widetilde{\boldsymbol{u}}_h^n, \boldsymbol{v}_h) +\nu(\nabla \widetilde{\boldsymbol{u}}_h^n, \nabla 
\boldsymbol{v}_h) +
c(\widetilde{\boldsymbol{u}}_h^n; \widetilde{\boldsymbol{u}}_h^n, \boldsymbol{v}_h)
 \\ &\quad
+ s_h(\widetilde{\boldsymbol{u}}_h^n, \widetilde{\boldsymbol{u}}_h^n, \widetilde{\boldsymbol{u}}_h^n, \boldsymbol{v}_h) + 
\gamma (\nabla\cdot \widetilde{\boldsymbol{u}}_h^n, \nabla\cdot \boldsymbol{v}_h) \\
&\quad= (\boldsymbol{f}^n, \boldsymbol{v}_h) +\left(\frac73 p_h^{n-1}-\frac53
p_h^{n-2}+\frac13 p_h^{n-3}, \nabla\cdot\boldsymbol{v}_h\right)\\
&\widetilde{\boldsymbol{u}}_h^n|_{\partial\Omega}=0
\label{eqn:full_implemented1}
 \end{aligned}
 \end{align}
 and equation (\ref{eqn:full2}) by
 \begin{align}
  \begin{aligned}
  (\nabla (p_h^n-p_h^{n-1}), \nabla q_h)&=\left(\frac{3 \nabla\cdot \widetilde{\boldsymbol{u}}_h^n}{2 \Delta t}, q_h\right)\\
  (\boldsymbol{n} \cdot \nabla p_h^n)|_{\partial\Omega} &=0.
 \label{eqn:full_implemented2}
 \end{aligned}
 \end{align}
  $\boldsymbol{u}_h^n$ can then be recovered according to
 \begin{align*}
 \boldsymbol{u}_h^n=\widetilde{\boldsymbol{u}}_h^n-\nabla (p_h^{n}-p_h^{n-1}).
 \end{align*}
 This is the approach that is used in the implementation. The equivalence of the two formulations (\ref{eqn:full1}), (\ref{eqn:full2})
 and (\ref{eqn:full_implemented1}), (\ref{eqn:full_implemented2}) has been considered by Guermond in \cite{guermond1996some} for 
 a first order unstabilized projection scheme.
 
 \begin{remark}
For the first time step we use a BDF1 instead of the BDF2 scheme. In particular, the convection-diffusion and the projection step in the fully discretized setting read:
\begin{center}
\begin{myminipage}
 Find $\widetilde{\boldsymbol{u}}_h^n\in\boldsymbol{V}\!_h,\boldsymbol{u}_h^n\in\boldsymbol{V}\!_h^{\,div}$ and $p_h^n \in Q_h$ such that
\begin{align}
\begin{aligned} 
\left(\frac{\widetilde{\boldsymbol{u}}^1_h-\boldsymbol{u}^0_h}{\Delta 
t}, \boldsymbol{v}_h\right) 
+\nu(\nabla \widetilde{\boldsymbol{u}}^1_h,\nabla \boldsymbol{v}_h) 
+ c(\widetilde{\boldsymbol{u}}^1_h; \widetilde{\boldsymbol{u}}^1_h,\boldsymbol{v}_h)&\\ 
+ s_h(\widetilde{\boldsymbol{u}}^1_h,\widetilde{\boldsymbol{u}}^1_h,\widetilde{\boldsymbol{u}}^1_h,\boldsymbol{v}_h) + 
\gamma (\nabla\cdot \widetilde{\boldsymbol{u}}^1_h,\nabla\cdot \boldsymbol{v}_h) &= (\boldsymbol{f}^1, \boldsymbol{v}_h) +( p^0_h,\nabla\cdot \boldsymbol{v}_h), \label{eqn:initial_diffusion} \\
\widetilde{\boldsymbol{u}}_h^1|_{\partial\Omega}&=0
\end{aligned}
\end{align}
\begin{align}
\begin{aligned}
  \left(\frac{\boldsymbol{u}^1_h-\widetilde{\boldsymbol{u}}^1_h}{\Delta t}+\nabla (p_h^1-p_h^0), \boldsymbol{y}_h\right)&=0, \\
  (\nabla \cdot \boldsymbol{u}_h^1, q_h)&=0 \label{eqn:initial_projection}\\ 
\boldsymbol{u}_h^1|_{\partial\Omega}&=0 
\end{aligned}
\end{align}
holds for all $\boldsymbol{v}_h\in \boldsymbol{V}\!_h,\boldsymbol{y}_h\in \boldsymbol{Y}\!_h$ and $q_h\in Q_h$.
\end{myminipage}
\end{center}
The initial values are chosen according to $\widetilde{\boldsymbol{u}}^0_h=\boldsymbol{u}^0_h=j_u \boldsymbol{u}(t_0)$ and $p_h^0=j_p p(t_0)$ using the interpolation operators defined later on.
\end{remark}

 \begin{definition} \label{def:norm_fullydiscrete}
Consider sequences $\boldsymbol{u}=(\boldsymbol{u}^1, \dots, \boldsymbol{u}^N)\in\boldsymbol{A}^N$ of vector-valued and $p=(p^1, \dots, p^N)\in B^N$ of scalar-valued quantities, where $\boldsymbol{A}$ and $B$ are normed spaces and $1\le n\le N$. The norms we want to bound the errors in are defined by
\begin{align*}
\|\boldsymbol{u}\|_{l^2(t_0,T;\boldsymbol{A})}^2 &:=\Delta t\sum_{n=1}^N\|\boldsymbol{u}^n\|_{\boldsymbol{A}}^2, \quad
\|p\|_{l^2(t_0,T;B)}^2:=\Delta t\sum_{n=1}^N\|p^n\|_B^2, \\
\|\boldsymbol{u}\|_{l^\infty(t_0,T;\boldsymbol{A})} &:=\max_{1\le n\le N} \|\boldsymbol{u}^n\|_{\boldsymbol{A}}, \quad
\|p\|_{l^\infty(t_0,T;B)}:=\max_{1\le n\le N} \|p^n\|_B.
\end{align*}
 For quantities $r$ that are continuous in time we identify $r$ by its evaluation at the discrete points in time $(r(t_1), \dots, r(t_N))^T$.
 \end{definition} 
 \section{Quasi-Robustness and Quasi-Optimal Spatial Error Estimates }\label{sec:full1}
 In the following, we follow the strategy described by Guermond in 
\cite{guermond1999resultat}. Compared to his work we consider all 
$\nu$-dependencies as well as a grad-div and LPS SU.
 
The idea is to first state quasi-optimal results for the initial time step with respect to spatial and temporal discretization. Afterwards we derive estimates for $n\geq2$ that are optimal with respect to the LPS norm but suboptimal for the energy error and hence for the pressure error as well.
 \subsection{The Interpolation Operator}
 For the interpolation into the discrete ansatz spaces, we consider $(\boldsymbol{w}_h,r_h)$ as solution of the Stokes problem
\begin{center}
 \begin{myminipage}
  Find $\boldsymbol{w}_h^n\in\boldsymbol{V}\!_h$ and $r_h^n \in Q_h$ such that
\begin{align}
\begin{aligned}
&\nu(\nabla \boldsymbol{w}_h^n, \nabla \boldsymbol{v}_h) + \gamma (\nabla\cdot \boldsymbol{w}_h^n, \nabla\cdot \boldsymbol{v}_h) 
-( r_h^n, \nabla\cdot \boldsymbol{v}_h)
 \\&\quad
= \nu(\nabla \boldsymbol{u}(t_n), \nabla \boldsymbol{v}_h) 
+ \gamma (\nabla\cdot \boldsymbol{u}(t_n), \nabla\cdot \boldsymbol{v}_h)
-( p(t_n), \nabla\cdot \boldsymbol{v}_h)\\
&(\nabla\cdot\boldsymbol{w}_h^n, q_h)=(\nabla\cdot \boldsymbol{u}(t_n), q_h)=0
\end{aligned}
\end{align}
holds for all $\boldsymbol{v}_h\in \boldsymbol{V}\!_h$ and $q_h\in Q_h$.
\end{myminipage}
\end{center}
We define errors according to
\begin{align}
\label{eqn:error_u}
 \boldsymbol{\eta}_u^n&=\boldsymbol{u}(t_n)-\boldsymbol{w}_h^n, 
 & \boldsymbol{e}_u^n&=\boldsymbol{w}_h^n-\boldsymbol{u}_h^n, 
 & \boldsymbol{\zeta}_u^n&:=\boldsymbol{u}(t_n)-\boldsymbol{u}_h^n=\boldsymbol{\eta}_u^n+\boldsymbol{e}_u^n, \\ 
 \label{eqn:error_tildeu}
 && \widetilde{\boldsymbol{e}}_u^n&=\boldsymbol{w}_h^n-\widetilde{\boldsymbol{u}}_h^n, 
 & \widetilde{\boldsymbol{\zeta}}_u^n&:=\boldsymbol{u}(t_n)-\widetilde{\boldsymbol{u}}_h^n=\boldsymbol{\eta}_u^n+\widetilde{\boldsymbol{e}}_u^n, \\ 
 \label{eqn:error_p}
 \eta_p^n&=p(t_n)-r_h^n, &
 e_p^n&=r_h^n-p_h^n, &
 \zeta_p^n&:=p(t_n)-p_h^n=\eta_p^n+e_p^n.
\end{align}

According to \cite[Theorem 1]{Jenkins2013b} the solution of the grad-div stabilized Stokes problem can be bounded by
\begin{align*}
 \nu\|\nabla \boldsymbol{\eta}_u^n\|_0^2&\leq C\Big( \inf_{\boldsymbol{w}_h\in\boldsymbol{V}\!_h^{\,div}}\Big(
 \nu\|\nabla (\boldsymbol{u}(t_n)-\boldsymbol{w}_h)\|_0^2+\gamma \|\nabla\cdot (\boldsymbol{u}(t_n)-\boldsymbol{w}_h)\|_0^2 \Big)
 \\&\qquad\quad\quad+\frac{1}{\gamma}\inf_{q_h\in Q_h}\|p(t_n)-q_h\|_0^2\Big)\\
 &\leq C\left( \inf_{\boldsymbol{w}_h\in\boldsymbol{V}\!_h^{\,div}}
 (\nu+\gamma)\|\nabla( \boldsymbol{u}(t_n)-\boldsymbol{w}_h)\|_0^2+\frac{1}{\gamma}\inf_{q_h\in Q_h}\|p(t_n)-q_h\|_0^2\right).
\end{align*}

This result can easily be extended to include the grad-div stabilization on the left-hand side and using inf-sup stability we arrive at
\begin{align*}
 &\nu\|\nabla \boldsymbol{\eta}_u^n\|_0^2+\gamma\|\nabla\cdot \boldsymbol{\eta}_u^n\|_0^2+\|e_p^n\|_0^2
 \\&\leq C\left( \inf_{\boldsymbol{w}_h\in\boldsymbol{V}\!_h^{\,div}}
 (\nu+\gamma)\|\nabla( \boldsymbol{u}(t_n)-\boldsymbol{w}_h)\|_0^2+\frac{1}{\gamma}\inf_{q_h\in Q_h}\|p(t_n)-q_h\|_0^2\right).
\end{align*}

Using interpolation results according to Assumption~\ref{assumption-A.3} and the local inverse inequality (Assumption~\ref{assumption-A.2}) we get the bound
\begin{align}
\begin{aligned}
\label{eqn:stokes_iterp}
&\nu\|\boldsymbol{\eta}_u^n\|_1^2+\gamma\| \nabla\cdot\boldsymbol{\eta}_u^n\|_0^2+\|\eta_p^n\|_0^2+h^2 \|\nabla\eta_p^n\|_0^2 
 \\ &\leq C
  (\nu+  \gamma)h^{2k_u}\|\boldsymbol{u}(t_n)\|_{W^{k_u+1,2}}^2+\gamma^{-1}h^{2k_p+2}\|p(t_n)\|_{W^{k_p+1,2}}^2
 \end{aligned}
\end{align}
and using the Aubin-Nitsche trick 
\begin{align}
\begin{aligned}
\label{eqn:stokes_iterp_2}
\|\boldsymbol{\eta}_u^n\|_0^2&\leq C h^2(\nu\|\boldsymbol{\eta}_u^n\|_1^2+\gamma\| \nabla\cdot\boldsymbol{\eta}_u^n\|_0^2)
 \\ &
\leq 
 C ((\nu+  \gamma)h^{2k_u+2}\|\boldsymbol{u}(t_n)\|_{W^{k_u+1,2}}^2+\gamma^{-1}h^{2k_p+4}\|p(t_n)\|_{W^{k_p+1,2}}^2).
 \end{aligned}
\end{align}
This gives stability according to 
\begin{align}
\begin{aligned}
&\max_{1\leq n\leq N}(
 \|\boldsymbol{w}_h^n\|_0^2
 +\nu\|\boldsymbol{w}_h^n\|_1^2+\gamma\| \nabla\cdot\boldsymbol{w}_h^n\|_0^2+\|p_h^n\|_0^2)
 \\ &
  \leq \max_{1\leq n\leq N}(
 \|\boldsymbol{\eta}_u^n\|_0^2
 +\nu\|\nabla \boldsymbol{\eta}_u^n\|_1^2+\gamma\|\nabla\cdot\boldsymbol{\eta}_u^n\|_0^2 +\|\eta_p^n\|_0^2 
  \\ &\quad\qquad\qquad
 + \|\boldsymbol{u}(t_n)\|_0^2
 +\nu\|\nabla \boldsymbol{u}(t_n)\|_1^2+\gamma\|\nabla\cdot\boldsymbol{u}(t_n)\|_0^2 +\|p(t_n)\|_0^2) \\
 &\leq C\max_{1\leq n\leq N}((\|\boldsymbol{u}(t_n)\|_0^2+\nu\|\nabla \boldsymbol{u}(t_n)\|_1^2+\gamma\|\nabla\cdot\boldsymbol{u}(t_n)\|_0^2 +\|p(t_n)\|_0^2))
 \leq C.
 \end{aligned}
\end{align}

 \subsection{Initial Error Estimates}
Before deriving bounds for the case $n\geq2$ we consider the initialization step. Although both estimates follow the same approach, we nevertheless state both for the convenience of the reader.

With the abbreviations
\begin{align*}
 \delta_t \boldsymbol{u}^n &:=\boldsymbol{u}^n-\boldsymbol{u}^{n-1}&
 \widetilde{D}_t \boldsymbol{u}^n &:=\frac{\delta_t \boldsymbol{u}^n}{\Delta t}=\frac{\boldsymbol{u}^n-\boldsymbol{u}^{n-1}}{\Delta t}
\end{align*}
we can derive the following result:

\begin{lemma}\label{lem:initial}
 Provided the continuous solutions satisfy the regularity assumptions
 \begin{align*}
 \boldsymbol{u}&\in W^{2,\infty}(t_0,T;L^2)\cap L^\infty(t_0,T;W^{k_u+1,2})\\
 p &\in W^{1,\infty}(t_0,T;H^1)\cap L^\infty(t_0,T;W^{k_p+1,2}),
 \end{align*}
 we obtain for the initial errors $\boldsymbol{e}_u^1=\boldsymbol{w}_h^1-\boldsymbol{u}_h^1$, $\widetilde{\boldsymbol{e}}_u^1=\boldsymbol{w}_h^1-\widetilde{\boldsymbol{u}}_h^1$ and $e_p^1=r_h^1-p_h^1$
 \begin{align} 
 \begin{aligned}
  &  \|\boldsymbol{e}_u^1\|_0^2+\|\widetilde{\boldsymbol{e}}_u^1\|_0^2+\Delta t \nu \|\nabla \widetilde{\boldsymbol{e}}_u^1\|_0^2 
 + \Delta t \gamma \|\nabla \cdot  \widetilde{\boldsymbol{e}}_u^1\|_0^2 +(\Delta t)^2\|\nabla e_p^1\|_0^2
 \\&\quad+\Delta t \sum_{M\in\mathcal 
M_h}\tau_M^1\|\kappa_M((\widetilde{\boldsymbol{u}}_M^1\cdot\nabla)\widetilde{\boldsymbol{e}}_u^1)\|_{0,M}^2
\\\label{eqn:err1}&\lesssim ((\Delta t)^2) \|\widetilde{D}_t\boldsymbol{\eta}_u^1\|_0^2+((\Delta t)^2)\|\widetilde{D}_t \boldsymbol{u}(t_1)-\partial_t \boldsymbol{u}(t_1)\|_0^2+ (\Delta t)^2\|\delta_t \nabla r_h^1\|_0^2
\\&\quad
+ h^{-2z}\Delta t \| \boldsymbol{\eta}_u^1\|_0^2  
+ 3  \Delta t h^{2z-2}||| \boldsymbol{\eta}_u^1|||^2_{LPS}  
+ 3  \Delta t h^{2z-2}||| \widetilde{\boldsymbol{e}}_u^1|||^2_{LPS} 
 \\&\quad
+ \Delta t\max_{M\in\mathcal M_h}\{\tau_M^1 |\widetilde{\boldsymbol{u}}_M^1|^2\} \left(\|\boldsymbol{\eta}_u^1\|_1^2+\sum_{M\in\mathcal M_h}\|\kappa_M(\nabla\boldsymbol{u}(t_1))\|_0^2\right)
\end{aligned}
\end{align}
where $z\in\{0,1\}$ if
\begin{align*}
 K_1&:= \frac14+\Delta t\Bigg[ |\boldsymbol{u}(t_1)|_{W^{1,\infty}} + C h^{2z} |\boldsymbol{u}(t_1)|_{W^{1,\infty}}^2 
+  C\frac{h^2}{\gamma} | \boldsymbol{u}(t_1)|^2_{W^{1,\infty}}
\\&\qquad\qquad\qquad
+ C \frac{h^{2z-2}}{\gamma}\|\boldsymbol{u}(t_1)\|_{\infty}^2 
+ C h^{2z-2}\|\widetilde{\boldsymbol{u}}_h^1\|_{\infty}^2\Bigg]<1.
\end{align*}
 \end{lemma}
\begin{proof}
The proof is similar to the one for $n\geq2$. One difference is the different estimate for the pressure terms which simplifies to
\begin{align*}
 (r_h^1-p_h^0, \nabla\cdot \widetilde{\boldsymbol{e}}_u^1)=-(\nabla(r_h^1-r_h^0), \widetilde{\boldsymbol{e}}_u^1)
 \leq \frac{1}{8 \Delta t} \|\widetilde{\boldsymbol{e}}_u^1\|_0^2+C \Delta t\|\delta_t \nabla r_h^1\|_0^2.
 \end{align*}
 and due to the linearity of the Stokes problem it holds the estimate
\begin{align}
\begin{aligned}
\|\nabla \delta_t r_h^1\|_0^2&=\|\nabla \delta_t(p(t_1)-\eta_p^1)\|_0^2\leq \|\nabla \delta_t\eta_p^1 \|_0^2+\|\nabla \delta_t p(t_1)\|_0^2
\\&\leq C (\Delta t)^2(\|\boldsymbol{u}\|_{W^{1,\infty}(t_0,T;W^{k_u+1,2})}^2+\|p\|_{W^{1,\infty}(t_0,T;W^{k_p+1,2})}^2).
\end{aligned}
 \label{eqn:diff_pressure_3_0}
\end{align}
The remaining parts of the proof with respect to the velocity errors are similar to the ones in Lemma~\ref{lem:errn2}.

For an estimate on the gradient of the pressure error the projection equation (\ref{eqn:initial_projection}) is utilized:
\begin{align*}
 \|\nabla  e_p^1\|_0^2&=(\nabla( e_p^1- e_p^0), \nabla e_p^1)=(\delta_t r_h^1, \nabla e_p^1)+\left(\frac{\widetilde{\boldsymbol{e}}_u^1-\boldsymbol{e}_u^1}{\Delta t}, \nabla e_p^1\right)\\
 &\leq C \|\delta_t r_h^1\|_0^2 + \frac{C}{(\Delta t)^2} \|\widetilde{\boldsymbol{e}}_u^1\|_0^2+\frac{1}{2}\|\nabla e_p^1\|_0^2
 \lesssim C \|\delta_t r_h^1\|_0^2 + \frac{C}{(\Delta t)^2} \|\widetilde{\boldsymbol{e}}_u^1\|_0^2.
\end{align*}
Here we used that $\boldsymbol{e}_u^1$ is a $L^2(\Omega)$-projection of $\widetilde{\boldsymbol{e}}_u^1$
and hence
\begin{align*}
 \|\widetilde{\boldsymbol{e}}_u^1-\boldsymbol{e}_u^1\|_0^2\leq C( \|\widetilde{\boldsymbol{e}}_u^1\|_0^2 + \|\boldsymbol{e}_u^1\|_0^2)\leq C \|\widetilde{\boldsymbol{e}}_u^1\|_0^2.
\end{align*}
Therefore, $(\Delta t)^2\|\nabla e_p^1\|_0^2$ can be bounded by the right-hand side in (\ref{eqn:err1}), too.
\end{proof}

\subsection{Error Estimates after Initialization}
Next, we are interested in the discretization errors $\widetilde{\boldsymbol{e}}_u^n, \boldsymbol{e}_u^n$ and $e_p^n$ for $n\geq 2$. The approach is the same as for the initialization but does not yield quasi-optimal results with respect to the energy norm of the velocity.

\begin{lemma}\label{lem:errn2}
 For all $1\leq m \leq N$ the discretization error $\boldsymbol{e}_u^m=\boldsymbol{w}_h^m-\boldsymbol{u}_h^m$, $\widetilde{\boldsymbol{e}}_u^m=\boldsymbol{w}_h^m-\widetilde{\boldsymbol{u}}_h^m$ and $e_p^m=r_h^m-p_h^m$ can be bounded by
 \begin{align}
\nonumber
&  \|\widetilde{\boldsymbol{e}}_u^m\|_0^2+\|2 \boldsymbol{e}_u^{m}- 
\boldsymbol{e}_u^{m-1}\|_0^2+\frac{4}{3}(\Delta t)^2 \|\nabla e_p^m\|_0^2 \\ \nonumber
&\quad+\sum_{n=2}^m\left(\Delta t \nu \|\nabla \widetilde{\boldsymbol{e}}_u^n\|_0^2 + \Delta t\gamma \|\nabla 
\cdot  \widetilde{\boldsymbol{e}}_u^n\|_0^2+\|\delta_{tt}\boldsymbol{e}_u^{n}\|_0^2
 \right.\\&\qquad\qquad\left.\nonumber
+ 2\Delta t \sum_{M\in\mathcal 
M_h}\tau_M^n\|\kappa_M((\widetilde{\boldsymbol{u}}_M^n\cdot\nabla)\widetilde{\boldsymbol{e}}_u^n)\|_{0,M}^2\right)\\ \label{eqn:fullerror1}
&\lesssim C_G\left((\Delta t)^2(\|\boldsymbol{u}\|_{W^{1,\infty}(t_0,T;W^{k_u+1,2})}^2+\|p\|_{W^{1,\infty}(t_0,T;W^{k_p+1,2})}^2)
\right.\\&\qquad\qquad\left.\nonumber
+(\nu+\gamma)h^{2k_u+2-2z}\|\boldsymbol{u}\|_{L^\infty(t_0,T;W^{k_u+1,2})}^2\right.\\\nonumber&\qquad\qquad+\gamma^{-1}h^{2k_p+4-2z}\|p\|_{L^\infty(t_0,T;W^{k_p+1,2})}^2
+(\Delta t)^5 \|\boldsymbol{u}\|_{W^{3,\infty}(t_0,T;L^2)}\\
&\qquad\qquad\nonumber
 +\left(\frac{\nu+\gamma}{\nu}h^{2k_u}\|\boldsymbol{u}\|_{L^\infty(t_0,T;W^{k_u+1,2})}^2+\frac{h^{2k_p+2}}{\gamma\nu}\|p\|_{L^\infty(t_0,T;W^{k_p+1,2})}^2
 \right.\\&\qquad\qquad\quad\left.\left.\nonumber
+h^{2s}\|\boldsymbol{u}\|_{L^\infty(t_0,T;W^{s+1,2})}^2\right)\max_{1\leq n\leq m}\max_{M\in\mathcal M_h}\{\tau_M^n |\widetilde{\boldsymbol{u}}_M^n|^2\}\right).
\nonumber
 \end{align}
 where $z\in\{0,1\}$ and $C_G\sim\exp\left(\frac{T}{1-K_n}\right)$ provided
 \begin{align*}
 K_n&:= \left(\frac14+C\Delta t\Bigg[ |\boldsymbol{u}(t_n)|_{W^{1,\infty}} +  
h^{2z} |\boldsymbol{u}(t_n)|_{W^{1,\infty}}^2 
+  \frac{h^2}{\gamma} | \boldsymbol{u}(t_n)|^2_{W^{1,\infty}}
 \right.\\&\qquad\qquad\qquad\left. 
+ \frac{h^{2z-2}}{\gamma}\|\boldsymbol{u}(t_n)\|_{\infty}^2 
+ h^{2z-2}\|\widetilde{\boldsymbol{u}}_h^n\|_{\infty}^2 \Bigg]\right)<1
  \end{align*}
  and the continuous solution fulfills the regularity assumptions
  $\boldsymbol{u}\in W^{2,\infty}(t_0,T;L^2)\cap L^\infty(t_0,T;W^{k_u+1,2})$ and 
 $p \in W^{1,\infty}(t_0,T;H^1)\cap L^\infty(t_0,T;W^{k_p+1,2})$. 
\end{lemma}
\begin{proof}
Plugging $\boldsymbol{w}_h^n$ into the fully discretized equation yields
\begin{align*}
 &(D_t\boldsymbol{w}_h^n , \boldsymbol{v}_h) 
+\nu(\nabla \boldsymbol{w}_h^n, \nabla \boldsymbol{v}_h) + c(\boldsymbol{w}_h^n, \boldsymbol{w}_h^n, \boldsymbol{v}_h)
 \\&\quad
+ \gamma (\nabla\cdot \boldsymbol{w}_h^n, \nabla\cdot \boldsymbol{v}_h) 
+ s_h(\boldsymbol{w}_h^n, \boldsymbol{w}_h^n, \boldsymbol{w}_h^n, \boldsymbol{v}_h)-(r_h^n, \nabla\cdot \boldsymbol{v}_h)
\\&
= (D_t\boldsymbol{w}_h^n , \boldsymbol{v}_h) 
+\nu(\nabla \boldsymbol{u}(t_n), \nabla \boldsymbol{v}_h) + c(\boldsymbol{w}_h^n, \boldsymbol{w}_h^n, \boldsymbol{v}_h)
 \\&\quad
+ \gamma (\nabla\cdot \boldsymbol{u}(t_n), \nabla\cdot \boldsymbol{v}_h) 
+ s_h(\boldsymbol{w}_h^n, \boldsymbol{w}_h^n, \boldsymbol{w}_h^n, \boldsymbol{v}_h)
.
\end{align*}
Hence, the error equation reads
\begin{align}
\begin{aligned}
 &\left(\frac{3 \widetilde{\boldsymbol{e}}_u^n-4\boldsymbol{e}_u^{n-1}+\boldsymbol{e}_u^{n-2}}{2 \Delta 
t}, \boldsymbol{v}_h\right) 
+\nu(\nabla \widetilde{\boldsymbol{e}}_u^n, \nabla \boldsymbol{v}_h) 
 \\&\quad 
+ \gamma (\nabla\cdot \widetilde{\boldsymbol{e}}_u^n, \nabla\cdot \boldsymbol{v}_h) 
+ s_h(\widetilde{\boldsymbol{u}}_h^n, \widetilde{\boldsymbol{e}}_u^n, \widetilde{\boldsymbol{u}}_h^n, \boldsymbol{v}_h)\\
&= -(\boldsymbol{f}^n, \boldsymbol{v}_h)
+(D_t\boldsymbol{w}_h^n, \boldsymbol{v}_h)
+\nu(\nabla \boldsymbol{u}(t_n), \nabla \boldsymbol{v}_h)
+ c(\widetilde{\boldsymbol{u}}_h^n, \widetilde{\boldsymbol{u}}_h^n, \boldsymbol{v}_h)
\\&\quad
+ s_h(\widetilde{\boldsymbol{u}}_h^n, \widetilde{\boldsymbol{u}}_h^n, \widetilde{\boldsymbol{u}}_h^n, \boldsymbol{v}_h) 
+ s_h(\widetilde{\boldsymbol{u}}_h^n, \widetilde{\boldsymbol{e}}_u^n, \widetilde{\boldsymbol{u}}_h^n, \boldsymbol{v}_h)+(r_h^n-p_h^{n-1}-p(t_n), \nabla\cdot \boldsymbol{v}_h)\\
&= -(\boldsymbol{f}^n, \boldsymbol{v}_h)
+(\partial_t \boldsymbol{u}(t_n), \boldsymbol{v}_h)
+\nu(\nabla \boldsymbol{u}(t_n), \nabla \boldsymbol{v}_h)
+ c(\boldsymbol{u}(t_n), \boldsymbol{u}(t_n), \boldsymbol{v}_h)
\\&\quad
-(p(t_n), \nabla\cdot \boldsymbol{v}_h)
-(D_t \boldsymbol{\eta}_u^n, \boldsymbol{v}_h)
+(D_t \boldsymbol{u}(t_n), \boldsymbol{v}_h)
-(\partial_t \boldsymbol{u}(t_n), \boldsymbol{v}_h)
  \\&\quad
+c(\widetilde{\boldsymbol{u}}_h^n, \widetilde{\boldsymbol{u}}_h^n, \boldsymbol{v}_h)- c(\boldsymbol{u}(t_n), \boldsymbol{u}(t_n), \boldsymbol{v}_h)\\
&\quad + s_h(\widetilde{\boldsymbol{u}}_h^n, \widetilde{\boldsymbol{e}}_u^n, \widetilde{\boldsymbol{u}}_h^n, \boldsymbol{v}_h)+ s_h(\widetilde{\boldsymbol{u}}_h^n, \widetilde{\boldsymbol{u}}_h^n, \widetilde{\boldsymbol{u}}_h^n, \boldsymbol{v}_h)
+(r_h^n-p_h^{n-1}, \nabla\cdot \boldsymbol{v}_h)\\
&=\underbrace{-(D_t\boldsymbol{\eta}_u^n, \boldsymbol{v}_h)}_{I}
+\underbrace{(D_t\boldsymbol{u}(t_n), \boldsymbol{v}_h)-(\partial_t \boldsymbol{u}(t_n), \boldsymbol{v}_h)}_{II}
\\&\quad
+ \underbrace{c(\widetilde{\boldsymbol{u}}_h^n, \widetilde{\boldsymbol{u}}_h^n, \boldsymbol{v}_h)- c(\boldsymbol{u}(t_n), \boldsymbol{u}(t_n), \boldsymbol{v}_h)}_{III}
 \\&\quad
+\underbrace{s_h(\widetilde{\boldsymbol{u}}_h^n, \widetilde{\boldsymbol{e}}_u^n, \widetilde{\boldsymbol{u}}_h^n, \boldsymbol{v}_h)+ s_h(\widetilde{\boldsymbol{u}}_h^n, \widetilde{\boldsymbol{u}}_h^n, \widetilde{\boldsymbol{u}}_h^n, \boldsymbol{v}_h)}_{IV}
+\underbrace{(r_h^n-p_h^{n-1}, \nabla\cdot \boldsymbol{v}_h)}_{V}
\end{aligned}
\end{align}
due to the fact that the pair $(\boldsymbol{u}(t_n),p(t_n))$ fulfills the (continuous) Navier-Stokes equations.\\
The terms with respect to time discretization can be bounded using Young's inequality according to
\begin{align}
\begin{aligned}
 I&=-\left(D_t \boldsymbol{\eta}_u^n, \widetilde{\boldsymbol{e}}_u^n\right)\leq C\|\widetilde{\boldsymbol{e}}_u^n\|_0\|D_t \boldsymbol{\eta}_u^n\|_0
\leq \frac{1}{32\Delta t}\|\widetilde{\boldsymbol{e}}_u^n\|_0^2+C\Delta t\|D_t \boldsymbol{\eta}_u^n\|_0^2\\
II&=(D_t \boldsymbol{u}(t_n)-\partial_t \boldsymbol{u}(t_n), \widetilde{\boldsymbol{e}}_u^n)
\leq C \Delta t \|D_t \boldsymbol{u}(t_n)-\partial_t \boldsymbol{u}(t_n)\|_0^2+\frac{1}{32\Delta t}\|\widetilde{\boldsymbol{e}}_u^n\|_0^2.
\end{aligned}
\end{align}
Noticing that the error equation for the projection step reads
\begin{align*}
 -(\frac{3\nabla\cdot\widetilde{\boldsymbol{e}}_u^n}{2\Delta t}, q_h)=-(\frac{3\boldsymbol{e}_u^n-3\widetilde{\boldsymbol{e}}_u^n}{2\Delta t}, \nabla q_h)=(\nabla(p_h^{n-1}-p_h^n), \nabla q_h), 
\end{align*}
we may write $V$ by choosing $q_h=r_h^n-p_h^{n-1}$ as
\begin{align*}
 -V&=-(r_h^n-p_h^{n-1}, \nabla\cdot \widetilde{\boldsymbol{e}}_u^n)
 =\frac{2\Delta t}{3}(\nabla (p_h^{n-1}-p_h^n), \nabla (r_h^n-p_h^{n-1}))\\
 &=\frac{2\Delta t}{3}(\nabla ((p_h^{n-1}-r_h^n)+(r_h^n-p_h^n)), \nabla (r_h^n-p_h^{n-1}))\\
 &=\frac{2\Delta t}{3}((\nabla (r_h^n-p_h^n), \nabla ((r_h^n-p_h^n)+(p_h^n-p_h^{n-1})))-\|\nabla(p_h^{n-1}-r_h^n)\|_0^2)\\
 &=\frac{2\Delta t}{3}(\|\nabla e_p^n\|_0^2+(\nabla ((r_h^n-p_h^{n-1})
  \\ &\quad
 +(p_h^{n-1}-p_h^n)), \nabla (p_h^n-p_h^{n-1}))-\|\nabla(p_h^{n-1}-r_h^n)\|_0^2)\\
 &=\frac{2\Delta t}{3}(\|\nabla e_p^n\|_0^2-\|\nabla(p_h^{n-1}-p_h^n)\|_0^2
  \\ &\quad
 +\underbrace{(\nabla (r_h^n-p_h^{n-1}), \nabla (p_h^n-p_h^{n-1}))}_{V}-\|\nabla(p_h^{n-1}-r_h^n)\|_0^2)\\
 &=\frac{\Delta t}{3}(\|\nabla e_p^n\|_0^2-\|\nabla(p_h^{n-1}-p_h^n)\|_0^2-\|\nabla(p_h^{n-1}-r_h^n)\|_0^2)\\
 &=\frac{\Delta t}{3}(\|\nabla e_p^n\|_0^2-\frac{9}{4(\Delta t)^2}\|\widetilde{\boldsymbol{e}}_u^n-\boldsymbol{e}_u^n\|_0^2-\|\nabla (e_p^{n-1}+(r_h^{n-1}-r_h^n))\|_0^2).
 \end{align*}
Using $(a+b)^2\leq (1+\epsilon)a^2+(1+\frac{1}{\epsilon})b^2$ we get for the last term ($\epsilon=\Delta t$)
\begin{align*}
&\|\nabla(e_p^{n-1}+(r_h^{n-1}-r_h^n))\|_0^2
\leq 
(1+\Delta t)\|\nabla e_p^{n-1}\|_0^2+\left(1+\frac{1}{\Delta t}\right)\|\nabla(r_h^{n-1}-r_h^n)\|_0^2
\end{align*}
and finally
\begin{align*}
 V \geq& \frac{\Delta t}{3}\|\nabla e_p^n\|_0^2-\frac{3}{4 \Delta t}\|\widetilde{\boldsymbol{e}}_u^n-\boldsymbol{e}_u^n\|_0^2-\frac{\Delta t}{3}(1+\Delta t)\|\nabla e_p^{n-1}\|_0^2
  \\ &
 -\frac{\Delta t}{3}\left(1+\frac{1}{\Delta t}\right)\|\nabla(r_h^{n-1}-r_h^n)\|_0^2.
\end{align*}
For $III$ we use the bound from Lemma~\ref{lem:convectiveterms}, i.e.
\begin{align*}
 &c(\boldsymbol{u}(t_n);\boldsymbol{u}(t_n), \widetilde{\boldsymbol{e}}_u^n) - c(\widetilde{\boldsymbol{u}}_h^n;\widetilde{\boldsymbol{u}}_h^n, \widetilde{\boldsymbol{e}}_u^n) \\& 
 \le \Bigg[ |\boldsymbol{u}(t_n)|_{W^{1,\infty}} + C h^{2z} |\boldsymbol{u}(t_n)|_{W^{1,\infty}}^2 
+  C\frac{h^2}{\gamma} | \boldsymbol{u}(t_n)|^2_{W^{1,\infty}}
 \\&\qquad\quad
+ C \frac{h^{2z-2}}{\gamma}\|\boldsymbol{u}(t_n)\|_{\infty}^2 
+ C  h^{2z-2} \|\widetilde{\boldsymbol{u}}_h^n\|_{\infty}^2 \Bigg] 
\| \widetilde{\boldsymbol{e}}_u^n\|_0^2
\\&\quad
+Ch^{-2z} \| \boldsymbol{\eta}_u^n\|_0^2  
+ 3 C h^{2-2z}||| \boldsymbol{\eta}_u^n|||^2_{LPS}  +   3 C h^{2-2z}||| \widetilde{\boldsymbol{e}}_u^n|||^2_{LPS}
\end{align*}
where $z\in\{0,1\}$.

The treatment of $IV$ is again as in the initialization:
\begin{align*}
IV&=s_h(\widetilde{\boldsymbol{u}}_h^n, \boldsymbol{w}_h^n, \widetilde{\boldsymbol{u}}_h^n, \widetilde{
\boldsymbol{e}}_u^n)
=s_h(\widetilde{\boldsymbol{u}}_h^n, \boldsymbol{u}(t_n), \widetilde{\boldsymbol{u}}_h^n, \widetilde{
\boldsymbol{e}}_u^n)-s_h(\widetilde{\boldsymbol{u}}_h^n, \boldsymbol{\eta}_u^n, \widetilde{\boldsymbol{u}}_h^n, \widetilde{
\boldsymbol{e}}_u^n)\\
s_h&(\widetilde{\boldsymbol{u}}_h^n, \boldsymbol{u}(t_n), \widetilde{\boldsymbol{u}}_h^n, \widetilde{
\boldsymbol{e}}_u^n)\\&\leq 
s_h(\widetilde{\boldsymbol{u}}_h^n, \widetilde{\boldsymbol{e}}_u^n, \widetilde{\boldsymbol{u}}_h^n, \widetilde{
\boldsymbol{e}}_u^n)^{1/2}
s_h(\widetilde{\boldsymbol{u}}_h^n, \boldsymbol{u}(t_n), \widetilde{\boldsymbol{u}}_h^n, \boldsymbol{u}(t_n))^{1/2} \\
&\leq \frac18 s_h(\widetilde{\boldsymbol{u}}_h^n, \widetilde{\boldsymbol{e}}_u^n, \widetilde{\boldsymbol{u}}_h^n, \widetilde{
\boldsymbol{e}}_u^n)
+C \max_{M\in\mathcal M_h}\{\tau_M^n |\widetilde{\boldsymbol{u}}_M^n|^2\} \sum_{M\in\mathcal M_h}\|\kappa_M(\nabla\boldsymbol{u}(t_n)\|_0^2)\\
s_h&(\widetilde{\boldsymbol{u}}_h^n, \boldsymbol{\eta}_u^n, \widetilde{\boldsymbol{u}}_h^n, \widetilde{
\boldsymbol{e}}_u^n)\\
&\leq 
s_h(\widetilde{\boldsymbol{u}}_h^n, \widetilde{\boldsymbol{e}}_u^n, \widetilde{\boldsymbol{u}}_h^n, \widetilde{
\boldsymbol{e}}_u^n)^{1/2}
s_h(\widetilde{\boldsymbol{u}}_h^n, \boldsymbol{\eta}_u^n, \widetilde{\boldsymbol{u}}_h^n, \boldsymbol{\eta}_u^n)^{1/2} \\
&\leq C \max_{M\in\mathcal M_h}\{\tau_M^n |\widetilde{\boldsymbol{u}}_M^n|^2\} \|\boldsymbol{\eta}_u^n\|_1^2 +\frac18 s_h(\widetilde{\boldsymbol{u}}_h^n, \widetilde{\boldsymbol{e}}_u^n, \widetilde{\boldsymbol{u}}_h^n, \widetilde{
\boldsymbol{e}}_u^n)
\end{align*}

Using the splitting (\ref{eqn:splitting}) from the appendix, the discrete time derivative can be written as
\begin{align*}
  &2( 3\widetilde{\boldsymbol{e}}_u^n-4\boldsymbol{e}_u^{n-1}+\boldsymbol{e}_u^{n-2}, \widetilde{\boldsymbol{e}}_u^n) \\
&=3\|\widetilde{\boldsymbol{e}}_u^n\|_0^2-3\|\widetilde{\boldsymbol{e}}_u^n-\boldsymbol{e}_u^n\|_0^2-2\|\boldsymbol{e}_u^n\|_0^2
+2( 3\boldsymbol{e}_u^n-4\boldsymbol{e}_u^{n-1}+\boldsymbol{e}_u^{n-2}, \widetilde{\boldsymbol{e}}_u^n- \boldsymbol{e}_u^n)
\\&\quad
+\|2\boldsymbol{e}_u^n-\boldsymbol{e}_u^{n-1}\|_0^2+\|\partial_{tt} \boldsymbol{e}_u^n\|_0^2
-\|\boldsymbol{e}_u^{n-1}\|_0^2-\|2\boldsymbol{e}_u^{n-1}-\boldsymbol{e}_u^{n-2}\|_0^2.
 \end{align*}
 where the term $2( 3\boldsymbol{e}_u^n-4\boldsymbol{e}_u^{n-1}+\boldsymbol{e}_u^{n-2}, \widetilde{\boldsymbol{e}}_u^n)$ vanishes due to the projection equation and the fact that $\boldsymbol{e}_u^n$ is discretely divergence-free.

Now, we can summarize the above estimates to obtain
\begin{align*}
&  \|\widetilde{\boldsymbol{e}}_u^m\|_0^2+\|2 \boldsymbol{e}_u^{m}- 
\boldsymbol{e}_u^{m-1}\|_0^2+\frac{4}{3}(\Delta t)^2 \|\nabla e_p^m\|_0^2 \\
&\quad+\sum_{n=2}^m\left(2\Delta t \nu \|\nabla \widetilde{\boldsymbol{e}}_u^n\|_0^2 + 4\Delta t\gamma \|\nabla 
\cdot  \widetilde{\boldsymbol{e}}_u^n\|_0^2+\|\delta_{tt}\boldsymbol{e}_u^{n}\|_0^2
 \right.\\&\qquad\qquad\left.
+ 3\Delta t \sum_{M\in\mathcal 
M_h}\tau_M^n\|\kappa_M((\widetilde{\boldsymbol{u}}_M^n\cdot\nabla)\widetilde{\boldsymbol{e}}_u^n)\|_{0,M}^2\right)
\\&
\leq \| \widetilde{\boldsymbol{e}}_u^1\|_0^2 +\|2 \boldsymbol{e}_u^1 - 
\boldsymbol{e}_u^0\|_0^2  + \frac{4}{3}(\Delta t)^2\|\nabla e_p^1\|_0^2 
\\&\quad
 +\sum_{n=2}^m\left\{K_n \|\widetilde{\boldsymbol{e}}_u^n\|_0^2
+\frac{4(\Delta t)^3}{3}\|\nabla e_p^n\|_0^2
+\frac{4(\Delta t)^2}{3}\left(1+\frac{1}{2\Delta t}\right)\|\nabla(r_h^{n-1}-r_h^n)\|_0^2
\right.\\&\qquad\qquad\quad
+C((\Delta t)^2) \|D_t\boldsymbol{\eta}_u^n\|_0^2
+C((\Delta t)^2)\|D_t \boldsymbol{u}(t_n)-\partial_t \boldsymbol{u}(t_n)\|_0^2
\\&\qquad\qquad\quad
+ Ch^{-2z}\Delta t \| \boldsymbol{\eta}_u^n\|_0^2 
+ 3 C \Delta t h^{2z-2}||| \boldsymbol{\eta}_u^n|||^2_{LPS}  
+ 3 C \Delta t h^{2z-2}||| \widetilde{\boldsymbol{e}}_u^n|||^2_{LPS} 
\\&\qquad\qquad\quad\left.
+C \frac{\Delta t}{\nu} \max_{M\in\mathcal M_h}\{\tau_M^n |\widetilde{\boldsymbol{u}}_M^n|^2\} (\nu\|\boldsymbol{\eta}_u^n\|_1^2+\nu\sum_{M\in\mathcal M_h}\|\kappa_M(\nabla\boldsymbol{u}(t_n)\|_0^2))\right\}.
\end{align*}

Due to the linearity of the Stokes problem and using the estimate (\ref{eqn:stokes_iterp}), we can bound the gradient of the pressure discretization by
\begin{align}
\begin{aligned}
 \|\nabla\delta_t \eta_p^n\|_0^2&\leq C (\nu+\gamma)h^{2k_u-2}\|\delta_t \boldsymbol{u}(t_n)\|_{W^{k_u+1,2}}^2+\gamma^{-1}h^{2k_p}\|\delta_t p(t_n)\|_{W^{k_p+1,2}}^2\\
 &\leq C(\Delta t)^2 ((\nu + \gamma)h^{2k_u-2}\|\boldsymbol{u}\|_{W^{1,\infty}(t_0,T;W^{k_u+1,2})}^2
\\&\qquad\qquad\quad  +\frac{h^{2k_p}}{\gamma}\|p\|_{W^{1,\infty}(t_0,T;W^{k_p+1,2})}^2)
 \label{eqn:diff_pressure_3}\\
\|\nabla \delta_t r_h^n\|_0^2&\leq \|\nabla \delta_t\eta_p^n \|_0^2+\|\nabla \delta_t p(t_n)\|_0^2
\\&
\leq C (\Delta t)^2(\|\boldsymbol{u}\|_{W^{1,\infty}(t_0,T;W^{k_u+1,2})}^2+\|p\|_{W^{1,\infty}(t_0,T;W^{k_p+1,2})}^2)
\end{aligned}
\end{align}
and the approximation of the time derivative by
\begin{align}
\|D_t \boldsymbol{u}(t_n)-\partial_t \boldsymbol{u}(t_n)\|_0^2 \leq C (\Delta t)^4 \|\boldsymbol{u}\|_{W^{3,\infty}(t_0,T;L^2)}^2.
\end{align}

Using the discrete Gronwall Lemma~\ref{lem:Gronwall_discrete} for $K_n\|\widetilde{\boldsymbol{e}}_u^n\|_0^2+\frac{4(\Delta t)^3}{3}\|\nabla e_p^n\|_0^2 $, the approximation properties of the interpolation operators (\ref{eqn:stokes_iterp})-(\ref{eqn:stokes_iterp_2}) and Assumption~\ref{assumption-A.4} for the fluctuation operators, we arrive at
\begin{align*}
&  \|\widetilde{\boldsymbol{e}}_u^m\|_0^2+\|2 \boldsymbol{e}_u^{m}- 
\boldsymbol{e}_u^{m-1}\|_0^2+\frac{4}{3}(\Delta t)^2 \|\nabla e_p^m\|_0^2 \\
&\quad+\sum_{n=2}^m\left(\Delta t \nu \|\nabla \widetilde{\boldsymbol{e}}_u^n\|_0^2 + \Delta t\gamma \|\nabla 
\cdot  \widetilde{\boldsymbol{e}}_u^n\|_0^2+\|\delta_{tt}\boldsymbol{e}_u^{n}\|_0^2
 \right.\\&\qquad\qquad\left.
+ 2\Delta t \sum_{M\in\mathcal 
M_h}\tau_M^n\|\kappa_M((\widetilde{\boldsymbol{u}}_M^n\cdot\nabla)\widetilde{\boldsymbol{e}}_u^n)\|_{0,M}^2\right)\\
&\lesssim C_G\left(
\| \widetilde{\boldsymbol{e}}_u^1\|_0^2 +\|2 \boldsymbol{e}_u^1 - 
\boldsymbol{e}_u^0\|_0^2  + \frac{4}{3}(\Delta t)^2\|\nabla e_p^1\|_0^2 
+(\Delta t)^2(\|\boldsymbol{u}\|_{W^{1,\infty}(t_0,T;W^{k_u+1,2})}^2
\right.\\&\qquad\qquad
+\|p\|_{W^{1,\infty}(t_0,T;W^{k_p+1,2})}^2)
+(\nu+\gamma)h^{2k_u+2-2z}\|\boldsymbol{u}\|_{L^\infty(t_0,T;W^{k_u+1,2})}^2
 \\&\qquad\qquad
+\gamma^{-1}h^{2k_p+4-2z}\|p\|_{L^\infty(t_0,T;W^{k_p+1,2})}^2
+(\Delta t)^5 \|\boldsymbol{u}\|_{W^{3,\infty}(t_0,T;L^2)}
\\&\qquad\qquad
 +\left((\nu+\gamma)h^{2k_u}\|\boldsymbol{u}\|_{L^\infty(t_0,T;W^{k_u+1,2})}^2+\frac{h^{2k_p+2}}{\gamma}\|p\|_{L^\infty(t_0,T;W^{k_p+1,2})}^2\right.\\
&\qquad\qquad\quad\left.\left.+\nu h^{2s}\|\boldsymbol{u}\|_{L^\infty(t_0,T;W^{s+1,2})}^2\right)\left(\frac{\max_{1\leq n\leq m}\max_{M\in\mathcal M_h}\{\tau_M^n |\widetilde{\boldsymbol{u}}_M^n|^2\}}{\nu}+1\right)\right).
\end{align*}
With the initial error estimates from Lemma~\ref{lem:initial}, we finally obtain
\begin{align*}
&  \|\widetilde{\boldsymbol{e}}_u^m\|_0^2+\|2 \boldsymbol{e}_u^{m}- 
\boldsymbol{e}_u^{m-1}\|_0^2+\frac{4}{3}(\Delta t)^2 \|\nabla e_p^m\|_0^2 \\
&\quad+\sum_{n=2}^m\left(\Delta t \nu \|\nabla \widetilde{\boldsymbol{e}}_u^n\|_0^2 + \Delta t\gamma \|\nabla 
\cdot  \widetilde{\boldsymbol{e}}_u^n\|_0^2+\|\delta_{tt}\boldsymbol{e}_u^{n}\|_0^2
 \right.\\&\qquad\qquad\left.
+ 2\Delta t \sum_{M\in\mathcal 
M_h}\tau_M^n\|\kappa_M((\widetilde{\boldsymbol{u}}_M^n\cdot\nabla)\widetilde{\boldsymbol{e}}_u^n)\|_{0,M}^2\right)\\
&\lesssim C_G\left((\Delta t)^2(\|\boldsymbol{u}\|_{W^{1,\infty}(t_0,T;W^{k_u+1,2})}^2+\|p\|_{W^{1,\infty}(t_0,T;W^{k_p+1,2})}^2)
\right.\\&\qquad\qquad
+(\nu+\gamma)h^{2k_u+2-2z}\|\boldsymbol{u}\|_{L^\infty(t_0,T;W^{k_u+1,2})}^2
\\&\qquad\qquad+\gamma^{-1}h^{2k_p+4-2z}\|p\|_{L^\infty(t_0,T;W^{k_p+1,2})}^2
+(\Delta t)^5 \|\boldsymbol{u}\|_{W^{3,\infty}(t_0,T;L^2)}\\
&\qquad\qquad
 +\left((\nu+\gamma)h^{2k_u}\|\boldsymbol{u}\|_{L^\infty(t_0,T;W^{k_u+1,2})}^2+\frac{h^{2k_p+2}}{\gamma}\|p\|_{L^\infty(t_0,T;W^{k_p+1,2})}^2
 \right.\\&\qquad\qquad\quad\left.\left.
+\nu h^{2s}\|\boldsymbol{u}\|_{L^\infty(t_0,T;W^{s+1,2})}^2\right)\left(\frac{\max_{1\leq n\leq m}\max_{M\in\mathcal M_h}\{\tau_M^n |\widetilde{\boldsymbol{u}}_M^n|^2\}}{\nu}+1\right)\right).
\end{align*}
The Gronwall constant $C_G$ in the above estimates is given according to
\begin{align}
 C_G\sim \exp\left(\frac{T}{1-K_n}\right)
\end{align}
and we require $K_n< 1$. 
\end{proof}
In combination with the interpolation error estimates we can now state an error estimate for the full error.

\begin{theorem}\label{thm:ansatz_guermond}
Let the continuous solutions fulfill the regularity assumptions 
\begin{align*}
\boldsymbol{u}&\in W^{2,\infty}(t_0,T;L^2)\cap L^\infty(t_0,T;W^{k_u+1,2}) \\ p &\in W^{1,\infty}(t_0,T;H^1)\cap L^\infty(t_0,T;W^{k_p+1,2}).
\end{align*}
\begin{itemize}
  \item 
 If $\max_{1\leq n\leq m}\max_{M\in\mathcal M_h}\{\tau_M^n |\widetilde{\boldsymbol{u}}_M^n|^2\}\lesssim \nu h^{2k_u-2s}$, $\gamma\sim  1$ and $\Delta t\lesssim C$, the total error 
 $\widetilde{\boldsymbol{\zeta}}_u^n=\boldsymbol{u}(t_n)-\widetilde{\boldsymbol{u}}_h^n$
 can be bounded as
 \begin{align*}
 &\|\widetilde{\boldsymbol{\zeta}}_u^m\|_0^2+\Delta t\sum_{n=2}^m\left(\nu \|\nabla \widetilde{\boldsymbol{\zeta}}_u^n\|_0^2 + \gamma \|\nabla 
\cdot  \widetilde{\boldsymbol{\zeta}}_u^n\|_0^2
\right.\\&\qquad\qquad\qquad\qquad\left.
+\sum_{M\in\mathcal 
M_h}\tau_M^n\|\kappa_M((\widetilde{\boldsymbol{u}}_M^n\cdot\nabla)\widetilde{\boldsymbol{\zeta}}_u^n)\|_{0,M}^2\right)\\
  &\lesssim C_G\left((\Delta t)^2 \|p\|_{W^{1,\infty}(t_0,T;H^2)}^2
+h^{2k_u}\|\boldsymbol{u}\|_{L^\infty(t_0,T;W^{k_u+1,2})}^2
 \right.\\&\qquad\qquad\left.
+h^{2k_p+2}\|p\|_{L^\infty(t_0,T;W^{k_p+1,2})}^2\right).
\end{align*}
\item If $\max_{1\leq n\leq m}\max_{M\in\mathcal M_h}\{\tau_M^n |\widetilde{\boldsymbol{u}}_M^n|^2\}\lesssim \nu h^{2k_u-2s+2}$, $\gamma\sim 1$ and $\Delta t\lesssim h^2$ we can bound the total energy error according to
  \begin{align*}
  \|\widetilde{\boldsymbol{\zeta}}_u^m\|_0^2
  &\lesssim C_G\left((\Delta t)^2 \|p\|_{W^{1,\infty}(t_0,T;H^2)}^2
+h^{2k_u+2}\|\boldsymbol{u}\|_{L^\infty(t_0,T;W^{k_u+1,2})}^2
 \right.\\&\qquad\qquad\left.
+h^{2k_p+4}\|p\|_{L^\infty(t_0,T;W^{k_p+1,2})}^2\right).
 \end{align*}
 \end{itemize}
\end{theorem}
\begin{proof}
 The claim follows from the combination of the interpolation properties and the estimate (\ref{eqn:fullerror1}) except for the nonlinear stabilization. For this last term we observe
   \begin{align*}
&\Delta t \sum_{n=1}^m\sum_{M\in\mathcal 
M_h}\tau_M^n\|\kappa_M((\widetilde{\boldsymbol{u}}_h^n\cdot\nabla)\boldsymbol{\zeta}_u^n)\|_{0,M}^2\\
&\leq
C \Delta t \sum_{n=1}^m\sum_{M\in\mathcal M_h} \left( 
\tau_M^n\|\kappa_M((\widetilde{\boldsymbol{u}}_h^n\cdot\nabla)\widetilde{\boldsymbol{e}}_u^n)\|_{0,M}^2
+ \tau_M^n\|\kappa_M((\widetilde{\boldsymbol{u}}_h^n\cdot\nabla)\boldsymbol{\eta}_u^n)\|_{0,M}^2 \right)\\
&\leq
C \Delta t \sum_{n=1}^m\sum_{M\in\mathcal M_h} \left( 
\tau_M^n\|\kappa_M((\widetilde{\boldsymbol{u}}_h^n\cdot\nabla)\widetilde{\boldsymbol{e}}_u^n)\|_{0,M}^2
+  \tau_M^n|\widetilde{\boldsymbol{u}}_M^n|^2\|\nabla\boldsymbol{\eta}_u^n\|_{0,M}^2\right) \\
&\lesssim C_G\left((\Delta t)^2 \|p\|_{W^{1,\infty}(t_0,T;H^2)}^2
+(\nu+\gamma)h^{2k_u+2-2z}\|\boldsymbol{u}\|_{L^\infty(t_0,T;W^{k_u+1,2})}^2
 \right.\\&\qquad\qquad
+\gamma^{-1}h^{2k_p+4-2z}\|p\|_{L^\infty(t_0,T;W^{k_p+1,2})}^2
+(\Delta t)^5 \|\boldsymbol{u}\|_{W^{3,\infty}(t_0,T;L^2)}
\\&\qquad\qquad
 +\left(\frac{\nu+\gamma}{\nu}h^{2k_u}\|\boldsymbol{u}\|_{L^\infty(t_0,T;W^{k_u+1,2})}^2+\frac{h^{2k_p+2}}{\gamma\nu}\|p\|_{L^\infty(t_0,T;W^{k_p+1,2})}^2\right.\\
&\qquad\qquad\qquad\left.\left.+h^{2s}\|\boldsymbol{u}\|_{L^\infty(t_0,T;W^{s+1,2})}^2\right)\max_{1\leq n\leq m}\max_{M\in\mathcal M_h}\{\tau_M^n |\widetilde{\boldsymbol{u}}_M^n|^2\}\right).
 \end{align*}
 For the first estimate we are satisfied with choosing $z=1$ due to the fact that then the order of convergence for the discretization error matches the one for the approximation error if $\max_{1\leq n\leq m}\max_{M\in\mathcal M_h}\{\tau_M^n |\widetilde{\boldsymbol{u}}_M^n|^2\}\lesssim \nu h^{2k_u-2s}$. The restriction on $K_n$ here simplifies to $\Delta t\lesssim C$.
 For the second estimate we choose in the above estimate and in (\ref{eqn:fullerror1}) $z=0$ to get the same order of convergence for the discretization as for the approximation error with respect to spatial discretization. This also means that we have to choose the LPS SU parameter according to $\max_{1\leq n\leq m}\max_{M\in\mathcal M_h}\{\tau_M^n |\widetilde{\boldsymbol{u}}_M^n|^2\}\lesssim \nu h^{2k_u-2s+2}$. The requirement on $K_n$ can only be satisfied if $\Delta t \lesssim h^{2-2z}$ which here means $\Delta t \lesssim h^2$.\\ 
\end{proof}

\begin{remark}
The above error estimates are quasi-robust both for the LPS and the energy norm, i.e. the right-hand sides do not depend explicitly on $\nu$. With respect to the LPS norm the results are quasi-optimal in space and time. However, considering the $L^2(\boldsymbol{u})$ norm the temporal order of convergence is suboptimal due 
to the estimate (\ref{eqn:diff_pressure_3}) for the gradient of the pressure error but quasi-optimal in space.
\end{remark}

\begin{remark}
For inf-sup stable pairs, we usually have $k_u=k_p+1$ and this is also what an equilibration of the spatial rates of convergence in
Theorem~\ref{thm:ansatz_guermond} suggests. To obtain the same error magnitude with respect to the time discretization for the LPS error $\Delta t\sim h^{k_u}$ and for the energy error $\Delta t\sim h^{k_u+1}$ should be chosen.
\end{remark}

\begin{remark}
Standard stability estimates for the semi-discretized Navier-Stokes problem do not imply the condition $\|\boldsymbol{u}_h\|_{l^\infty(t_0,T;[L^\infty(\Omega)]^d)}\lesssim C(\boldsymbol{u},p)$ that would be necessary for a Gronwall constant independent from the discrete solution in Lemma~\ref{lem:convectiveterms}. Hence, the results give no a priori bounds. Avoiding to use the discrete solution on the right-hand side of the estimate for the convective terms, Lemma~\ref{lem:convectiveterms}, would lead to mesh width restrictions of the form 
\begin{align*} 
Re_M  &= \frac{h_M\|\boldsymbol{u}_h\|_{\infty,M}}{\nu} \le  \frac{1}{\sqrt{\nu}}
\end{align*}
similar to the ones obtained in \cite{arndt2014local}.\\
Another way, different from the approach taken here, to circumvent this mesh width restrictions would be to consider only fine and coarse ansatz spaces that fulfill
the compatibility condition \cite{matthies}:
\begin{assumption}\label{ass:compcoarsefine}
For $Y_{h,k_u}(M) = \{v_h|_M  \colon~v_h\in Y_{h,k_u},~v_h=0 \mbox{ on } \Omega\setminus M\}$ 
there exists $\widetilde{\beta}_{u,h}>0$ such that 
\begin{align}  \label{eqn:L2-infsup}
 \inf_{\boldsymbol{w}_h \in D_M} \sup_{\boldsymbol{v}_h\in [Y_{h,k_u}(M)]^d} 
         \frac{(\boldsymbol{v}_h,\boldsymbol{w}_h)_M}{\|\boldsymbol{v}_h\|_{0,M} \|\boldsymbol{w}_h\|_{0,M}} \geq \widetilde{\beta}_{u,h}.
\end{align}
\end{assumption}
\noindent For the sake of brevity, we omit restating the results in this framework.
\end{remark}

\section{Quasi-Optimal Errors in Time}\label{sec:temporal-spatial}

In this section, we aim for improving the rate of convergence with respect to the temporal discretization. 
In comparison to \cite{guermond1999resultat}, the dependence on the diffusion coefficient $\nu$ is critically considered.\\

The strategy that we consider here is to first consider the temporal discretization and afterwards its spatial approximation.
For the discretization in time, we in fact do not consider the Navier-Stokes problem but take the convective term from the continuous problem as a force term.
By the triangle inequality, the total error is then bounded by
\begin{align}
 \| \boldsymbol{\mathcal{U}}-\boldsymbol{\mathcal{U}}_h\| \leq \|\boldsymbol{\mathcal{U}}-\boldsymbol{\mathcal{W}}_t\|+\|
\boldsymbol{\mathcal{W}}_t-\boldsymbol{\mathcal{U}}_h\|.
\label{eqn:ts_strategy_error}
\end{align}

\begin{assumption}\label{ass:regularity_ts}
The temporal discretized quantities fulfill for any $1\le n\le N$ the regularity assumptions
\begin{align}
 \boldsymbol{w}_t^n &\in [W^{k_u+1,2}(\Omega)]^d,& r_t^n&\in W^{k_p+1,2}(\Omega), &
 \widetilde{\boldsymbol{u}}_h^n &\in [H^1(\Omega)]^d,& p_h^n&\in H^1(\Omega)
\end{align}
and we assume for the continuous solution
\begin{align*}
 \boldsymbol{u} &\in L^\infty(t_0,T;[W^{k_u+1,2}(\Omega)]^d), \quad \partial_t \boldsymbol{u} \in L^\infty(t_0,T;[H^2(\Omega)]^d), \\
 p &\in L^2(t_0,T; W^{k_p+1,2}(\Omega))\cap H^2(t_0,T; H^1(\Omega)).
\end{align*}
\end{assumption}

\subsection{Time Discretization}\label{sec:errors_timedisc}
We consider a grad-div stabilized, but spatially continuous problem:
\begin{center}
\begin{myminipage}
 Find $\widetilde{\boldsymbol{w}}_t^n\in\boldsymbol{V}, \boldsymbol{w}_t^n\in\boldsymbol{V}^{div}$ and $r_t^n\in Q$ 
 such that
\begin{align}
 \begin{aligned}
   &\left(\frac{3 \widetilde{\boldsymbol{w}}_t^n-4\boldsymbol{w}_t^{n-1}+\boldsymbol{w}_t^{n-2}}{2 \Delta 
t}, \boldsymbol{v}\right) +\nu(\nabla \widetilde{\boldsymbol{w}}_t^n, \nabla \boldsymbol{v}) 
+ \gamma(\nabla \cdot \widetilde{\boldsymbol{w}}_t^n, \nabla \cdot\boldsymbol{v}) 
\\&\quad
= (\boldsymbol{f}^n, \boldsymbol{v}) -(\nabla r_t^{n-1}, \boldsymbol{v})
- c(\boldsymbol{u}(t_n);\boldsymbol{u}(t_n),{\boldsymbol{v}})\\
&\widetilde{\boldsymbol{w}}_t^n|_{\partial\Omega}=0
\label{eqn:ts_time1_aux}
 \end{aligned}
 \end{align}
 \begin{align}
 \begin{aligned}
  \left(\frac{3 \boldsymbol{w}_t^n-3\widetilde{\boldsymbol{w}}_t^n}{2 \Delta t}+\nabla (r_t^{n}-r_t^{n-1}), 
\boldsymbol{v}\right)&=0\\
(\nabla \cdot \boldsymbol{w}_t^n, q)&=0\\
\boldsymbol{w}_t^n|_{\partial\Omega}&=0
\label{eqn:ts_time2_aux}
 \end{aligned}
\end{align}
holds for all $\boldsymbol{v}\in \boldsymbol{V}$ and $q\in Q$.
\end{myminipage}
\end{center}
 Compared to the fully discrete case where we test the projection equation in $\boldsymbol{Y}\!_h$
 we do not need an auxiliary space due to the Helmholtz decomposition
 \begin{align*}
  \boldsymbol{V}=\boldsymbol{V}^{div}\oplus\nabla Q.
 \end{align*}
\begin{definition}
We denote the errors for this semi-discretization by 
\begin{align}
\label{eqn:error_semi2t}
\boldsymbol{\eta}_{u,t}^n:= \boldsymbol{u}(t_n)- \boldsymbol{w}_t^n \quad
\widetilde{\boldsymbol{\eta}}_{u,t}^n := \boldsymbol{u}(t_n)- \widetilde{\boldsymbol{w}}_t^n, \quad \eta_{p,t}^n:=p(t_n)-r_t^n.
\end{align}
\end{definition}

For the error estimates we refer to our considerations in \cite{technicalreportAD2015}. There we prove quasi-robust and quasi-optimal
error estimates for the time discretization alone. The steps taken there are again
quite similar to \cite{guermond1999resultat}. The idea is to first achieve a bound on the propagation error $\|\delta_t \widetilde{\boldsymbol{\eta}}_u^n\|_0$, on $\|\widetilde{\boldsymbol{\eta}}_u^n-\boldsymbol{\eta}_u^n\|_0$ and on $\sqrt{\nu}\|\nabla \widetilde{\boldsymbol{\eta}}_u^n\|_0+\sqrt{\gamma}\|\nabla\cdot \widetilde{\boldsymbol{\eta}}_u^n\|_0$. 
 With these preparations, the pressure term that caused the suboptimal temporal results can be eliminated in the error equation by testing with an inverse Stokes operator applied to $\widetilde{\boldsymbol{e}}_u^n$. This then gives the desired estimate on $\|\widetilde{\boldsymbol{e}}_u^n\|_0$.

\begin{theorem}
 The error due to time discretization can be bounded for all $m\leq N$ according to
 \begin{align*}
  \|\widetilde{\boldsymbol{\eta}}_{u,t}^m\|_0^2&\leq C (\Delta t)^4, &
  \Delta t\sum_{n=1}^m \nu \|\nabla\widetilde{\boldsymbol{\eta}}_{u,t}^m\|_0^2
  +\gamma \|\nabla\cdot\widetilde{\boldsymbol{\eta}}_{u,t}^m\|_0^2&\leq C(\Delta t)^2.
 \end{align*}
 where the $C$ depends only depends on Sobolev norms of continuous solution and not explicitly on $\nu$.
\end{theorem}

\subsection{Spatial Discretization}
Now that we considered the error due temporal discretization, we finalize this approach by comparing the full discretization (\ref{eqn:full1})-(\ref{eqn:full2}) with the discretization in time:
\begin{center}
 \begin{myminipage}
 Find $\widetilde{\boldsymbol{w}}_t^n\in\boldsymbol{V}, \boldsymbol{w}_t^n\in\boldsymbol{V}^{div}$ and $r_t^n\in Q$ such that
 \begin{align}
 \begin{aligned}
  &\left(\frac{3 \widetilde{\boldsymbol{w}}_t^n-4\boldsymbol{w}_t^{n-1}+\boldsymbol{w}_t^{n-2}}{2 \Delta 
t}, \boldsymbol{v}\right) 
+\nu(\nabla \widetilde{\boldsymbol{w}}_t^n, \nabla \boldsymbol{v}) 
+ \gamma (\nabla\cdot \widetilde{\boldsymbol{u}}^n_{t}, \nabla\cdot \boldsymbol{v}) 
\\&\quad
= (\boldsymbol{f}^n, \boldsymbol{v}) -(\nabla r_t^{n-1}, \boldsymbol{v})-c(\boldsymbol{u}(t_n);\boldsymbol{u}(t_n), \boldsymbol{v})\\
\widetilde{\boldsymbol{w}}_t^n|&_{\partial\Omega}=0, 
\label{eqn:ts_semi-full1}
 \end{aligned}
 \end{align}
 \begin{align}
 \begin{aligned}
  \left(\frac{3 \boldsymbol{w}_t^n-3\widetilde{\boldsymbol{w}}_t^n}{2 \Delta t}+\nabla (r_t^n-r_t^{n-1}), \boldsymbol{y}\right)&=0, \\
(\nabla \cdot \boldsymbol{w}_t^n, q)&=0, \\
\boldsymbol{w}_t^n|_{\partial \Omega}&=0.
\label{eqn:ts_semi-full2}
 \end{aligned}
\end{align}
holds for all $\boldsymbol{v}\in \boldsymbol{V}$, $\boldsymbol{y}\in [L^2(\Omega)]^d$ and $q\in Q$.
\end{myminipage}
\end{center}

\subsubsection{Notation}
We use the abbreviations 
\begin{align}
 \label{eqn:error_semi2h_eta}
\widetilde{\boldsymbol{\eta}}_{u,h}^n:=\widetilde{\boldsymbol{w}}_t^n - j_u\widetilde{\boldsymbol{w}}_t^n, 
\quad \boldsymbol{\eta}_{u,h}^n:=\boldsymbol{w}_t^n - j_u \boldsymbol{w}_t^n, 
\quad \eta_{p,h}^n:=r_t^n - j_p r_t^n
\end{align}
for the approximation errors and 
\begin{align}
\label{eqn:error_semi2h_e}
\widetilde{\boldsymbol{e}}_{u,h}^n:=j_u\widetilde{\boldsymbol{w}}_t^n-\widetilde{\boldsymbol{u}}_h^n, 
\quad \boldsymbol{e}_{u,h}^n:=j_u\boldsymbol{w}_t^n - \boldsymbol{u}_h^n, 
\quad e_{p,h}^n:=j_p r_t^n - p_h^n
\end{align}
for the discretization errors.
Hence, the errors due to spatial discretization can be written as
\begin {align*}
&\widetilde{\boldsymbol{\xi}}_{u,h}^n:=\widetilde{\boldsymbol{w}}_t^n-\widetilde{\boldsymbol{u}}_h^n=\widetilde{
\boldsymbol{\eta}}_{u,h}^n+ \widetilde{\boldsymbol{e}}_{u,h}^n, 
\quad \boldsymbol{\xi}_{u,h}^n:=\boldsymbol{w}_t^n-\boldsymbol{u}_h^n=\boldsymbol{\eta}_{u,h}^n+ 
\boldsymbol{e}_{u,h}^n, \quad
 \\ &
\xi_{p,h}^n:=r_t^n-p_h^n=\eta_{p,h}^n+ e_{p,h}^n.
\end {align*}

For the interpolation we again use Stokes interpolants defined analogously to the first approach:
 $(\boldsymbol{w}_h,\widetilde{r}_h)$ is given as solution of the Stokes problem
\begin{center}
 \begin{myminipage}
  Find $\boldsymbol{w}_h^n\in\boldsymbol{V}\!_h$ and $\widetilde{r}_h^{n-1} \in Q_h$ such that
\begin{align*}
&\nu(\nabla \boldsymbol{w}_h^n, \nabla \boldsymbol{v}_h) + \gamma (\nabla\cdot 
\boldsymbol{w}_h^n, \nabla\cdot \boldsymbol{v}_h) 
-( \widetilde{r}_h^{n-1}, \nabla\cdot \boldsymbol{v}_h)
 \\&
= \nu(\nabla \boldsymbol{w}_t^n, \nabla \boldsymbol{v}_h) 
+ \gamma (\nabla\cdot \boldsymbol{w}_t^n, \nabla\cdot \boldsymbol{v}_h)
-( r_t^{n-1}, \nabla\cdot \boldsymbol{v}_h)\\
(\nabla\cdot\boldsymbol{w}_h^n, q_h)&=(\nabla\cdot \boldsymbol{w}_t^n, q_h)
\end{align*}
holds for all $\boldsymbol{v}_h\in \boldsymbol{V}\!_h$ and $q_h\in Q_h$.
\end{myminipage}
\end{center}
 and $(\widetilde{\boldsymbol{w}}_h,r_h)$ is given as solution of the Stokes problem
 \begin{center}
 \begin{myminipage}
  Find $\widetilde{\boldsymbol{w}}_h^n\in\boldsymbol{V}\!_h$ and $r_h^{n-1} \in Q_h$ such that
\begin{align*}
&\nu(\nabla\widetilde{\boldsymbol{w}}_h^n, \nabla \boldsymbol{v}_h) + \gamma 
(\nabla\cdot\widetilde{\boldsymbol{w}}_h^n, \nabla\cdot \boldsymbol{v}_h) 
-( r_h^{n-1}, \nabla\cdot \boldsymbol{v}_h)
 \\&
= \nu(\nabla\widetilde{\boldsymbol{w}}_t^n, \nabla \boldsymbol{v}_h) 
+ \gamma (\nabla\cdot\widetilde{\boldsymbol{w}}_t^n, \nabla\cdot \boldsymbol{v}_h)
-( r_t^{n-1}, \nabla\cdot \boldsymbol{v}_h)\\
(\nabla\cdot\boldsymbol{\boldsymbol{w}}_h^n, q_h)&=(\nabla\cdot\widetilde{
\boldsymbol{w}}_t^n, q_h)
\end{align*}
holds for all $\boldsymbol{v}_h\in \boldsymbol{V}\!_h$ and $q_h\in Q_h$.
\end{myminipage}
\end{center}

Analogously to the results above, the Stokes interpolations yield the interpolation errors
\begin{align}
\begin{aligned}
&\|\boldsymbol{\eta}_{u,t}^n\|_0^2+h^2(\nu\|\boldsymbol{\eta}_{u,t}^n\|_1^2+\gamma\| \nabla\cdot\boldsymbol{\eta}_{u,t}^n\|_0^2+\|\eta_{p,t}^n\|_0^2)
 \\&\quad
+\|\widetilde{\boldsymbol{\eta}}_{u,t}^n\|_0^2+h^2(\nu\|\widetilde{\boldsymbol{\eta}}_{u,t}^n\|_1^2+\gamma\| \nabla\cdot\widetilde{\boldsymbol{\eta}}_{u,t}^n\|_0^2+\|\widetilde{\eta}_p^n\|_0^2)\\
 &\leq C ((\nu+  \gamma)h^{2k_u+2}\|\boldsymbol{w}_t^n\|_{W^{k_u+1,2}}^2+\gamma^{-1}h^{2k_p+4}\|r_t^{n-1}\|_{W^{k_p+1,2}}^2)
 \\ &\quad
 + C ((\nu+  \gamma)h^{2k_u+2}\|\widetilde{\boldsymbol{w}}_t^n\|_{W^{k_u+1,2}}^2+\gamma^{-1}h^{2k_p+4}\|\widetilde{r}_t^{n-1}\|_{W^{k_p+1,2}}^2). 
 \end{aligned}
\end{align}
provided the the semi-discretized quantities are sufficiently smooth.

\subsubsection{Initial Errors}
For the first time step we use a BDF1 instead of the BDF2 scheme. In particular, the convection-diffusion and the projection step
in the fully discretized setting read:
\begin{center}
 \begin{myminipage}
   Find $\widetilde{\boldsymbol{u}}_h^1\in\boldsymbol{V}\!_h,\boldsymbol{u}_h^1\in\boldsymbol{Y}\!_h$ and $p_h^1 \in Q_h$ such that
\begin{align}
 \begin{aligned}
   &\left(\frac{\widetilde{\boldsymbol{u}}_h^1-\boldsymbol{u}_h^0}{\Delta 
t}, \boldsymbol{v}_h\right) 
+\nu(\nabla \widetilde{\boldsymbol{u}}_h^1, \nabla \boldsymbol{v}_h) 
+ c(\widetilde{\boldsymbol{u}}_h^1;\widetilde{\boldsymbol{u}}_h^1, \boldsymbol{v}_h)
\\&\quad
+ s_h(\widetilde{\boldsymbol{u}}_h^1, \widetilde{\boldsymbol{u}}_h^1, 
\widetilde{\boldsymbol{u}}_h^1, \boldsymbol{v}_h) + \gamma (\nabla\cdot \widetilde{\boldsymbol{u}}_h^1, \nabla\cdot \boldsymbol{v}_h) 
= 
(\boldsymbol{f}^1, \boldsymbol{v}_h) -(\nabla p_h^0, \boldsymbol{v}_h)
\quad \forall \boldsymbol{v}_h\in \boldsymbol{V}\!_h\\
&\widetilde{\boldsymbol{u}}_h^1|_{\partial\Omega}=0, 
 \end{aligned}
 \end{align}
 \begin{align}
 \begin{aligned}
  \left(\frac{\boldsymbol{u}_h^1-\widetilde{\boldsymbol{u}}_h^1}{\Delta t}+\nabla (p_h^1-p_h^0), \nabla 
\boldsymbol{y}_h\right)&=0 \quad \forall \boldsymbol{y_h}\in \boldsymbol{Y}\!_h\\
(\nabla \cdot \boldsymbol{u}_h^1, q_h)&=0 \quad \forall q_h\in Q_h,\\
\boldsymbol{u}_h^1|_{\partial \Omega}&=0
 \end{aligned}
\end{align}
with the initial values $\widetilde{\boldsymbol{u}}_h^0=j_u\widetilde{\boldsymbol{w}}_t^0, \boldsymbol{u}_h^0=j_u 
\boldsymbol{w}_t^0$ and $p_h^0=j_p r_t^0$.
\end{myminipage}
\end{center}
Therefore, the initial discretization errors vanish: 
$\|\widetilde{\boldsymbol{e}}_{u,h}^0\|=\|\boldsymbol{e}_{u,h}^0\|=\|\eta_{p,t}^0\|=0$.

The technique used in the next Lemma is the same as later for the case $n\geq2$. 
For the sake of completeness and the convenience of the reader, we nevertheless consider both cases.

\begin{lemma}
 The initial errors due to spatial discretization $\boldsymbol{e}_{u,h}^1=j_u\boldsymbol{w}_t^1-\boldsymbol{u}_h^1$ and $\widetilde{\boldsymbol{e}}_{u,h}^1=j_u\widetilde{\boldsymbol{w}}_t^1-\widetilde{\boldsymbol{u}}_h^1$ can be bounded as
  \begin{align}
\begin{aligned}
 &C\left(1-\frac{\Delta t}{\nu}-\Delta t\frac{(\nu+\gamma)^2h^{4k_u}+\gamma^{-2}h^{4k_p+4} + (\Delta t)^4 (\nu+\gamma+\gamma^{-1})^2}{\nu^5}\right)\|\widetilde{\boldsymbol{e}}_{u,h}^1\|_0^2
 \\&\quad
+\nu\Delta t\|\nabla \widetilde{\boldsymbol{e}}_{u,h}^1\|_0^2 
+ \gamma\Delta t \|\nabla \cdot\widetilde{\boldsymbol{e}}_{u,h}^1\|_0^2
+2\Delta t \sum_{M\in\mathcal M_h} \tau_M^1 
\|\kappa_M((\widetilde{\boldsymbol{u}}_h^1\cdot\nabla)\widetilde{\boldsymbol{e}}_{u,h}^1)\|_{0,M}^2
\\&\leq
C \left((\nu+\gamma)h^{2k_u+2}+\gamma^{-1}h^{2k_p+4} + (\Delta t)^2 h^2 \left(\frac{\nu+\gamma}{\nu}+\frac{1}{\nu\gamma}\right)\right)\left(1+\frac{\Delta t}{\nu}\right)
\\&\quad
+C\Delta t\frac{(\nu+\gamma)^2h^{4k_u}+\gamma^{-2}h^{4k_p+4} + (\Delta t)^4 (\nu+\gamma+\gamma^{-1})^2}{\nu^3}
\\&\quad
+C \frac{\Delta t}{\nu}
\max_{M\in\mathcal M_h}\{\tau_M^1 |\widetilde{\boldsymbol{u}}_M^1 |^2\}
\Big( (\nu+\gamma)h^{2k_u}+\frac{h^{2k_p+2}}{\gamma} + (\Delta t)^2 \Big(\nu+\gamma+\frac{1}{\gamma}\Big)+\nu h^{2s}\Big).
\label{eqn:ts_initial_space_errors}
\end{aligned}
\end{align}
\end{lemma}

\begin{proof}
Testing the difference of the convection-diffusion equations with $\widetilde{\boldsymbol{e}}_{u,h}^1$ gives
 \begin{align}
 \begin{aligned}
 &\|\widetilde{\boldsymbol{e}}_{u,h}^1\|_0^2
+\nu\Delta t\|\nabla \widetilde{\boldsymbol{e}}_{u,h}^1\|_0^2 + \gamma\Delta t \|\nabla 
\cdot\widetilde{\boldsymbol{e}}_{u,h}^1\|_0^2
+\Delta t \sum_{M\in \mathcal{M}_h}\left(\tau_M^1 \|\kappa_M((\widetilde{\boldsymbol{u}}_h^1\cdot\nabla)\widetilde{\boldsymbol{e}}_{u,h}^1)\|_0^2\right)\\
&= 
\Delta t(c(\widetilde{\boldsymbol{u}}_h^1, \widetilde{\boldsymbol{u}}_h^1, \widetilde{\boldsymbol{e}}_{u,h}^1)-c(\widetilde{
\boldsymbol{w}}_t^1, \widetilde{\boldsymbol{w}}_t^1, \widetilde{\boldsymbol{e}}_{u,h}^1))
- \Delta t(\nabla \eta_{p,h}^0, \widetilde{\boldsymbol{e}}_{u,h}^1)  
 \\&\quad 
+\Delta t( 
s_h(\widetilde{\boldsymbol{u}}_h^1, \widetilde{\boldsymbol{u}}_h^1, \widetilde{\boldsymbol{u}}_h^1, \widetilde{
\boldsymbol{e}}_{u,h}^1)
+s_h(\widetilde{\boldsymbol{u}}_h^1, \widetilde{\boldsymbol{e}}_{u,h}^1, \widetilde{\boldsymbol{u}}_h^1, \widetilde{
\boldsymbol{e}}_{u,h}^1))
\\&\quad
-(\widetilde{\boldsymbol{\eta}}_{u,h}^1-\boldsymbol{\eta}_{u,h}^0, \widetilde{\boldsymbol{e}}_{u,h}^1) 
-\nu \Delta t(\nabla \widetilde{\boldsymbol{\eta}}_{u,h}^1, \nabla \widetilde{\boldsymbol{e}}_{u,h}^1) - \gamma\Delta 
t (\nabla \cdot\widetilde{\boldsymbol{\eta}}_{u,h}^1, \nabla \cdot\widetilde{\boldsymbol{e}}_{u,h}^1)\\
&= 
\Delta t(c(\widetilde{\boldsymbol{u}}_h^1, \widetilde{\boldsymbol{u}}_h^1, \widetilde{\boldsymbol{e}}_{u,h}^1)-c(\widetilde{
\boldsymbol{w}}_t^1, \widetilde{\boldsymbol{w}}_t^1, \widetilde{\boldsymbol{e}}_{u,h}^1)) 
+\Delta t( 
s_h(\widetilde{\boldsymbol{u}}_h^1, \widetilde{\boldsymbol{u}}_h^1, \widetilde{\boldsymbol{u}}_h^1, \widetilde{
\boldsymbol{e}}_{u,h}^1)
\\&\quad+s_h(\widetilde{\boldsymbol{u}}_h^1, \widetilde{\boldsymbol{e}}_{u,h}^1, \widetilde{\boldsymbol{u}}_h^1, \widetilde{
\boldsymbol{e}}_{u,h}^1))
-(\widetilde{\boldsymbol{\eta}}_{u,h}^1-\boldsymbol{\eta}_{u,h}^0, \widetilde{\boldsymbol{e}}_{u,h}^1).
 \end{aligned}
\end{align}
In the last step we used the special choice of the interpolant.

From Lemma~\ref{lem:convective2}, the error with respect to the convective terms is given by
\begin{align*}
 &c(\boldsymbol{u}(t_1), \boldsymbol{u}(t_1), \widetilde{\boldsymbol{e}}_{u,h}^1)- 
c(\widetilde{\boldsymbol{u}}_h^1, \widetilde{\boldsymbol{u}}_h^1, \widetilde{\boldsymbol{e}}_{u,h}^1) \\
&\leq 
C\|\widetilde{\boldsymbol{e}}_{u,h}^1\|_0^2\left(\frac{\|\widetilde{\boldsymbol{\eta}}_{u,t}^1\|_1^4+\|\widetilde{\boldsymbol{\eta}}_{u,h}^1\|_1^4}{\nu^3}+\frac{\|\boldsymbol{u}(t_1)\|_2^2}{\nu}\right)+\frac{\nu}{4}\|\widetilde{\boldsymbol{e}}_{u,h}^1\|_1^2 
 \\&\quad
+\frac{C}{\nu}(\|\widetilde{\boldsymbol{\eta}}_{u,t}^1\|_1^4+\|\widetilde{\boldsymbol{\eta}}_{u,h}^1\|_1^4+\|\widetilde{\boldsymbol{\eta}}_{u,h}^1\|_0^2\|\boldsymbol{u}(t_1)\|_2^2).
\end{align*}

For the nonlinear stabilization we again use $\widetilde{\boldsymbol{w}}_h^1= \boldsymbol{u}(t_1) - 
(\widetilde{\boldsymbol{\eta}}_{u,t}^1+\widetilde{\boldsymbol{\eta}}_{u,h}^1)$ and obtain
\begin{align*}
&s_h(\widetilde{\boldsymbol{u}}_h^1, \widetilde{\boldsymbol{w}}_h^1, \widetilde{\boldsymbol{u}}_h^1, \widetilde{
\boldsymbol{e}}_{u,h}^1)
= 
s_h(\widetilde{\boldsymbol{u}}_h^1, \boldsymbol{u}(t_1), \widetilde{\boldsymbol{u}}_h^1, \widetilde{\boldsymbol{e}}_{
u, h}^1) 
- 
s_h(\widetilde{\boldsymbol{u}}_h^1, \widetilde{\boldsymbol{\eta}}_{u,t}^1+\widetilde{\boldsymbol{\eta}}_{u,h}^1, \widetilde{\boldsymbol{u}}_h^1, \widetilde{
\boldsymbol{e}}_{u,h}^1)\\
&\leq C \sum_{M\in\mathcal M_h} \tau_M^1 |\widetilde{\boldsymbol{u}}_M^1|^2 
\|\kappa_M(\nabla\boldsymbol{u}(t_1))\|_{0,M}^2
+ C \sum_{M\in\mathcal M_h} \tau_M^1 |\widetilde{\boldsymbol{u}}_M^1|^2 
\|\nabla(\widetilde{\boldsymbol{\eta}}_{u,t}^1+\widetilde{\boldsymbol{\eta}}_{u,h}^1)\|_{0,M}^2
\\&\quad
+\frac14 \sum_{M\in\mathcal M_h} 
\tau_M^1\|\kappa_M((\widetilde{\boldsymbol{u}}_M^1\cdot\nabla)\widetilde{\boldsymbol{e}}_{u,h}^1)\|_{0,M}^2\\
&\leq C \max_{M\in\mathcal M_h} \{ \tau_M^1 |\widetilde{\boldsymbol{u}}_M^1|^2\} 
\left(\|\nabla(\widetilde{\boldsymbol{\eta}}_{u,t}^1+\widetilde{\boldsymbol{\eta}}_{u,h}^1)\|_0^2 + \sum_{M\in\mathcal 
M_h}\|\kappa_M(\nabla\boldsymbol{u}(t_1))\|_{0,M}^2 \right)
\\&\quad
+\frac14 \sum_{M\in\mathcal M_h} 
\tau_M^1\|\kappa_M((\widetilde{\boldsymbol{u}}_M^1\cdot\nabla)\widetilde{\boldsymbol{e}}_{u,h}^1)\|_{0,M}^2.
\end{align*}

We summarize the estimates in
\begin{align}
\begin{aligned}
 &C\left(1-\frac{\Delta t}{\nu}\|\boldsymbol{u}(t_1)\|_2^2-\Delta t\frac{\|\widetilde{\boldsymbol{\eta}}_{u,t}^1\|_1^4+\|\widetilde{\boldsymbol{\eta}}_{u,h}^1\|_1^4}{\nu^3}\right)\|\widetilde{\boldsymbol{e}}_{u,h}^1\|_0^2
+\nu\Delta t\|\nabla \widetilde{\boldsymbol{e}}_{u,h}^1\|_0^2 
\\&\quad
+ \gamma\Delta t \|\nabla \cdot\widetilde{\boldsymbol{e}}_{u,h}^1\|_0^2
+2\Delta t \sum_{M\in\mathcal M_h} \tau_M^1 
\|\kappa_M((\widetilde{\boldsymbol{u}}_h^1\cdot\nabla)\widetilde{\boldsymbol{e}}_{u,h}^1)\|_{0,M}^2\\
&\leq  C \Delta t \max_{M\in\mathcal M_h}\{\tau_M^1 |\widetilde{\boldsymbol{u}}_M^1 |^2\}
\left(\|\widetilde{\boldsymbol{\eta}}_{u,t}^1\|_1^2+\|\nabla\widetilde{\boldsymbol{\eta}}_{u,h}^1\|_0^2 + \sum_{M\in\mathcal 
M_h}\|\kappa_M(\nabla\boldsymbol{u}(t_1))\|_{0,M}^2 \right)
\\&\quad
+C\|\widetilde{\boldsymbol{\eta}}_{u,h}^1-\boldsymbol{\eta}_{u,h}^0\|_0^2 
+C \frac{\Delta t}{\nu}(\|\widetilde{\boldsymbol{\eta}}_{u,t}^1\|_1^4+\|\widetilde{\boldsymbol{\eta}}_{u,h}^1\|_1^4+\|\widetilde{\boldsymbol{\eta}}_{u,h}^1\|_0^2\|\boldsymbol{u}(t_1)\|_2^2)
\end{aligned}
\end{align}
due to 
\begin{align*}
(\widetilde{\boldsymbol{\eta}}_{u,h}^1-\boldsymbol{\eta}_{u,h}^0, \widetilde{\boldsymbol{e}}_{u,h}^1)&\leq \frac14 \|\widetilde{\boldsymbol{e}}_{u,h}^1\|_0^2+C\|\widetilde{\boldsymbol{\eta}}_{u,h}^1-\boldsymbol{\eta}_{u,h}^0\|_0^2.
\end{align*}

Finally, we obtain:
\begin{align}
 \begin{aligned}
  &C\left(1-\frac{\Delta t}{\nu}-\Delta t\frac{(\nu+\gamma)^2h^{4k_u}+\gamma^{-2}h^{4k_p+4} + (\Delta t)^4}{\nu^5}\right)\|\widetilde{\boldsymbol{e}}_{u,h}^1\|_0^2
+\nu\Delta t\|\nabla \widetilde{\boldsymbol{e}}_{u,h}^1\|_0^2 
\\&\quad
+ \gamma\Delta t \|\nabla \cdot\widetilde{\boldsymbol{e}}_{u,h}^1\|_0^2
+2\Delta t \sum_{M\in\mathcal M_h} \tau_M^1 
\|\kappa_M((\widetilde{\boldsymbol{u}}_h^1\cdot\nabla)\widetilde{\boldsymbol{e}}_{u,h}^1)\|_{0,M}^2
\\ &\leq
C \left((\nu+\gamma)h^{2k_u+2}+\gamma^{-1}h^{2k_p+4}\right)\left(1+\frac{\Delta t}{\nu}\right)
 \\&\quad
+C\Delta t\frac{(\nu+\gamma)^2h^{4k_u}+\gamma^{-2}h^{4k_p+4} + (\Delta t)^4}{\nu^3}
\\&\quad
+C \frac{\Delta t}{\nu}
\max_{M\in\mathcal M_h}\{\tau_M^1 |\widetilde{\boldsymbol{u}}_M^1 |^2\}
\left( (\nu+\gamma)h^{2k_u}+\gamma^{-1}h^{2k_p+2} + (\Delta t)^2 +\nu h^{2s}\right).
 \end{aligned}
\end{align}

For an estimate on the gradient of the pressure we test the projection error equation with $\nabla (\xi_{p,h}^1-\xi_{p,h}^0)$ to obtain
 \begin{align}
 \begin{aligned}
 \frac{2\Delta t}{3}\|\nabla(\xi_{p,h}^1-\xi_{p,h}^0)\|_0
 &=\frac{2\Delta t}{3} \frac{\|\nabla(\xi_{p,h}^1-\xi_{p,h}^0)\|_0^2}{\|\nabla(\xi_{p,h}^1-\xi_{p,h}^0)\|_0}
 =\frac{(\boldsymbol{\xi}_{u,h}-\widetilde{\boldsymbol{\xi}}_{u,h}, \nabla(\xi_{p,h}^1-\xi_{p,h}^0))}{\|\nabla(\xi_{p,h}^1-\xi_{p,h}^0)\|_0}\\
&\leq \|\boldsymbol{\xi}_{u,h}-\widetilde{\boldsymbol{\xi}}_{u,h}\|_0\leq \|\boldsymbol{e}_{u,h}-\widetilde{\boldsymbol{e}}_{u,h}\|_0+\|\boldsymbol{\eta}_{u,h}-\widetilde{\boldsymbol{\eta}}_{u,h}\|_0\\
  \Rightarrow (\Delta t)^2\| \nabla (\xi_{p,h}^1-\xi_{p,h}^0)\|_0^2 
  &  \leq  
  \|\widetilde{\boldsymbol{e}}_{u,h}\|_0^2+\|\widetilde{\boldsymbol{\eta}}_{u,h}\|_0^2.
 \end{aligned}
\end{align}
\end{proof}

\subsubsection{Discretization Error after Initialization}

Now, we are in position to state the error bounds also for $n\geq 2$.

\begin{lemma}
For all $1\leq m\leq N$ the discretization error due to spatial discretization $\boldsymbol{e}_{u,h}^m=j_u\boldsymbol{w}_t^m-\boldsymbol{u}_h^m$ and $\widetilde{\boldsymbol{e}}_{u,h}^m=j_u\widetilde{\boldsymbol{w}}_t^m-\widetilde{\boldsymbol{u}}_h^m$ can be bounded as
\begin{align}
\begin{aligned}
& \| \widetilde{\boldsymbol{e}}_{u,h}^m\|_0^2+\|2 \boldsymbol{e}_{u,h}^{m}- 
\boldsymbol{e}_{u,h}^{m-1}\|_0^2+\frac{4}{3}(\Delta t)^2 \|\nabla \xi_{p,h}^m\|_0^2 \\
&\quad+\sum_{n=2}^m\left(\Delta t\nu \|\nabla \widetilde{\boldsymbol{e}}_{u,h}^n\|_0^2 +\Delta t\gamma \|\nabla \cdot  
\widetilde{\boldsymbol{e}}_{u,h}^n\|_0^2+\|\delta_{tt}\boldsymbol{e}_{u,h}^{n}\|_0^2
 \right.\\&\qquad\qquad\left.
+ 2\Delta t \sum_{M\in\mathcal 
M_h}\tau_M^n\|\kappa_M((\widetilde{\boldsymbol{u}}_M^n\cdot\nabla)\widetilde{\boldsymbol{e}}_{u,h}^n)\|_{0,M}^2\right)\\
&\leq 
\frac{C_{G,h}}{\Delta t} \left((\nu+\gamma)h^{2k_u+2}+\frac{h^{2k_p+4}}{\gamma}\right)\left(1+\frac{\Delta t}{\nu}\right)
\\&\quad
+\frac{C_{G,h}}{\nu\gamma}\left((\nu+\gamma)h^{2k_u}+\frac{h^{2k_p+2}}{\gamma}\right)
\\&\quad
+\frac{C_{G,h}}{\nu^3}\left((\nu+\gamma)^2h^{4k_u}+\gamma^{-2}h^{4k_p+4} + (\Delta t)^4 \right)
\\&\quad
+\frac{C_{G,h}}{\nu}
\max_{M\in\mathcal M_h}\{\tau_M^1 |\widetilde{\boldsymbol{u}}_M^1 |^2\}
\left( (\nu+\gamma)h^{2k_u}+\gamma^{-1}h^{2k_p+2} + (\Delta t)^2 +\nu h^{2s}\right).
\end{aligned}
\end{align}
The Gronwall term $C_{G,h}$ behaves like
$C_{G,h}\sim\exp\left(\frac{T}{1-K}\right)$
where 
\begin{align}
\begin{aligned}
 \label{eqn:ts_Gronwall_spatial}
 K&:=\frac14+C\frac{\Delta t}{\nu}+C\Delta t\frac{\|\widetilde{\boldsymbol{\eta}}_{u,t}^n\|_1^4+\|\widetilde{\boldsymbol{\eta}}_{u,h}^n\|_1^4}{\nu^3}
 \\&
 \leq \frac14+C\frac{\Delta t}{\nu}+C\Delta t \frac{(\nu+\gamma)^2h^{4k_u}+\gamma^{-2}h^{4k_p+4}+(\Delta t)^4}{\nu^5}
 \end{aligned}
\end{align}
and $K< 1$ is required.
\label{lem:ts_spatial}
\end{lemma}

\begin{proof}
Subtracting the convection-diffusion and projection equations for $\widetilde{\boldsymbol{w}}_t^n$ and 
$\widetilde{\boldsymbol{u}}_h^n$ from each other gives
for all $\boldsymbol{v}_h\in \boldsymbol{V}\!_h$ and $\phi_h\in Q_h$
\begin{align}
\begin{aligned}
  &\left(\frac{3 \widetilde{\boldsymbol{e}}_{u,h}^n-4\boldsymbol{e}_{u,h}^{n-1}+\boldsymbol{e}_{u,h}^{n-2}}{2 \Delta 
t}, \boldsymbol{v}_h\right) 
+\nu(\nabla \widetilde{\boldsymbol{e}}_{u,h}^n, \nabla \boldsymbol{v}_h) + \gamma (\nabla 
\cdot\widetilde{\boldsymbol{e}}_{u,h}^n, \nabla \cdot\boldsymbol{v}_h) 
\\&\quad
+ (\nabla \xi_{p,h}^{n-1}, \boldsymbol{v}_h) + 
c(\widetilde{\boldsymbol{w}}_t^n, \widetilde{\boldsymbol{w}}_t^n, \boldsymbol{v}_h)- 
c(\widetilde{\boldsymbol{u}}_h^n, \widetilde{\boldsymbol{u}}_h^n, \boldsymbol{v}_h) \
-s_h(\widetilde{\boldsymbol{u}}_h^n, \widetilde{\boldsymbol{u}}_h^n, \widetilde{\boldsymbol{u}}_h^n, \boldsymbol{v}
_h) \\
&= -\left(\frac{3 \widetilde{\boldsymbol{\eta}}_{u,h}^n-4\boldsymbol{\eta}_{u,h}^{n-1}+\boldsymbol{\eta}_{u,h}^{n-2}}{2 
\Delta t}, \boldsymbol{v}_h\right) 
-\nu(\nabla \widetilde{\boldsymbol{\eta}}_{u,h}^n, \nabla \boldsymbol{v}_h) - \gamma (\nabla 
\cdot\widetilde{\boldsymbol{\eta}}_{u,h}^n, \nabla \cdot\boldsymbol{v}_h)
\label{eqn:ts_h_nonlin1_error}
 \end{aligned}
\end{align}
and
 \begin{align}
 \begin{aligned}
  \left(\frac{3 \boldsymbol{\xi}_{u,h}^n-3\widetilde{\boldsymbol{\xi}}_{u,h}^n}{2 \Delta t}
  +\nabla (\xi_{p,h}^n-\xi_{p,h}^{n-1}), \nabla 
\phi_h \right)=0.
\label{eqn:ts_h_nonlin2_error}
 \end{aligned}
\end{align}
We first test (\ref{eqn:ts_h_nonlin1_error}) with $4\Delta t \boldsymbol{e}_{u,h}^n$ to get
\begin{align}
 \begin{aligned}
  &2\left(3 
\widetilde{\boldsymbol{e}}_{u,h}^n-4\boldsymbol{e}_{u,h}^{n-1}+\boldsymbol{e}_{u,h}^{n-2}, \widetilde{\boldsymbol{e}}_{u,
h}^n\right) 
\\&\quad
+4\Delta t\nu(\nabla \widetilde{\boldsymbol{e}}_{u,h}^n, \nabla \widetilde{\boldsymbol{e}}_{u,h}^n) 
+ 4\Delta t s_h(\widetilde{\boldsymbol{u}}_h^n, \widetilde{\boldsymbol{e}}_{u,h}^n, \widetilde{\boldsymbol{u}}_h^n, 
\widetilde{\boldsymbol{e}}_{u,h}^n)\\
&=-4\Delta t\left(c(\widetilde{\boldsymbol{w}}_t^n, \widetilde{\boldsymbol{w}}_t^n, \widetilde{\boldsymbol{e}}_{u,h}^n) 
- c(\widetilde{\boldsymbol{u}}_h^n, \widetilde{\boldsymbol{u}}_h^n, \widetilde{\boldsymbol{e}}_{u,h}^n) \right)  
 \\&\quad
+ 4\Delta t\left(s_h(\widetilde{\boldsymbol{u}}_h^n, \widetilde{\boldsymbol{u}}_h^n, \widetilde{\boldsymbol{u}}_h^n, 
\widetilde{\boldsymbol{e}}_{u,h}^n) 
+ s_h(\widetilde{\boldsymbol{u}}_h^n, \widetilde{\boldsymbol{e}}_{u,h}^n, \widetilde{\boldsymbol{u}}_h^n, 
\widetilde{\boldsymbol{e}}_{u,h}^n) \right)\\
&\quad -\left(\frac{3 \widetilde{\boldsymbol{\eta}}_{u,h}^n-4\boldsymbol{\eta}_{u,h}^{n-1}+\boldsymbol{\eta}_{u,h}^{n-2}}{2 
\Delta t}, \widetilde{\boldsymbol{e}}_h\right) 
-\nu(\nabla \widetilde{\boldsymbol{\eta}}_{u,h}^n, \nabla \widetilde{\boldsymbol{e}}_h) 
 \\&\quad
- \gamma (\nabla \cdot\widetilde{\boldsymbol{\eta}}_{u,h}^n, \nabla \cdot\widetilde{\boldsymbol{e}}_h)-4\Delta t (\nabla \xi_{p,h}^{n-1}, \widetilde{\boldsymbol{e}}_{u,h}^n)\\
&=-4\Delta t\left(c(\widetilde{\boldsymbol{w}}_t^n, \widetilde{\boldsymbol{w}}_t^n, \widetilde{\boldsymbol{e}}_{u,h}^n) 
- c(\widetilde{\boldsymbol{u}}_h^n, \widetilde{\boldsymbol{u}}_h^n, \widetilde{\boldsymbol{e}}_{u,h}^n) \right)  
 \\&\quad
+ 4\Delta t\left(s_h(\widetilde{\boldsymbol{u}}_h^n, \widetilde{\boldsymbol{u}}_h^n, \widetilde{\boldsymbol{u}}_h^n, 
\widetilde{\boldsymbol{e}}_{u,h}^n) 
+ s_h(\widetilde{\boldsymbol{u}}_h^n, \widetilde{\boldsymbol{e}}_{u,h}^n, \widetilde{\boldsymbol{u}}_h^n, 
\widetilde{\boldsymbol{e}}_{u,h}^n) \right)\\
&\quad -\left(\frac{3 \widetilde{\boldsymbol{\eta}}_{u,h}^n-4\boldsymbol{\eta}_{u,h}^{n-1}+\boldsymbol{\eta}_{u,h}^{n-2}}{2 
\Delta t}, \widetilde{\boldsymbol{e}}_h\right) -4\Delta t (\nabla e_{p,h}^{n-1}, \widetilde{\boldsymbol{e}}_{u,h}^n)
 \end{aligned}
 \end{align}
 where the last step follows from the choice for the interpolation operator.
 The last pressure term in the above equation can be bounded by
 \begin{align*}
  &(\nabla e_{p,h}^{n-1}, \widetilde{\boldsymbol{e}}_{u,h}^n)=(\nabla e_{p,h}^{n-1}, \widetilde{\boldsymbol{\xi}}_{u,h}^n)
  =\frac{2 \Delta t}{3}(\nabla(\xi_{p,h}^n-\xi_{p,h}^{n-1}), \nabla e_{p,h}^{n-1})
   \\  &
  =\frac{2 \Delta t}{3}(\nabla(\xi_{p,h}^n-\xi_{p,h}^{n-1}), \nabla (\xi_{p,h}^{n-1}-\eta_{p,h}^{n-1}))\\
  &= \frac{\Delta t}{3}(\|\nabla \xi_{p,h}^n\|_0^2-\|\nabla (\xi_{p,h}^n-\xi_{p,h}^{n-1})\|_0^2-\|\nabla \xi_{p,h}^{n-1}\|_0^2)
  -\frac{2 \Delta t}{3}
  (\nabla(\xi_{p,h}^n-\xi_{p,h}^{n-1}), \nabla \eta_{p,h}^{n-1})\\
  &= \frac{\Delta t}{3}(\|\nabla \xi_{p,h}^n\|_0^2-\|\nabla \xi_{p,h}^{n-1}\|_0^2)-\frac{3}{4 \Delta t}\|\widetilde{\boldsymbol{e}}_{u,h}^n-\boldsymbol{e}_{u,h}^n\|_0^2
  + (\nabla\cdot\widetilde{\boldsymbol{e}}_{u,h}^n, \eta_{p,h}^{n-1})\\
  &\geq \frac{\Delta t}{3}(\|\nabla \xi_{p,h}^n\|_0^2-\|\nabla \xi_{p,h}^{n-1}\|_0^2)-\frac{3}{4 \Delta t}\|\widetilde{\boldsymbol{e}}_{u,h}^n-\boldsymbol{e}_{u,h}^n\|_0^2
  -\frac{\gamma}{4} \|\nabla\cdot\widetilde{\boldsymbol{e}}_{u,h}^n\|_0^2-\frac{C}{\gamma}\|\eta_{p,h}^{n-1}\|_0^2.
 \end{align*}

A splitting (cf. (\ref{eqn:splitting})) of the time derivative term gives
\begin{align*}
&(2(3 \widetilde{\boldsymbol{e}}_{u,h}^n - 4 \boldsymbol{e}_{u,h}^{n-1}+ \boldsymbol{e}_{u,h}^{n-2}), 
\widetilde{\boldsymbol{e}}_{u,h}^n)=I_1+I_2+I_3\\
&= 3\| \widetilde{\boldsymbol{e}}_{u,h}^n\|_0^2+3\| \boldsymbol{e}_{u,h}^{n}- \widetilde{
\boldsymbol{e}}^n \|_0^2 - 2\| \boldsymbol{e}_{u,h}^{n}\|_0^2 +\|2 \boldsymbol{e}_{u,h}^{n}- 
\boldsymbol{e}_{u,h}^{n-1}\|_0^2
 \\&\quad
+\|\delta_{tt}\boldsymbol{e}_{u,h}^{n}\|_0^2-\| \boldsymbol{e}
^{n-1}\|_0^2-\|2 \boldsymbol{e}_{u,h}^{n-1}- 
\boldsymbol{e}_{u,h}^{n-2}\|_0^2.
\end{align*}
The second term $I_2$ vanishes due to the fact that $\boldsymbol{e}_{u,h}^{n}$ is weakly divergence-free. 

With the identity $(a-b, b)=\frac{1}{2}(\|a\|_0^2-\|a-b\|_0^2-\|b\|_0^2)$ and  $\|\boldsymbol{e}_{u,h}^{n}\|\stackrel{(\ref{rem:l2proj})}{\leq} \| 
\widetilde{\boldsymbol{e}}_{u,h}^n \| $ we have so far
\begin{align}
 \begin{aligned}
& \| \widetilde{\boldsymbol{e}}_{u,h}^n\|_0^2 +\|2 \boldsymbol{e}_{u,h}^{n}- 
\boldsymbol{e}_{u,h}^{n-1}\|_0^2+\frac{4}{3}(\Delta t)^2 \|\nabla \xi_{p,h}^n\|_0^2 
+4 \Delta t\nu \|\nabla \widetilde{\boldsymbol{e}}_{u,h}^n\|_0^2 
\\&\quad
 + 3\Delta t\gamma \|\nabla \cdot  
\widetilde{\boldsymbol{e}}_{u,h}^n\|_0^2+\|\delta_{tt}\boldsymbol{e}_{u,h}^{n}\|_0^2 
+ 4\Delta t \|\kappa_M((\widetilde{\boldsymbol{u}}_h^n\cdot\nabla)\widetilde{\boldsymbol{e}}_{u,h}^n)\|_0^2\\
&\leq \| \widetilde{\boldsymbol{e}}_{u,h}^{n-1}\|_0^2 +\|2 \boldsymbol{e}_{u,h}^{n-1} - 
\boldsymbol{e}_{u,h}^{n-2}\|_0^2 + \frac{4}{3}(\Delta t)^2\| \nabla \xi_{p,h}^{n-1}\|_0^2  
 \\&\quad
-4\Delta t\left(c(\widetilde{\boldsymbol{w}}_t^n, \widetilde{\boldsymbol{w}}_t^n, \widetilde{\boldsymbol{e}}_{u,h}^n) - 
c(\widetilde{\boldsymbol{u}}_h^n, \widetilde{\boldsymbol{u}}_h^n, \widetilde{\boldsymbol{e}}_{u,h}^n) \right)  
\\&\quad
+ 4\Delta 
t\left(s_h(\widetilde{\boldsymbol{u}}_h^n, \widetilde{\boldsymbol{w}}_t^n, \widetilde{\boldsymbol{u}}_h^n, \widetilde
{\boldsymbol{e}}_{u,h}^n) 
- s_h(\widetilde{\boldsymbol{u}}_h^n, \widetilde{\boldsymbol{\eta}}_{u,h}^n, \widetilde{\boldsymbol{u}}_h^n, 
\widetilde{\boldsymbol{e}}_{u,h}^n) \right)
\\&\quad
-\left(\frac{3 \widetilde{\boldsymbol{\eta}}_{u,h}^n-4\boldsymbol{\eta}_{u,h}^{n-1}+\boldsymbol{\eta}_{u,h}^{n-2}}{2 
\Delta t}, \widetilde{\boldsymbol{e}}_{u,h}\right) +
\frac{C}{\gamma}\Delta t\|\eta_{p,h}^{n-1}\|_0^2.
\end{aligned}
\end{align}

For the terms on the right-hand side containing approximations errors we use the estimate
\begin{align*}
(3\widetilde{\boldsymbol{\eta}}_{u,h}^n-4\boldsymbol{\eta}_{u,h}^{n-1}+\boldsymbol{\eta}_{u,h}^{n-2}, \widetilde{
\boldsymbol{e}}_{u,h}^n)
  &\leq C (
\|\widetilde{\boldsymbol{\eta}}_{u,h}^n\|_0^2+\|\boldsymbol{\eta}_{u,h}^{n-1}\|_0^2+\|\boldsymbol{\eta}_{u,h}^{n-2}\|_0^2)
+\frac14\|\widetilde{\boldsymbol{e}}_{u,h}^n\|_0^2.
\end{align*}

From Lemma~\ref{lem:convective2}, the error with respect to the convective terms is given by
\begin{align*}
 &c(\boldsymbol{u}(t_n), \boldsymbol{u}(t_n), \widetilde{\boldsymbol{e}}_{u,h}^n)- 
c(\widetilde{\boldsymbol{u}}_h^n, \widetilde{\boldsymbol{u}}_h^n, \widetilde{\boldsymbol{e}}_{u,h}^n) \\
&\leq 
C\|\widetilde{\boldsymbol{e}}_{u,h}^n\|_0^2\left(\frac{\|\widetilde{\boldsymbol{\eta}}_{u,t}^n\|_1^4+\|\widetilde{\boldsymbol{\eta}}_{u,h}^n\|_1^4}{\nu^3}+\frac{\|\boldsymbol{u}(t_n)\|_2^2}{\nu}\right)+\frac{\nu}{4}\|\widetilde{\boldsymbol{e}}_{u,h}^n\|_1^2 
 \\&\quad
+\frac{C}{\nu}(\|\widetilde{\boldsymbol{\eta}}_{u,t}^n\|_1^4+\|\widetilde{\boldsymbol{\eta}}_{u,h}^n\|_1^4+\|\widetilde{\boldsymbol{\eta}}_{u,h}^n\|_0^2\|\boldsymbol{u}(t_n)\|_2^2).
\end{align*}

For the nonlinear stabilization we again use $\widetilde{\boldsymbol{w}}_h^n= \boldsymbol{u}(t_n) - 
(\widetilde{\boldsymbol{\eta}}_{u,t}^n+\widetilde{\boldsymbol{\eta}}_{u,h}^n)$ and obtain
\begin{align*}
&s_h(\widetilde{\boldsymbol{u}}_h^n, \widetilde{\boldsymbol{w}}_h^n, \widetilde{\boldsymbol{u}}_h^n, \widetilde{
\boldsymbol{e}}_{u,h}^n)
= 
s_h(\widetilde{\boldsymbol{u}}_h^n, \boldsymbol{u}(t_n), \widetilde{\boldsymbol{u}}_h^n, \widetilde{\boldsymbol{e}}_{
u, h}^n) 
- 
s_h(\widetilde{\boldsymbol{u}}_h^n, \widetilde{\boldsymbol{\eta}}_{u,t}^n+\widetilde{\boldsymbol{\eta}}_{u,h}^n, \widetilde{\boldsymbol{u}}_h^n, \widetilde{
\boldsymbol{e}}_{u,h}^n)\\
&\leq \sum_{M\in\mathcal M_h} \tau_M^n 
\left(C|\widetilde{\boldsymbol{u}}_M^n|^2 
\|\kappa_M(\nabla\boldsymbol{u}(t_n))\|_{0,M}^2
+ C |\widetilde{\boldsymbol{u}}_M^n|^2 
\|\nabla(\widetilde{\boldsymbol{\eta}}_{u,t}^n+\widetilde{\boldsymbol{\eta}}_{u,h}^n)\|_{0,M}^2
\right. \\&\qquad\qquad\qquad \left.
+\frac14\|\kappa_M((\widetilde{\boldsymbol{u}}_M^n\cdot\nabla)\widetilde{
\boldsymbol{e}}_{u,h}^n)\|_{0,M}^2\right)\\
&\leq C \max_{M\in\mathcal M_h} \{ \tau_M^n |\widetilde{\boldsymbol{u}}_M^n|^2\} 
\left(\|\nabla(\widetilde{\boldsymbol{\eta}}_{u,t}^n+\widetilde{\boldsymbol{\eta}}_{u,h}^n)\|_0^2 + \sum_{M\in\mathcal 
M_h}\|\kappa_M(\nabla\boldsymbol{u}(t_n))\|_{0,M}^2 \right)
\\&\quad
+\frac14 \sum_{M\in\mathcal M_h} 
\tau_M^n\|\kappa_M((\widetilde{\boldsymbol{u}}_M^n\cdot\nabla)\widetilde{\boldsymbol{e}}_{u,h}^n)\|_{0,M}^2.
\end{align*}

Now, we collect all the estimates and sum the resulting inequality from $n=2$ to 
$m\leq N$: 
\begin{align*}
& \| \widetilde{\boldsymbol{e}}_{u,h}^m\|_0^2+\|2 \boldsymbol{e}_{u,h}^{m}- 
\boldsymbol{e}_{u,h}^{m-1}\|_0^2+\frac{4}{3}(\Delta t)^2 \|\nabla \xi_{p,h}^m\|_0^2 \\
&\quad+\sum_{n=2}^m\left(\Delta t \nu \|\nabla \widetilde{\boldsymbol{e}}_{u,h}^n\|_0^2 + \Delta t\gamma \|\nabla 
\cdot  \widetilde{\boldsymbol{e}}_{u,h}^n\|_0^2+\|\delta_{tt}\boldsymbol{e}_{u,h}^{n}\|_0^2
 \right.\\&\qquad\qquad\left.
+ 2\Delta t \sum_{M\in\mathcal 
M_h}\tau_M^n\|\kappa_M((\widetilde{\boldsymbol{u}}_M^n\cdot\nabla)\widetilde{\boldsymbol{e}}_{u,h}^n)\|_{0,M}^2\right)\\
&\leq \| \widetilde{\boldsymbol{e}}_{u,h}^1\|_0^2 +\|2 \boldsymbol{e}_{u,h}^1 - 
\boldsymbol{e}_{u,h}^0\|_0^2  + \frac{4}{3}(\Delta t)^2\|\nabla \xi_{p,h}^1\|_0^2 
\\&\quad 
+\sum_{n=2}^m\left\{ \left(\frac14+C\frac{\Delta t}{\nu}+C\Delta t\frac{\|\widetilde{\boldsymbol{\eta}}_{u,t}^n\|_1^4+\|\widetilde{\boldsymbol{\eta}}_{u,h}^n\|_1^4}{\nu^3} 
\right)\|\widetilde{\boldsymbol{e}}_{u,h}^n\|_0^2
 \right.\\&\qquad\qquad
+C \frac{\Delta t}{\nu}
(\|\widetilde{\boldsymbol{\eta}}_{u,t}^n\|_1^4+\|\widetilde{\boldsymbol{\eta}}_{u,h}^n\|_1^4)
\\&\qquad\qquad + C \Delta t  \max_{M\in\mathcal M_h} \{ \tau_M^n |\widetilde{\boldsymbol{u}}_M^n|^2\} \sum_{M\in\mathcal 
M_h}\|\kappa_M(\nabla\boldsymbol{u}(t_n))\|_{0,M}^2
 \\&\qquad\qquad
+ C \Delta t\max_{M\in\mathcal M_h}\{\tau_M^n |\widetilde{\boldsymbol{u}}_M^n|^2\} 
(\|\widetilde{\boldsymbol{\eta}}_{u,h}^n\|_1^2+\|\widetilde{\boldsymbol{\eta}}_{u,t}^n\|_1^2)\\
&\qquad\qquad+\left. C\left(1+\frac{\Delta t}{\nu}\right) 
\|\widetilde{\boldsymbol{\eta}}_{u,h}^n\|_0^2+C\|\boldsymbol{\eta}_{u,h}^{n-1}\|_0^2+C\|\boldsymbol{\eta}_{u,h}^{n-2}
\|_0^2+\frac{C}{\gamma}\Delta t\|\eta_{p,h}^{n-1}\|_0^2\right\}.
\end{align*}

Provided $C\Delta t \left(1/\nu+(\|\widetilde{\boldsymbol{\eta}}_{u,t}^n\|_1^4+\|\widetilde{\boldsymbol{\eta}}_{u,h}^n\|_1^4)/\nu^3 
\right)<1$ we can use the discrete Gronwall Lemma~\ref{lem:Gronwall_discrete} for 
\begin{align*}
\left(\frac14+C\frac{\Delta t}{\nu}+C\Delta t\frac{\|\widetilde{\boldsymbol{\eta}}_{u,t}^n\|_1^4+\|\widetilde{\boldsymbol{\eta}}_{u,h}^n\|_1^4}{\nu^3}\right)\|\widetilde{\boldsymbol{e}}_{u,h}^n\|_0^2
\end{align*}
and arrive at
\begin{align*}
& \| \widetilde{\boldsymbol{e}}_{u,h}^m\|_0^2+\|2 \boldsymbol{e}_{u,h}^{m}- 
\boldsymbol{e}_{u,h}^{m-1}\|_0^2+\frac{4}{3}(\Delta t)^2 \|\nabla \xi_{p,h}^m\|_0^2 
\\&\quad
+\sum_{n=2}^m\left(\Delta t\nu \|\nabla \widetilde{\boldsymbol{e}}_{u,h}^n\|_0^2 +\Delta t\gamma \|\nabla \cdot  
\widetilde{\boldsymbol{e}}_{u,h}^n\|_0^2+\|\delta_{tt}\boldsymbol{e}_{u,h}^{n}\|_0^2
 \right.\\&\qquad\qquad\left.
+ 2\Delta t \sum_{M\in\mathcal 
M_h}\tau_M^n\|\kappa_M((\widetilde{\boldsymbol{u}}_M^n\cdot\nabla)\widetilde{\boldsymbol{e}}_{u,h}^n)\|_{0,M}^2\right)\\
&\leq C_{G,h}\left(\| \widetilde{\boldsymbol{e}}_{u,h}^1\|_0^2 +\|2 \boldsymbol{e}_{u,h}^1 - 
\boldsymbol{e}_{u,h}^0\|_0^2  + \frac{4}{3}(\Delta t)^2\|\nabla \xi_{p,h}^1\|_0^2 \right)
\\&\quad 
+C_{G,h}\sum_{n=2}^m\left\{
 \frac{\Delta t}{\nu}
(\|\widetilde{\boldsymbol{\eta}}_{u,t}^n\|_1^4+\|\widetilde{\boldsymbol{\eta}}_{u,h}^n\|_1^4)
\right.\\&\qquad\qquad\qquad
+  \Delta t  \max_{M\in\mathcal M_h} \{ \tau_M^n |\widetilde{\boldsymbol{u}}_M^n|^2\} \sum_{M\in\mathcal 
M_h}\|\kappa_M(\nabla\boldsymbol{u}(t_n))\|_{0,M}^2
\\&\qquad\qquad\qquad
+  \Delta t\max_{M\in\mathcal M_h}\{\tau_M^n |\widetilde{\boldsymbol{u}}_M^n|^2\} 
(\|\widetilde{\boldsymbol{\eta}}_{u,h}^n\|_1^2+\|\widetilde{\boldsymbol{\eta}}_{u,t}^n\|_1^2)\\
&\qquad\qquad\qquad+\left. \left(1+\frac{\Delta t}{\nu}\right) 
\|\widetilde{\boldsymbol{\eta}}_{u,h}^n\|_0^2+\|\boldsymbol{\eta}_{u,h}^{n-1}\|_0^2+\|\boldsymbol{\eta}_{u,h}^{n-2}
\|_0^2+\frac{\Delta t}{\gamma}\|\eta_{p,h}^{n-1}\|_0^2\right\}.
\end{align*}
where $C_{G,h}$ is defined as in (\ref{eqn:ts_Gronwall_spatial}).

Due to the initial error estimates 
we finally obtain an estimate for the velocity terms:\\
For all $1\leq m\leq N$ the discretization error due to spatial discretization can be bounded as
\begin{align*}
& \| \widetilde{\boldsymbol{e}}_{u,h}^m\|_0^2+\|2 \boldsymbol{e}_{u,h}^{m}- 
\boldsymbol{e}_{u,h}^{m-1}\|_0^2+\frac{4}{3}(\Delta t)^2 \|\nabla \xi_{p,h}^m\|_0^2 \\
&\quad+\sum_{n=2}^m\left(\Delta t\nu \|\nabla \widetilde{\boldsymbol{e}}_{u,h}^n\|_0^2 +\Delta t\gamma \|\nabla \cdot  
\widetilde{\boldsymbol{e}}_{u,h}^n\|_0^2+\|\delta_{tt}\boldsymbol{e}_{u,h}^{n}\|_0^2
 \right.\\&\qquad\qquad\left.
+ 2\Delta t \sum_{M\in\mathcal 
M_h}\tau_M^n\|\kappa_M((\widetilde{\boldsymbol{u}}_M^n\cdot\nabla)\widetilde{\boldsymbol{e}}_{u,h}^n)\|_{0,M}^2\right)\\
&\leq 
\frac{C_{G,h}}{\Delta t} \left((\nu+\gamma)h^{2k_u+2}+\gamma^{-1}h^{2k_p+4}\right)\left(1+\frac{\Delta t}{\nu}\right)
\\&\quad
+\frac{C_{G,h}}{\nu\gamma}\left((\nu+\gamma)h^{2k_u}+\gamma^{-1}h^{2k_p+2}\right)
\\&\quad
+\frac{C_{G,h}}{\nu^3}\left((\nu+\gamma)^2h^{4k_u}+\gamma^{-2}h^{4k_p+4} + (\Delta t)^4 \right)
\\&\quad
+\frac{C_{G,h}}{\nu}
\max_{M\in\mathcal M_h}\{\tau_M^1 |\widetilde{\boldsymbol{u}}_M^1 |^2\}
\left( (\nu+\gamma)h^{2k_u}+\gamma^{-1}h^{2k_p+2} + (\Delta t)^2 +\nu h^{2s}\right).
\end{align*}
\end{proof}

\begin{remark}
 Although these estimate are not quasi-robust with respect to $\nu$, the usual scaling of the Gronwall constant $C_{G,h}\sim T/\nu^3$ is improved to essentially $C_{G,h}\sim T/\nu$.
\end{remark}

\subsection{Combined Error Estimates}
We considered both the error due to the temporal discretization of the continuous problem and the spatial discretization of the 
semi-discretized problem. All that is left is to combine the error results.

\begin{theorem}
\label{thm:ts_full_errors_old}
Provided the intermediate solutions are sufficiently smooth, 
for all $1\leq m\leq N$ the total error due to spatial discretization and discretization in time
\begin{align*}
\widetilde{\boldsymbol{\zeta}}_u^n&=\widetilde{\boldsymbol{\eta}}_{u,t}^n+\widetilde{\boldsymbol{\xi}}_{u,h}^n=\boldsymbol{u}(t_n)-\boldsymbol{u}_h^n, &
\boldsymbol{\zeta}_u^n&=\boldsymbol{\eta}_{u,t}^n+\boldsymbol{\xi}_{u,h}^n=\boldsymbol{u}(t_n)-\boldsymbol{u}_h^n, \\
\zeta_p^n&=\eta_{p,t}^n+\xi_{p,h}^n=p(t_n)-p_h^n
\end{align*}
can be bounded as
 \begin{align}
  \nonumber
 \|\boldsymbol{\zeta}_u^m\|_0^2  
  &\leq \frac{C_{G,h}}{\Delta t} \Big((\nu+\gamma)h^{2k_u+2}+\gamma^{-1}h^{2k_p+4} + (\Delta t)^2 h^2 \Big(\frac{\nu+\gamma}{\nu}+\frac{1}{\nu\gamma}\Big)\Big)\Big(1+\frac{\Delta t}{\nu}\Big)
\\&\quad \label{eqn:ts_full_error_old}
+\frac{C_{G,h}}{\nu^3}\left((\nu+\gamma)^2h^{4k_u}+\gamma^{-2}h^{4k_p+4} + (\Delta t)^4 \right)
\\&\quad \nonumber
+\frac{C_{G,h}}{\nu}
\max_{M\in\mathcal M_h}\{\tau_M^1 |\widetilde{\boldsymbol{u}}_M^1 |^2\}
\left( (\nu+\gamma)h^{2k_u}+\gamma^{-1}h^{2k_p+2} + (\Delta t)^2 +\nu h^{2s}\right)
\\&\quad+C(\Delta t)^4
\nonumber
 \end{align}
 \begin{align}
 \nonumber
  &\Delta t \sum_{n=1}^m\left(\nu \|\nabla \boldsymbol{\zeta}_u^n\|_0^2 + \gamma \|\nabla \cdot  
\boldsymbol{\zeta}_u^n\|_0^2
+ 
\sum_{M\in\mathcal 
M_h}\tau_M^n\|\kappa_M((\widetilde{\boldsymbol{u}}_h^n\cdot\nabla)\boldsymbol{\zeta}_u^n)\|_{0,M}^2
\right)\\\nonumber
  &\leq \frac{C_{G,h}}{\Delta t} \Big((\nu+\gamma)h^{2k_u+2}+\gamma^{-1}h^{2k_p+4} + (\Delta t)^2 h^2 \Big(\frac{\nu+\gamma}{\nu}+\frac{1}{\nu\gamma}\Big)\Big)\Big(1+\frac{\Delta t}{\nu}\Big)
\\&\quad
+\frac{C_{G,h}}{\nu^3}\left((\nu+\gamma)^2h^{4k_u}+\gamma^{-2}h^{4k_p+4} + (\Delta t)^4 \right)
\\&\quad\nonumber
+\frac{C_{G,h}}{\nu}
\max_{M\in\mathcal M_h}\{\tau_M^1 |\widetilde{\boldsymbol{u}}_M^1 |^2\}
\left( (\nu+\gamma)h^{2k_u}+\gamma^{-1}h^{2k_p+2} + (\Delta t)^2 +\nu h^{2s}\right)
\\&\quad+C(\Delta t)^2  
\nonumber
 \end{align}
with the same Gronwall term $C_{G,h}$ as in Lemma~\ref{lem:ts_spatial} provided $K < 1$.
\end{theorem}
\begin{proof}
Summing up interpolation and discretization errors gives the estimate for all considered error norms apart from the nonlinear stabilization.
All that is left is an estimate due to time discretization for this error:
\begin{align*}
&\Delta t \sum_{n=1}^m\sum_{M\in\mathcal 
M_h}\tau_M^n\|\kappa_M((\widetilde{\boldsymbol{u}}_h^n\cdot\nabla)\boldsymbol{\zeta}_u^n)\|_{0,M}^2\\
&\leq
C \Delta t \sum_{n=1}^m\sum_{M\in\mathcal M_h}\!\!\!\! \left( 
\tau_M^n\|\kappa_M((\widetilde{\boldsymbol{u}}_h^n\cdot\nabla)\widetilde{\boldsymbol{\xi}}_{u,h}^n)\|_{0,M}^2
+ \tau_M^n\|\kappa_M((\widetilde{\boldsymbol{u}}_h^n\cdot\nabla)\widetilde{\boldsymbol{\eta}}_{u,t}^n)\|_{0,M}^2 \right)\\
&\leq
C \Delta t \sum_{n=1}^m\sum_{M\in\mathcal M_h}\!\!\!\! \left( 
\tau_M^n\|\kappa_M((\widetilde{\boldsymbol{u}}_h^n\cdot\nabla)\widetilde{\boldsymbol{\xi}}_{u,h}^n)\|_{0,M}^2
+  \tau_M^n|\widetilde{\boldsymbol{u}}_M^n|^2\|\nabla\widetilde{\boldsymbol{\eta}}_{u,t}^n\|_{0,M}^2\right) \\
&\leq 
C \Delta t \sum_{n=1}^m\sum_{M\in\mathcal M_h}\!\!\!\! \tau_M^n( 
\|\kappa_M((\widetilde{\boldsymbol{u}}_h^n\cdot\nabla)\widetilde{\boldsymbol{e}}_{u,h}^n)\|_{0,M}^2
+  |\widetilde{\boldsymbol{u}}_M^n|^2(\|\nabla\widetilde{\boldsymbol{\eta}}_{u,t}^n\|_{0,M}^2+\|\nabla\widetilde{\boldsymbol{\eta}}_{u,h}^n\|_{0,M}^2)).
\end{align*} 
The first term is part of the left side of the discretization error estimate and the second term part of the right-hand side of the discretization error estimate. This gives the claim also for the nonlinear stabilization.
\end{proof}

\section{Quasi-Optimal Error Estimates}\label{sec:optimal}

We can now combine the results from Section~\ref{sec:full1} and Section~\ref{sec:temporal-spatial} to obtain quasi-optimal error estimates for all considered errors norms. In contrast to the previous sections, we also derive estimates on the pressure error. 
\begin{theorem}
\label{thm:ts_full_errors}
For all $1\leq m\leq N$ the total pressure error due to spatial discretization and discretization in time
$\zeta_p^n=p(t_n)-p_h^n$ 
 can be bounded as
 
 \begin{align*}
\Delta t\sum_{n=1}^N \|\zeta_p^{n-1}\|_0^2
&\leq C\left(\frac{1}{(\Delta t)^2}+\|\boldsymbol{u}\|_{l^2(t_0,T;[H^2(\Omega)]^d)}^2\right)\|\widetilde{\boldsymbol{\zeta}}_u\|_{l^\infty(t_0,T;[L^2(\Omega)]^d)}^2
 \\&\quad
+C\left(\nu+\gamma+\max_{1\leq n\leq N}\max_{M\in\mathcal M_h}\{\tau_M^n|\widetilde{\boldsymbol{u}}_M^n|^2\}\right)\|\widetilde{\boldsymbol{\zeta}}_u\|_{l^2(t_0,T;LPS)}^2
\\&\quad
+C \max_{1\leq n\leq N}\max_{M\in\mathcal M_h}\{\tau_M^n |\widetilde{\boldsymbol{u}}_M^n|^2 \}^2 h^{2k_u}\|\boldsymbol{u}\|_{l^2(t_0,T;[W^{k_u+1,2}(\Omega)]^d)}^2
 \\&\quad 
+C\frac{\|\widetilde{\boldsymbol{\zeta}}_u\|_{l^2(t_0,T;LPS)}^4}{\nu^2\Delta t}
+ C (\Delta t)^2.
\end{align*} 
\end{theorem}
\begin{proof}
In order to obtain the estimate for the pressure error in the $L^2(\Omega)$ norm, we utilize the discrete inf-sup 
stability of the ansatz spaces, i.e.
\begin{align}
\exists \boldsymbol{w}_h\in \boldsymbol{V}\!_h\colon~\|\nabla\boldsymbol{w}_h\|_0\leq \| \zeta_p^n \|_0/ \beta, \quad 
-(\nabla\cdot\boldsymbol{w}_h,\zeta_p^n)=\|\zeta_p^n\|_0^2.
\end{align}

We test the advection-diffusion error equation with $\boldsymbol{w}_h$:
\begin{align}
 \begin{aligned}
  &\left(\frac{3 \widetilde{\boldsymbol{\zeta}}_u^n-4\boldsymbol{\zeta}_u^{n-1}+\boldsymbol{\zeta}_u^{n-2}}{2\Delta 
t}, \boldsymbol{w}_h \right) +\nu(\nabla \widetilde{\boldsymbol{\zeta}}_u^n, \nabla \boldsymbol{w}_h) + \gamma (\nabla 
\cdot  \widetilde{\boldsymbol{\zeta}}_u^n, \nabla \cdot \boldsymbol{w}_h) 
\\&=-c(\boldsymbol{u}(t_n), \boldsymbol{u}(t_n), \boldsymbol{w}_h) + 
c(\widetilde{\boldsymbol{u}}_h^n, \widetilde{\boldsymbol{u}}_h^n, \boldsymbol{w}_h) 
+s_h(\widetilde{\boldsymbol{u}}_h^n, \widetilde{\boldsymbol{u}}_h^n, \widetilde{\boldsymbol{u}}_h^n, \boldsymbol{w
}_h)  
\\&\quad
+ (D_t\boldsymbol{u}(t_n)-\partial_t \boldsymbol{u}(t_n), \boldsymbol{w}_h)
- (\nabla (p(t_{n})-p_h^{n-1}), \boldsymbol{w}_h).
 \end{aligned}
 \end{align}
where $D_t\boldsymbol{u}(t_n):= (3 \boldsymbol{u}(t_n)-4\boldsymbol{u}(t_{n-1})+\boldsymbol{u}(t_{n-2})) /(2\Delta t)$ 
and $\partial_t \boldsymbol{u}$ is the time derivative of $\boldsymbol{u}$.\\
Noticing $\|\boldsymbol f\|_{-1}\leq\|\boldsymbol f\|_0$ we obtain
\begin{align*}
&\|\nabla \boldsymbol{w}_h\|_0 \|\zeta_p^{n-1}\|_0
\\&\leq\frac{1}{\beta}\|\zeta_p^{n-1}\|_0^2=-(\nabla 
\zeta_p^{n-1}, \boldsymbol{w}_h) \\
&\leq \left\|\frac{3 \widetilde{\boldsymbol{\zeta}}_u^n-4\boldsymbol{\zeta}_u^{n-1}+\boldsymbol{\zeta}_u^{n-2}}{2\Delta 
t}\right\|_{-1}\|\nabla \boldsymbol{w}_h\|_0
+\nu\|\nabla \widetilde{\boldsymbol{\zeta}}_u^n\|_0 \|\nabla \boldsymbol{w}_h\|_0
+ \gamma \|\nabla \cdot  \widetilde{\boldsymbol{\zeta}}_u^n\|_0 \|\nabla \cdot \boldsymbol{w}_h\|_0\\
&\quad+c(\boldsymbol{u}(t_n), \boldsymbol{u}(t_n), \boldsymbol{w}_h) - 
c(\widetilde{\boldsymbol{u}}_h^n, \widetilde{\boldsymbol{u}}_h^n, \boldsymbol{w}_h) 
+s_h(\widetilde{\boldsymbol{u}}_h^n, \widetilde{\boldsymbol{u}}_h^n, \widetilde{\boldsymbol{u}}_h^n, \boldsymbol{w
}_h) \\
&\quad+\| D_t\boldsymbol{u}(t_n)-\partial_t \boldsymbol{u}(t_n)\|_{-1}\|\nabla \boldsymbol{w}_h\|_0
+\|p(t_{n})-p(t_{n-1})\|_0 \|\nabla\cdot \boldsymbol{w}_h\|_0.
\end{align*}
Using Lemma~\ref{lemma:convectivebound1} we bound the convective terms according to
\begin{align*}
&|c(\boldsymbol{u}(t_n), \boldsymbol{u}(t_n), \boldsymbol{w}_h)- 
c(\widetilde{\boldsymbol{u}}_h^n, \widetilde{\boldsymbol{u}}_h^n, \boldsymbol{w}_h) |\\
&=|c(\widetilde{\boldsymbol{\zeta}}_u^n, \boldsymbol{u}(t_n), \boldsymbol{w}_h)- 
c(\widetilde{\boldsymbol{u}}_h^n, \widetilde{\boldsymbol{\zeta}}_u^n, \boldsymbol{w}_h)| \\
&= |c(\widetilde{\boldsymbol{\zeta}}_u^n, \boldsymbol{u}(t_n), \boldsymbol{w}_h) 
- c(\boldsymbol{u}(t_n), \widetilde{\boldsymbol{\zeta}}_u^n, \boldsymbol{w}_h)
-c(\widetilde{\boldsymbol{\zeta}}_u^n, \widetilde{\boldsymbol{\zeta}}_u^n, \boldsymbol{w}_h)|\\
& \leq C \|\widetilde{\boldsymbol{\zeta}}_u^n\|_0 \|\boldsymbol{u}(t_n)\|_2 \|\boldsymbol{w}_h \|_1
+ C \|\widetilde{\boldsymbol{\zeta}}_u^n\|_1^2 \|\boldsymbol{w}_h \|_1.
\end{align*}
For the nonlinear stabilization we get 
\begin{align*}
&s_h(\widetilde{\boldsymbol{u}}_h^n, \widetilde{\boldsymbol{u}}_h^n, \widetilde{\boldsymbol{u}}_h^n, \boldsymbol{w
}_h)
=s_h(\widetilde{\boldsymbol{u}}_h^n, 
\boldsymbol{u}(t_n)-\widetilde{\boldsymbol{\zeta}}_u^n, \widetilde{\boldsymbol{u}}_h^n, \boldsymbol{w}_h)\\
&\leq C \sum_{M\in\mathcal M_h} \tau_M^n |\widetilde{\boldsymbol{u}}_M^n|^2 \|\kappa_M(\nabla\boldsymbol{u}(t_n))\|_{0,M}\| 
\|\boldsymbol{w}_h\|_{1,M}
 \\&\quad
+ C \sum_{M\in\mathcal M_h} 
\tau_M^n\|\kappa_M((\widetilde{\boldsymbol{u}}_M^n\cdot\nabla)\widetilde{\boldsymbol{\zeta}}_u^n)\|_{0,M} 
|\widetilde{\boldsymbol{u}}_M^n| \|\boldsymbol{w}_h\|_{1,M}\\
&\leq C 
\left( \max_{M\in\mathcal M_h}\{\tau_M^n |\widetilde{\boldsymbol{u}}_M^n|^2 \|\kappa_M(\nabla\boldsymbol{u}(t_n))\|_{0,M} 
\}
 \right.\\&\left.\quad \quad \quad
+\sum_{M\in\mathcal M_h} \tau_M^n |\widetilde{\boldsymbol{u}}_M^n| 
\|\kappa_M((\widetilde{\boldsymbol{u}}_M^n\cdot\nabla)\widetilde{\boldsymbol{\zeta}}_u^n)\|_{0,M} \right)
\|\nabla \boldsymbol{w}_h\|_0.
\end{align*}
using the Cauchy-Schwarz inequality and Young's inequality.

We combine these results and obtain due to the approximation property of $\kappa_M$ and the estimates for 
$\|\widetilde{\boldsymbol{\zeta}}_u\|_{l^\infty(t_0,T;[L^2(\Omega)]^d)}$ and 
$\|\widetilde{\boldsymbol{\zeta}}_u\|_{l^2(t_0,T;LPS)}$:
\begin{align*}
&\Delta t\sum_{n=1}^N \|\zeta_p^{n-1}\|_0^2\\
&\leq C\Bigg\{ \frac{1}{(\Delta t)^2} \|\widetilde{\boldsymbol{\zeta}}_u\|_{l^\infty(t_0,T;[L^2(\Omega)]^d)}^2
+ \nu^2 \|\nabla \widetilde{\boldsymbol{\zeta}}_u\|_{l^2(t_0,T;[L^2(\Omega)]^d)}^2 \\
&\qquad\quad + \gamma^2 \|\nabla \cdot \widetilde{\boldsymbol{\zeta}}_u\|_{l^2(t_0,T;[L^2(\Omega)]^d)}^2
 + \|\widetilde{\boldsymbol{\zeta}}_u^n\|_{l^\infty(t_0,T;[L^2(\Omega)]^d)}^2 \|\boldsymbol{u}\|_{l^2(t_0,T;[H^2(\Omega)]^d)}^2 
\\
&\qquad\quad +  \| \widetilde{\boldsymbol{\zeta}}_u^n\|_{l^\infty(t_0,T;[H^1(\Omega)]^d)}^2 \| 
\widetilde{\boldsymbol{\zeta}}_u^n\|_{l^2(t_0,T;[H^1(\Omega)]^d)}^2 + (\Delta t)^2\\
&\qquad\quad + \max_{1\leq n\leq N}\max_{M\in\mathcal M_h}\{\tau_M^n |\widetilde{\boldsymbol{u}}_M^n|^2 \}^2 
h^{2k_u}\|\boldsymbol{u}\|_{l^2(t_0,T;[W^{k_u+1,2}(\Omega)]^d)}^2\\
&\qquad\quad +\max_{1\leq n\leq N}\max_{M\in\mathcal M_h}\{\tau_M^n\} \Delta t\sum_{n=1}^N \sum_{M\in\mathcal M_h} \tau_M^n 
|\widetilde{\boldsymbol{u}}_M^n|^2 
\|\kappa_M((\widetilde{\boldsymbol{u}}_M^n\cdot\nabla)\widetilde{\boldsymbol{\zeta}}_u^n)\|_{0,M}^2 \Bigg\}
\\&\leq C\left(\frac{1}{(\Delta t)^2}+\|\boldsymbol{u}\|_{l^2(t_0,T;[H^2(\Omega)]^d)}^2\right)\|\widetilde{\boldsymbol{\zeta}}_u\|_{l^\infty(t_0,T;[L^2(\Omega)]^d)}^2
 \\&\quad
+C\left(\nu+\gamma+\max_{1\leq n\leq N}\max_{M\in\mathcal M_h}\{\tau_M^n|\widetilde{\boldsymbol{u}}_M^n|^2\}\right)\|\widetilde{\boldsymbol{\zeta}}_u\|_{l^2(t_0,T;LPS)}^2
\\&\quad
+C \max_{1\leq n\leq N}\max_{M\in\mathcal M_h}\{\tau_M^n |\widetilde{\boldsymbol{u}}_M^n|^2 \}^2 h^{2k_u}\|\boldsymbol{u}\|_{l^2(t_0,T;[W^{k_u+1,2}(\Omega)]^d)}^2 + C (\Delta t)^2
 \\&\quad
+C\frac{\|\widetilde{\boldsymbol{\zeta}}_u\|_{l^2(t_0,T;LPS)}^4}{\nu^2\Delta t}
\end{align*}
\end{proof}

\begin{corollary}
 Theorem~\ref{thm:ansatz_guermond} states in the limiting case $\Delta t\to 0$
\begin{align*}  
 &\|\boldsymbol{\zeta}_u^m\|_0^2 \\&
\leq C_G\Big((\nu+\gamma)h^{2k_u+2} +\gamma^{-1}h^{2k_p+4}
\\&\qquad\quad
 +\Big((\nu+\gamma)h^{2k_u}+\frac{h^{2k_p+2}}{\gamma}+\nu h^{2s}\Big)\Big(\nu^{-1}\max_{1\leq n\leq m}\max_{M\in\mathcal M_h}\{\tau_M^n |\widetilde{\boldsymbol{u}}_M^n|^2\}+1\Big)\Big)\\
  &\Delta t \sum_{n=1}^m\left(\nu \|\nabla \boldsymbol{\zeta}_u^n\|_0^2 + \gamma \|\nabla \cdot  
\boldsymbol{\zeta}_u^n\|_0^2
+ 
\sum_{M\in\mathcal 
M_h}\tau_M^n\|\kappa_M((\widetilde{\boldsymbol{u}}_h^n\cdot\nabla)\boldsymbol{\zeta}_u^n)\|_{0,M}^2
\right)\\
&\leq C_G\Big((\nu+\gamma)h^{2k_u} +\gamma^{-1}h^{2k_p+2}
\\&\qquad\quad
 +\Big((\nu+\gamma)h^{2k_u}+\frac{h^{2k_p+2}}{\gamma}+\nu h^{2s}\Big)\Big(\nu^{-1}\max_{1\leq n\leq m}\max_{M\in\mathcal M_h}\{\tau_M^n |\widetilde{\boldsymbol{u}}_M^n|^2\}+1\Big)\Big)  
 \end{align*}
 and in the limiting case $h\to 0$ Theorem~\ref{thm:ts_full_errors_old} gives
 \begin{align}
 \nonumber
  &\|\boldsymbol{\zeta}_u^m\|_0^2 \leq C_{G,h}\max_{1\leq n\leq m}\max_{M\in\mathcal M_h}\{\tau_M^n |\widetilde{\boldsymbol{u}}_M^n|^2\} \frac{(\Delta t)^2}{\nu} +C\frac{(\Delta t)^4}{\nu^3}\\
  &\Delta t \sum_{n=1}^m\left(\nu \|\nabla \boldsymbol{\zeta}_u^n\|_0^2 + \gamma \|\nabla \cdot  
\boldsymbol{\zeta}_u^n\|_0^2
+ 
\sum_{M\in\mathcal 
M_h}\tau_M^n\|\kappa_M((\widetilde{\boldsymbol{u}}_h^n\cdot\nabla)\boldsymbol{\zeta}_u^n)\|_{0,M}^2
\right)\\
&\quad\leq C_{G,h}\max_{1\leq n\leq m}\max_{M\in\mathcal M_h}\{\tau_M^n 
|\widetilde{\boldsymbol{u}}_M^n|^2\} \frac{(\Delta t)^2}{\nu} +C(\Delta t)^2.
 \nonumber
 \end{align}

\label{cor:ts_rates}
\end{corollary}
\begin{remark}
Equilibrating the spatial rates of convergence yields $k_u=k_p+1$ as it is for most inf-sup stable finite element pairs.
 Furthermore, the above equations show that the error estimates are quasi-optimal with respect to the LPS error if the LPS parameters are bounded, $\gamma\sim 1$ and $h^{k_u}+\Delta t \lesssim \nu$.
 However, the errors on the energy norm are optimal only for the stricter bound 
 \begin{align}
  \max_{1\leq n\leq m}\max_{M\in\mathcal M_h}\{\tau_M^n |\widetilde{\boldsymbol{u}}_M^n|^2\}\lesssim\min\{h^2, (\Delta t)^2\}
 \end{align}
 on the LPS parameter if the mesh width and the time step size fulfill
 \begin{align*} 
  \Delta t \lesssim \min\{\nu, h^2\}\quad\text{and}\quad h^{k_u} \lesssim \nu.
    \end{align*}
    \end{remark}
    \begin{remark}
Comparing the physical dimensions in the momentum equation, we obtain
\begin{align*}
\left[\tau_M^n(\widetilde{\boldsymbol{u}}_M^n)\right]\frac{m^2}{s^4}&= \left[s_u(\widetilde{\boldsymbol{u}}_h;\widetilde{\boldsymbol{u}}_h,\widetilde{\boldsymbol{u}}_h)\right]=\left[\left(\frac{\partial \widetilde{\boldsymbol{u}}_h}{\partial t}, \widetilde{\boldsymbol{u}}_h\right)\right]=\frac{m^2}{s^3}\\
\left[\gamma(\widetilde{\boldsymbol{u}}_M^n)\right]\frac{1}{s^2}&=\left[\gamma\|\nabla\cdot\widetilde{\boldsymbol{u}}_h^n\|_0^2\right]=\left[\left(\frac{\partial \widetilde{\boldsymbol{u}}_h}{\partial t}, \widetilde{\boldsymbol{u}}_h\right)\right]=\frac{m^2}{s^3}.
\end{align*}
This suggests a parameter design as 
\begin{align}
\label{eqn:physical_parameter}
\tau_M^n(\widetilde{\boldsymbol{u}}_M^n)&\sim \min\left\{
\frac{h_M}{|\widetilde{\boldsymbol{u}}_M^n|}, \frac{h_M^2}{\nu}\right\},& \gamma &\sim \max\left\{h_M|\widetilde{\boldsymbol{u}}|, \frac{|p^n|_{k, M}}{|\widetilde{\boldsymbol{u}}^n|_{k+1, M}}-\nu\right\}.
\end{align}
For the LPS SU parameter this is within the parameter bounds giving quasi-optimal results for the LPS error if we neglect quasi-robustness. 
The choice $\tau_M^n(\widetilde{\boldsymbol{u}}_M^n)=h_M/(2|\widetilde{\boldsymbol{u}}_M^n|)$ is in accordance with the setup
of the LPS parameter in \cite{colomes2016mixed} for the convection dominated case and we will stick to this in the numerical examples.
Note however, that we never require a lower bound for the LPS parameter in the error estimates. In particular, quasi-optimality and quasi-robustness also hold for $\tau_M^n = 0$.\\

A practical and analytically satisfying answer to the scaling of the grad-div stabilization parameter is still open. 
Choosing $\gamma\sim h_M|\widetilde{\boldsymbol{u}}|$ is not consistent with the our error estimates.
According to \cite{Jenkins2013b}, a good choice depends on (unknown) norms of the solution and on the question if the space of weakly divergence-free subspaces has an optimal
approximation property. 
Note that for $\gamma\to\infty$ the usage of Taylor-Hood elements on barycentrically refined simplicial meshes is equivalent to using discretely divergence-free Scott-Vogelius elements \cite{Case-2011}.
However, neither of these results seems to give a practical design .
E.g. in case $\boldsymbol{u}\equiv0$, the right-hand side in the error estimates basically reduces to $\min\{\gamma^{-1},\nu^{-1}\}h^{2k_p+4-2z}\|p\|_{L^\infty(t_0,T;W^{k_p+1,2})}^2$
and hence $\gamma\geq 1$ is required. Numerical examples (cf. \cite{arndt2014local, Dallmann2015, Olshanskii2009}) show that also in this case and for a manufactured solution a bounded grad-div parameter $\gamma=\mathcal{O}(1)$ gives the best results with respect to the $H^1$ norm for the velocity. 
This motivates to choose a constant grad-div stabilization parameter $\gamma=\mathcal{O}(1)$ also in this paper.

\end{remark}

\section{Numerical Results}\label{sec:numerics}
We comprehend our considerations with some numerical results. On the one hand, we want to investigate in which respect the proven estimates might be sharp. On the other hand, we want to study the influence of the stabilization. Therefore, we consider two examples. In the first one, the method of manufactured solutions is used to compute rates of convergence numerically. For a more realistic case, the Taylor-Green vortex is considered in the second example.
\subsection{Academic example}

The considered example is one for which we compute the forcing term $\boldsymbol{f}$ such that 
 \begin{align*}
  \boldsymbol{u}(x, y, t) &:= (\sin(1 - x) \sin(y + t), \cos(1 - x)\cos(y + t))^T\\
  p(x, y, t) &:= -\cos(1 - x) \sin(y + t)
  \end{align*}
  is the solution to the time-dependent Navier-Stokes problem in the domain $\Omega=[-1, 1]^2$ and for $t\in [0, 1]$. 
  
  The standard incremental pressure-correction scheme considered in this work is compared with the rotational pressure-correction scheme proposed by
  Timmermans in \cite{Timmermans1996} and analyzed for the Stokes problem by Guermond and Shen in \cite{guermond2004error}. For our setting it reads (with $\chi=1$):
  \begin{center}
   \begin{myminipage}
Find $\widetilde{\boldsymbol{u}}_h^n\in\boldsymbol{V}\!_h,\boldsymbol{u}_h^n\in\boldsymbol{Y}\!_h$ and $p_h^n \in Q_h$ such that
\begin{align}
\begin{aligned}
&\left(\frac{3 \widetilde{\boldsymbol{u}}_h^n-4\boldsymbol{u}^{n-1}_h+\boldsymbol{u}^{n-2}_h}{2 \Delta 
t}, \boldsymbol{v}_h\right) 
+\nu(\nabla \widetilde{\boldsymbol{u}}_h^n, \nabla \boldsymbol{v}_h) + c(\widetilde{\boldsymbol{u}}_h^n, \widetilde{\boldsymbol{u}}_h^n, \boldsymbol{v}_h)\\
&\quad + \gamma (\nabla\cdot \widetilde{\boldsymbol{u}}_h^n, \nabla\cdot \boldsymbol{v}_h) 
+ s_h(\widetilde{\boldsymbol{u}}_h^n, \widetilde{\boldsymbol{u}}_h^n, 
\widetilde{\boldsymbol{u}}_h^n, \boldsymbol{v}_h)
= (\boldsymbol{f}^n, \boldsymbol{v}_h) +( p^{n-1}_h,\nabla\cdot 
\boldsymbol{v}_h)\\
&\widetilde{\boldsymbol{u}}_h^n|_{\partial\Omega}=0
\label{eqn:rot1}
\end{aligned}
\end{align}
\begin{align}
\begin{aligned}
\left(\frac{3 \boldsymbol{u}_h^n-3\widetilde{\boldsymbol{u}}_h^n}{2 \Delta t}
+\nabla (p_h^{n}-p_h^{n-1}+\chi\nu\pi_{Q_h}(\nabla\cdot\widetilde{\boldsymbol{u}}_h^n)), \boldsymbol{y}_h \right)&=0\\
(\nabla\cdot \boldsymbol{u}_h^n, q_h)&=0\\
\boldsymbol{u}_h^n|_{\partial\Omega} &=0
\label{eqn:rot2}
\end{aligned}
\end{align}
holds for all $\boldsymbol{v}_h\in \boldsymbol{V}\!_h,~\boldsymbol{y}_h\in \boldsymbol{Y}\!_h $ and $q_h\in Q_h$.
\end{myminipage}
\end{center}
In this algorithm $\pi_{Q_h}$, denotes the $L^2(\Omega)$ projection into the discrete pressure ansatz space $Q_h$.
The modification in the projection equation should prevent that the otherwise artificial boundary condition $\boldsymbol{n} \cdot \nabla p_h^n = \dots =\boldsymbol{n} \cdot \nabla p_h^0$
dominates the error. For the Stokes problem, it can be shown that both the velocity error with respect to the $H^1(\Omega)$ norm and the
the pressure error with respect to the $L^2(\Omega)$ norm a rate of convergence as $(\Delta t)^{3/2}$ can be expected.\\
Although we did not carry out the analysis for this algorithm adapted
to our approach, we still believe that similar results hold true and can be observed numerically.
Note that the above scheme reduces to the standard incremental pressure-correction scheme analyzed in this work for $\chi=0$ .
  
To study the dependence of the error on the diffusion parameter we consider three different Reynolds numbers $Re\in\{10^{-2}, 1, 10^2\}$.
Additionally, we want to investigate whether stabilization really improves the results numerically. 
As a first result we figured out that LPS SU does not show any significant influence on the error in the considered parameter regime. 
Therefore, we just consider grad-div stabilization in the following.

In Figures~\ref{fig:Re10-0ul2}, \ref{fig:Re10-0uh1}, \ref{fig:Re10-0ul2div} and \ref{fig:Re10-0pl2} the case $Re=1$ is considered.
For the errors with respect to the $L^2(\Omega)$ and $H^1(\Omega)$ norm of the velocity we see almost no influence
whether we choose the rotational or incremental form or stabilization or not. The rates of convergence with respect to spatial discretization are again as expected.
This time we observe for both quantities approximately second order of convergence with respect to time.
For the divergence of the velocity we again see a quite big influence of the rotational compared to the standard incremental form.
In this case the grad-div stabilization is of minor importance.
Finally, for the $L^2(\Omega)$ velocity norm we see four distinct results. The rotational form gives a smaller error than the standard form and within these groups stabilization increases.
This in fact is the first result in which we see that grad-div stabilization is harmful for an error.

\begin{figure}
\centering	
\hspace*{\fill}\begin{subfigure}[b]{.39\textwidth}
\includegraphics[clip, trim=0cm 0cm 2cm 1cm, width=\textwidth]{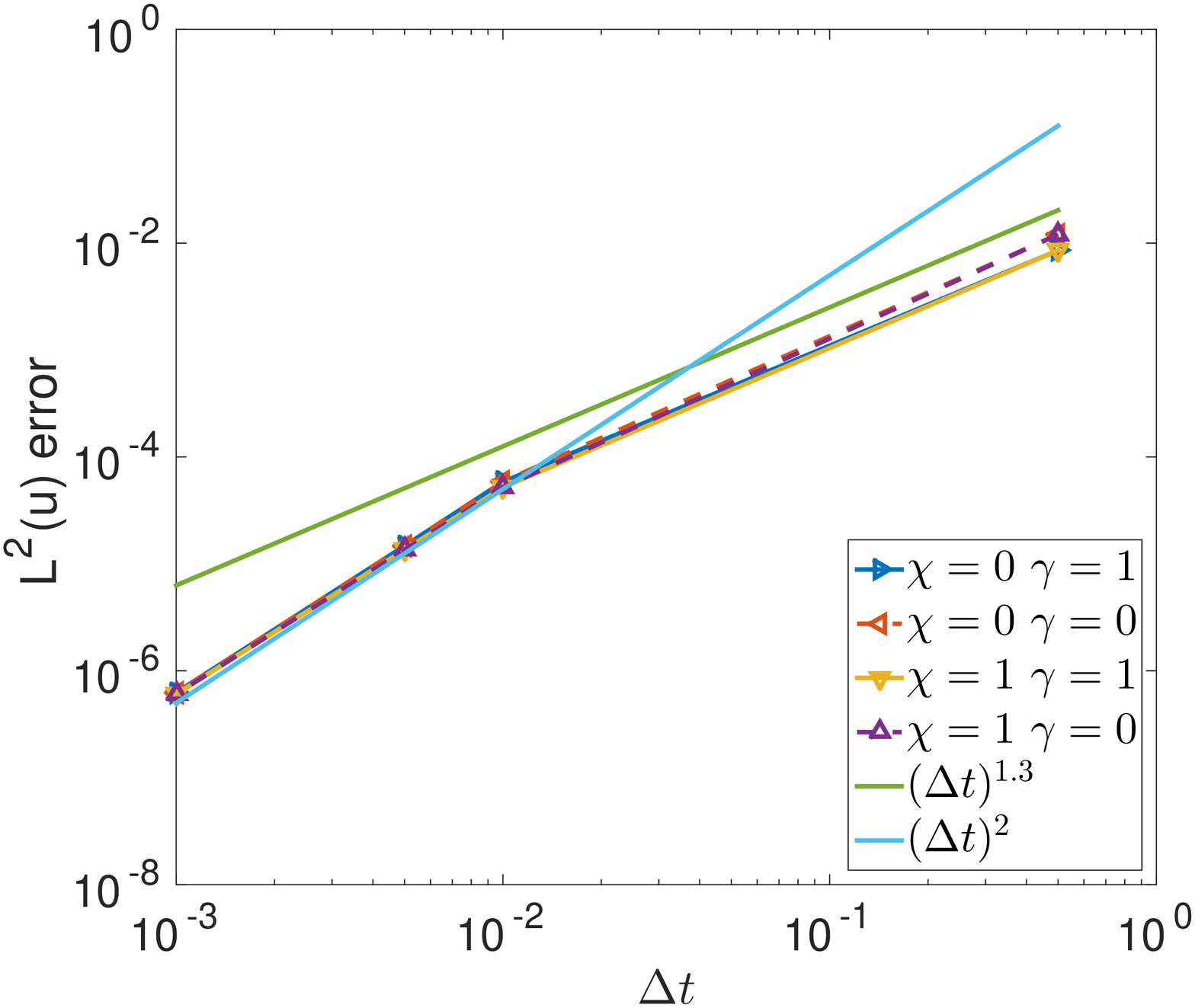}
\caption{$h=2^{-4.5}$ fixed}
  \end{subfigure}\hfill
  \begin{subfigure}[b]{.39\textwidth}
\includegraphics[clip, trim=0cm 0cm 2cm 1cm, width=\textwidth]{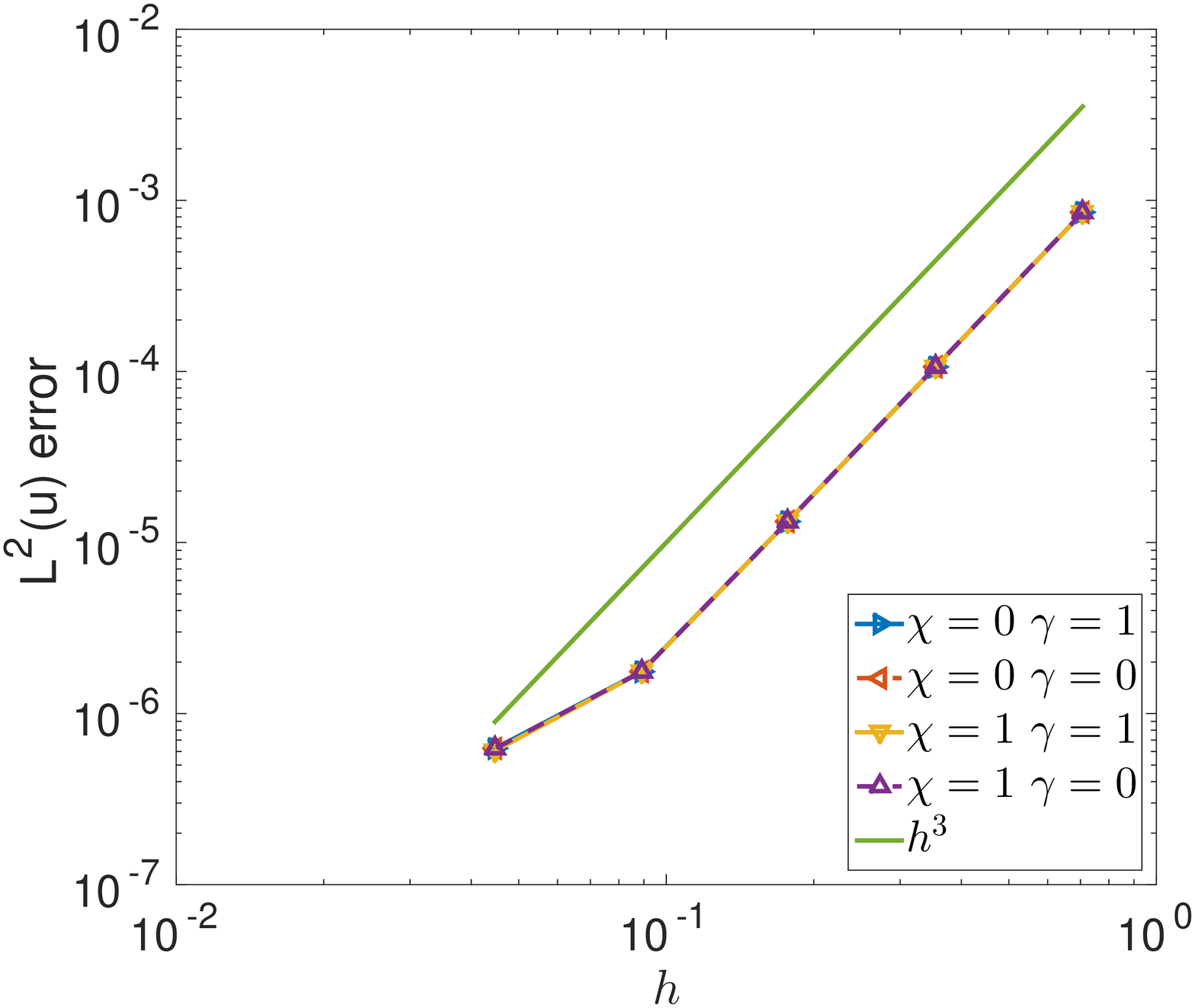}
\caption{$\Delta t=10^{-3}$ fixed}
  \end{subfigure}\hspace*{\fill}
 \caption{$Re=1$, errors for the velocity with respect to $L^2(\boldsymbol{u})$}
 \label{fig:Re10-0ul2}
\end{figure}

\begin{figure}
\centering
\hspace*{\fill}\begin{subfigure}[b]{.39\textwidth}
\includegraphics[clip, trim=0cm 0cm 2cm 1cm, width=\textwidth]{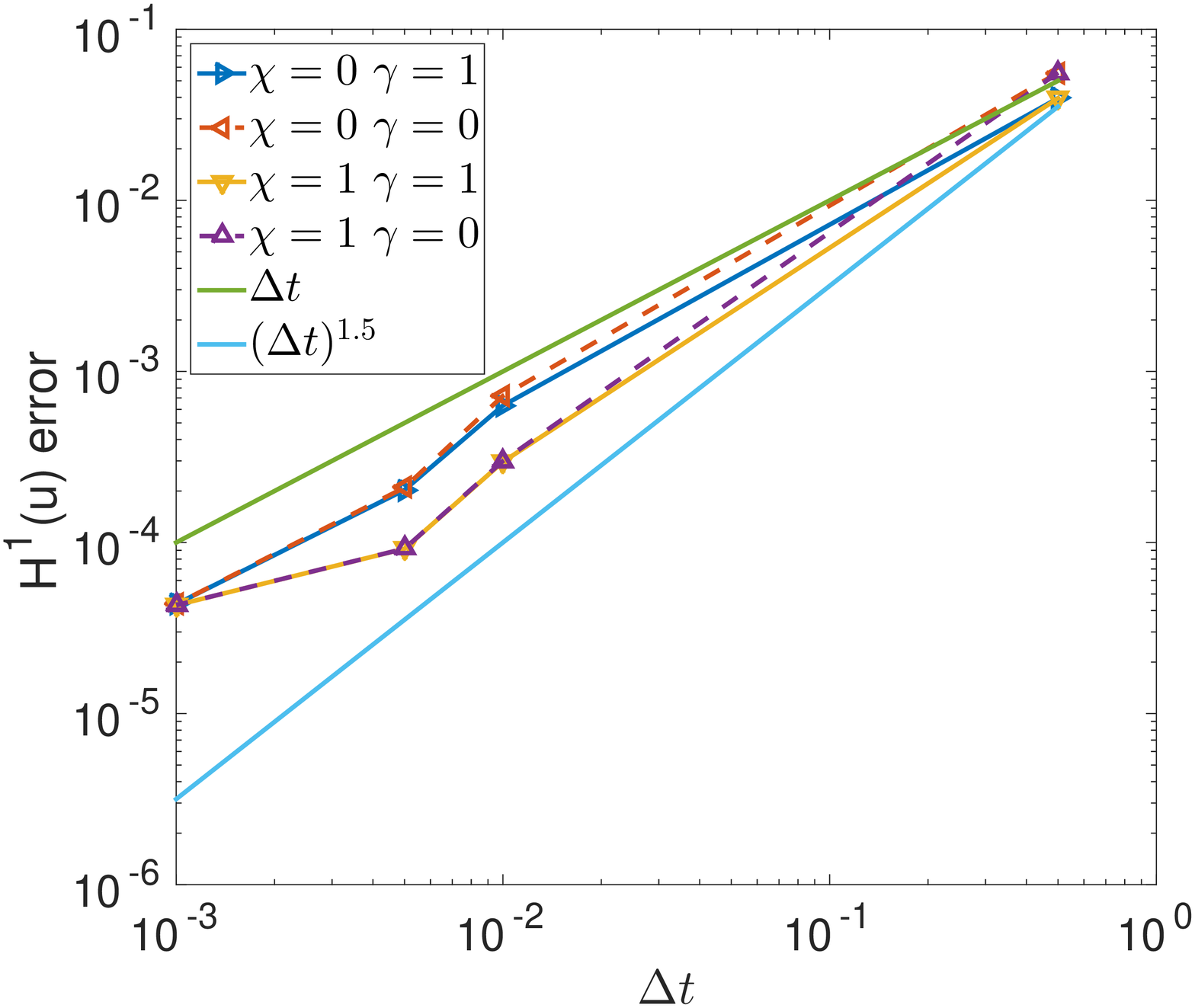}
\caption{$h=2^{-4.5}$ fixed}
  \end{subfigure}\hfill
  \begin{subfigure}[b]{.39\textwidth}
\includegraphics[clip, trim=0cm 0cm 2cm 1cm, width=\textwidth]{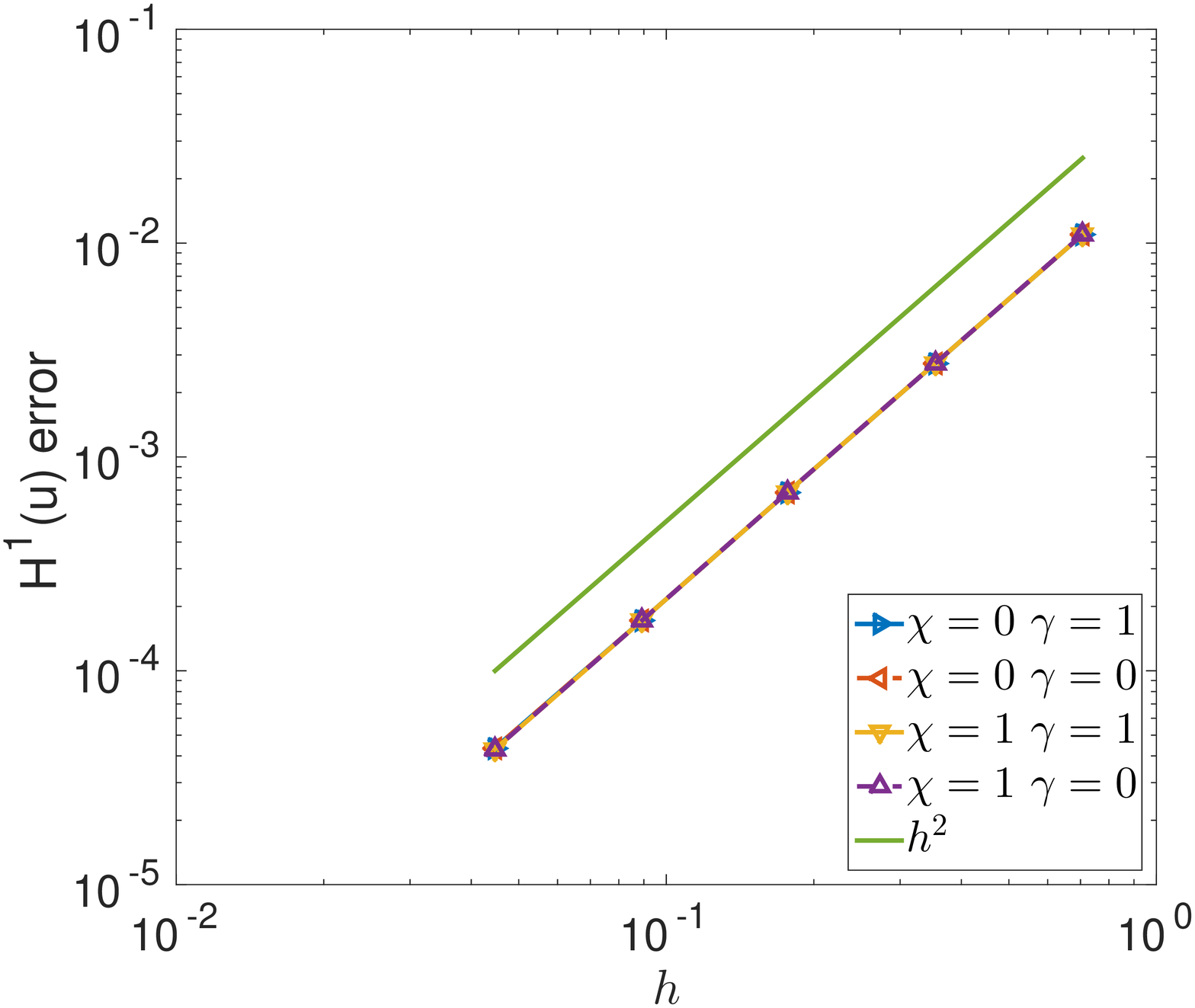}
\caption{$\Delta t=10^{-3}$ fixed}
  \end{subfigure}\hspace*{\fill}
 \caption{$Re=1$, errors for the velocity with respect to $H^1(\boldsymbol{u})$}
 \label{fig:Re10-0uh1}
\end{figure}

\begin{figure}
\centering
\hspace*{\fill}\begin{subfigure}[b]{.39\textwidth}
\includegraphics[clip, trim=0cm 0cm 2cm 1cm, width=\textwidth]{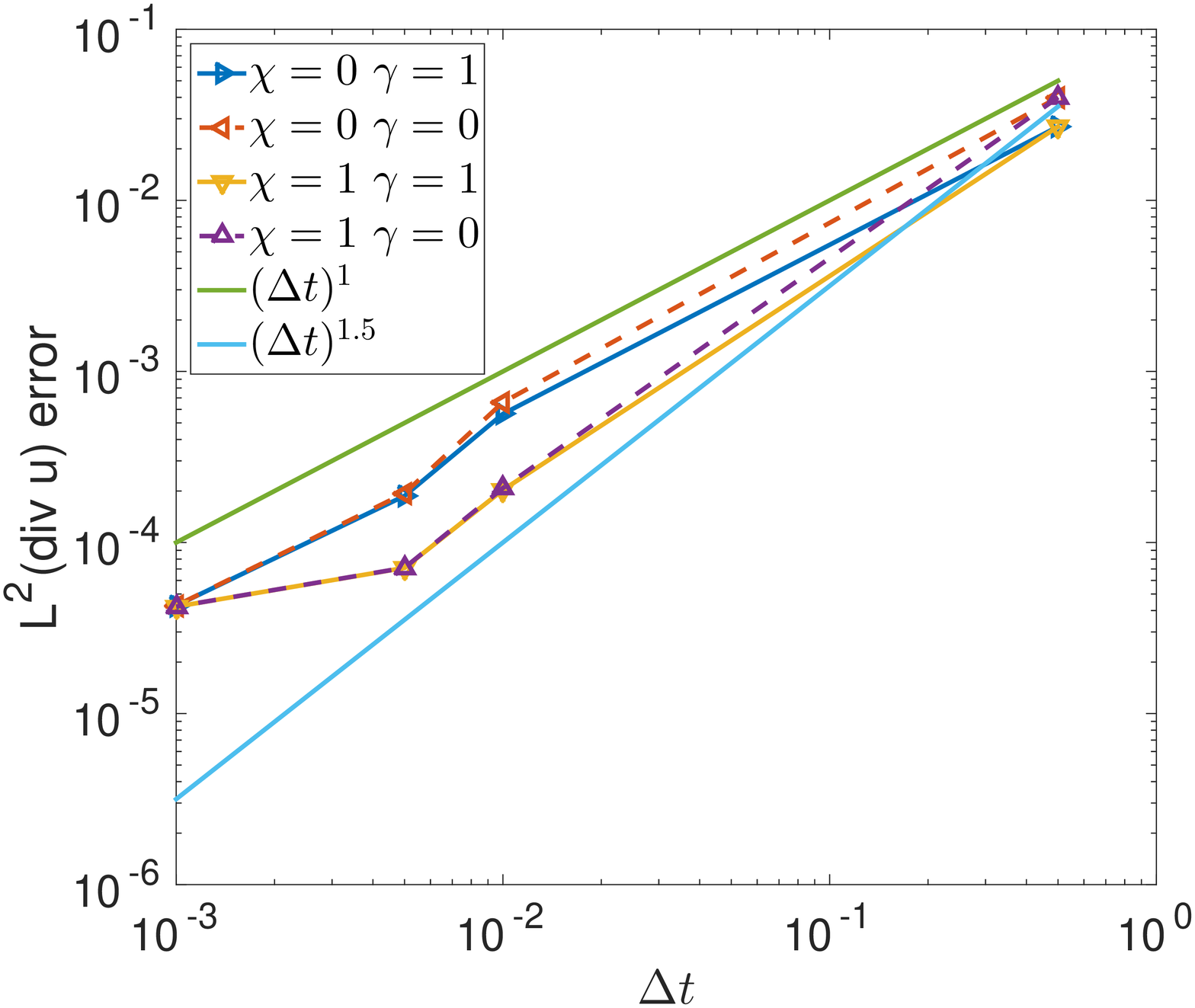}
\caption{$h=2^{-4.5}$ fixed}
  \end{subfigure}\hfill
  \begin{subfigure}[b]{.39\textwidth}
\includegraphics[clip, trim=0cm 0cm 2cm 1cm, width=\textwidth]{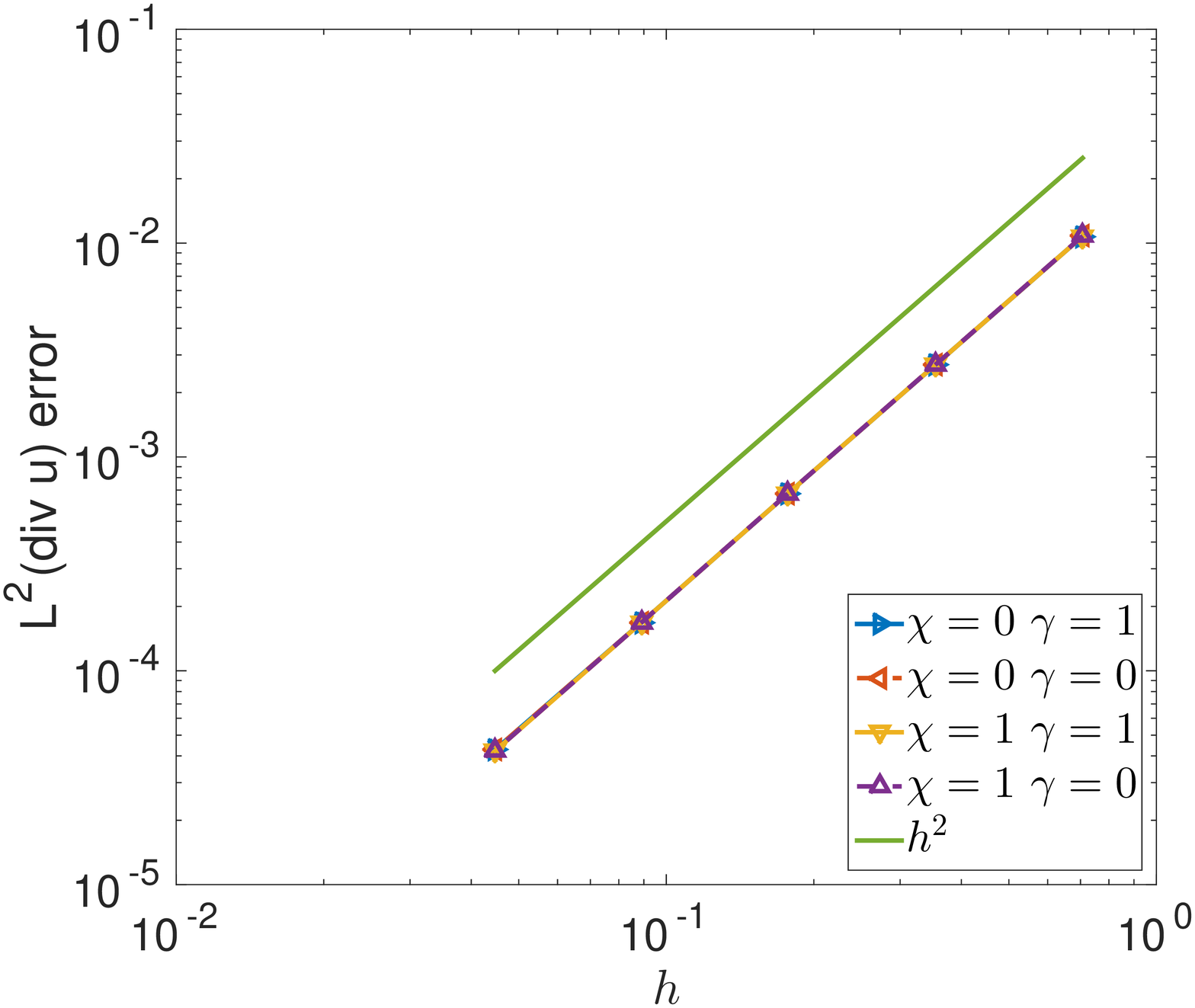}
\caption{$\Delta t=10^{-3}$ fixed}
  \end{subfigure}\hspace*{\fill}
  \caption{$Re=1$, errors for the velocity with respect to $L^2(\nabla\cdot\boldsymbol{u})$}
  \label{fig:Re10-0ul2div}
 \end{figure}

\begin{figure}
\centering
\hspace*{\fill}\begin{subfigure}[b]{.39\textwidth}
\includegraphics[clip, trim=0cm 0cm 2cm 1cm, width=\textwidth]{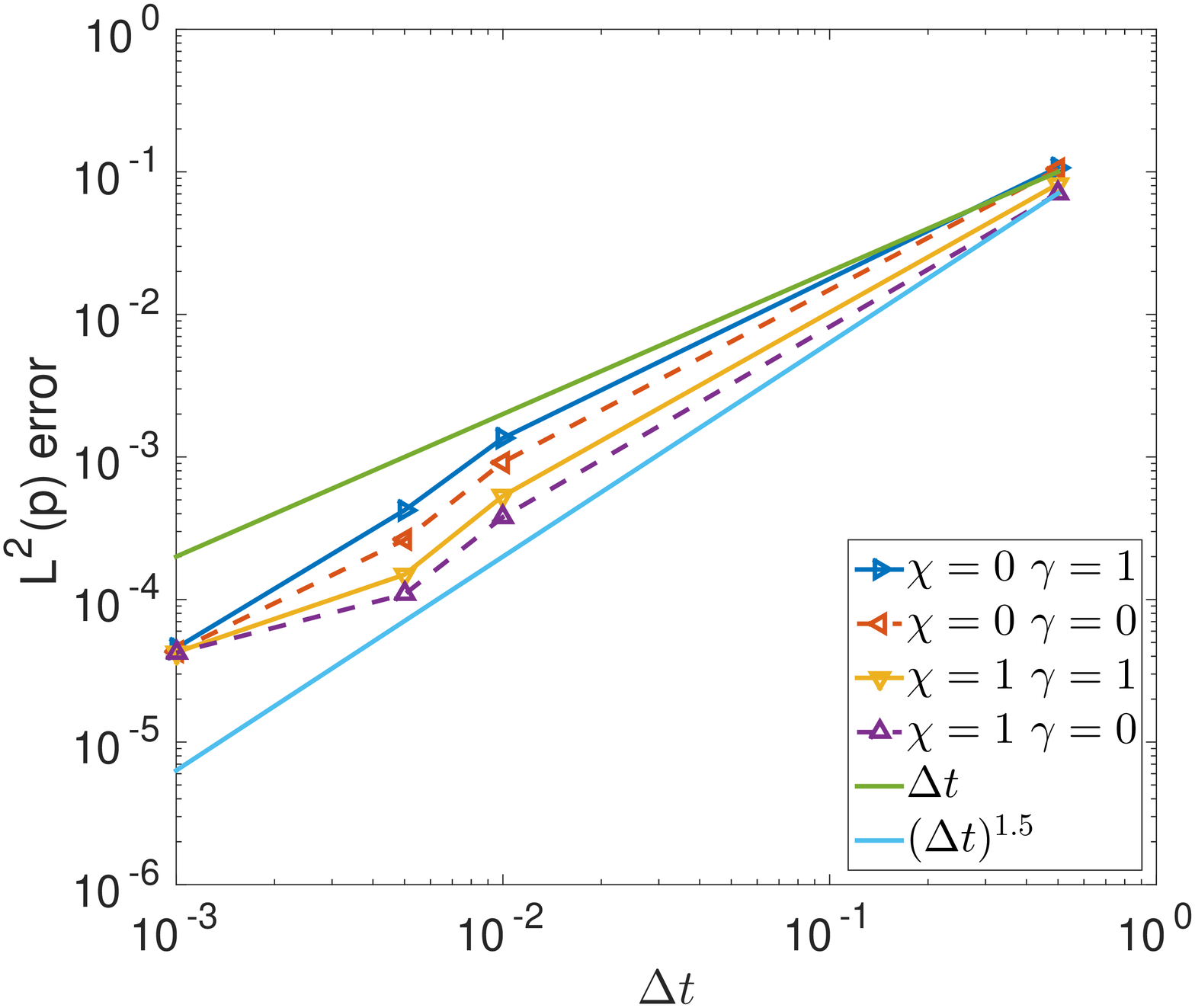}
\caption{$h=2^{-4.5}$ fixed}
  \end{subfigure}\hfill
  \begin{subfigure}[b]{.39\textwidth}
\includegraphics[clip, trim=0cm 0cm 2cm 1cm, width=\textwidth]{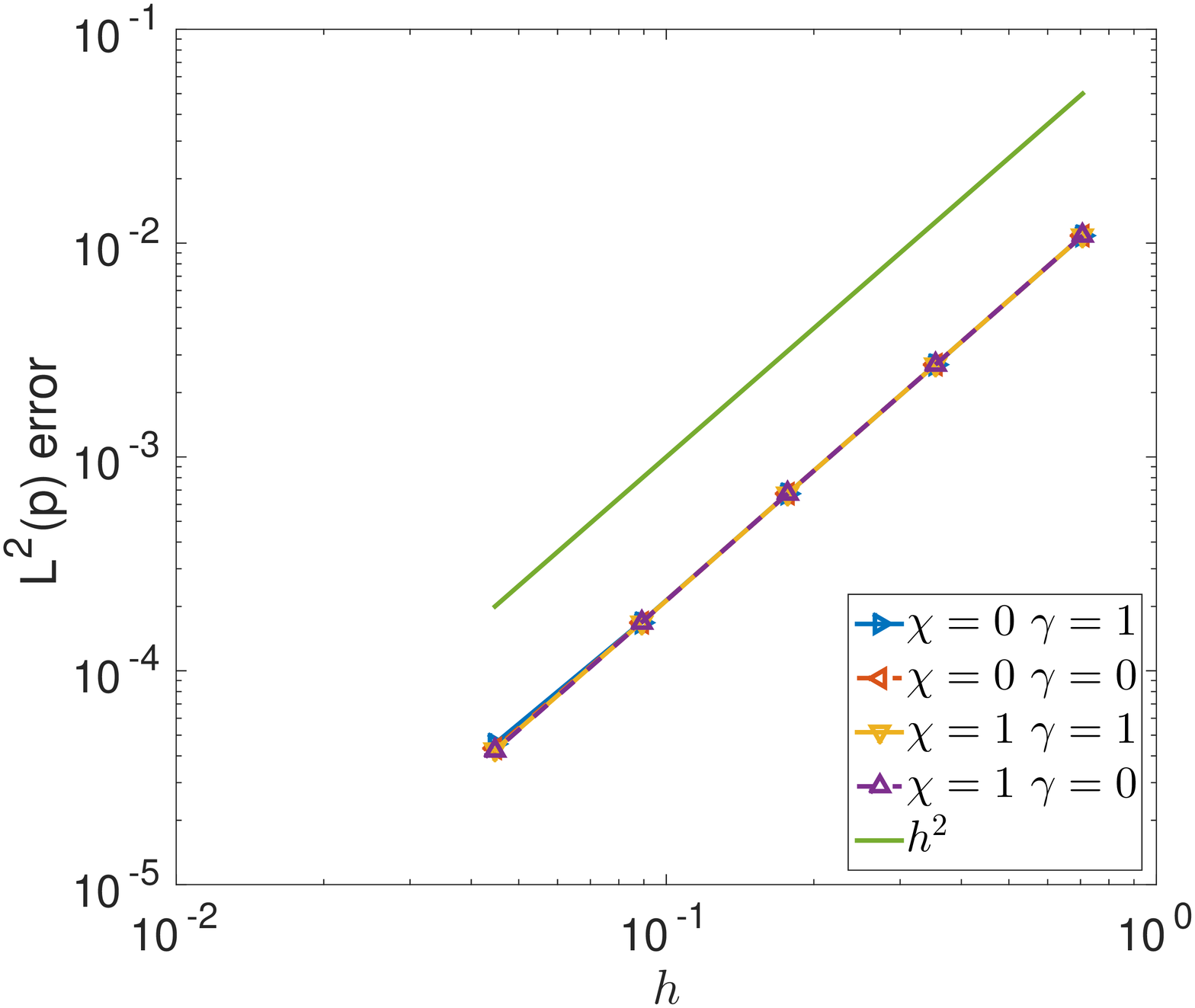}
\caption{$\Delta t=10^{-3}$ fixed}
  \end{subfigure}\hspace*{\fill}
\caption{$Re=1$, errors for the pressure with respect to $L^2(p)$}
  \label{fig:Re10-0pl2}
  \end{figure}

Finally, we consider $Re=10^2$ in Figures~\ref{fig:Re10+2ul2}, 
\ref{fig:Re10+2uh1}, \ref{fig:Re10+2ul2div} and \ref{fig:Re10+2pl2}.
For the the first three errors we get a clear picture. Grad-div stabilization diminishes the error by a fixed factor. 
In fact our analysis tells us in this parameter regime that grad-div stabilization improves the dependence on the Reynolds number $Re$ from
$Re^{1}$ to $Re^{1/2}$. This is exactly the behavior that we observe. On the other hand, whether we choose the rotational or incremental form is of minor
influence. Since the correction terms vanishes with decreasing $\nu$ this is not too surprising.
The rates of convergence that we observe are again optimal in the sense that we achieve the rates of the interpolation operators with respect to spatial discretization and for the time discretization a behavior like $(\Delta t)^2$ respectively $(\Delta t)^{1.5}$. 
Note that in view of the analysis carried out for these types of schemes the results are superconvergent with respect to time discretization and the LPS and the pressure error.\\
Compared to the velocity, for the pressure error the behavior with respect to stabilization is the other way around, similar to the results for $Re=1$. Again, the smallest error is obtained when no stabilization is used and the 
effect of the rotational correction is negligible.
This reults strengthenes that the choice of stabilization parameters strongly depends on the error norm that is to be minimized.
There is apparently no choice rule that is best for both velocity and pressure.
Experiments with higher Reynolds number show the same qualitative behavior of the errors.\\

 In summary, one might say that the rotational correction does never harm and even improves the results considerably if the viscosity $\nu$ is not too small.
Grad-div stabilization however seems to be beneficial whenever the main interest is in the velocity solution. For the pressure the above example suggests that disabling the stabilization is the best option.

\begin{figure}
\centering
\hspace*{\fill}\begin{subfigure}[b]{.39\textwidth}
\includegraphics[clip, trim=0cm 0cm 2cm 1cm, width=\textwidth]{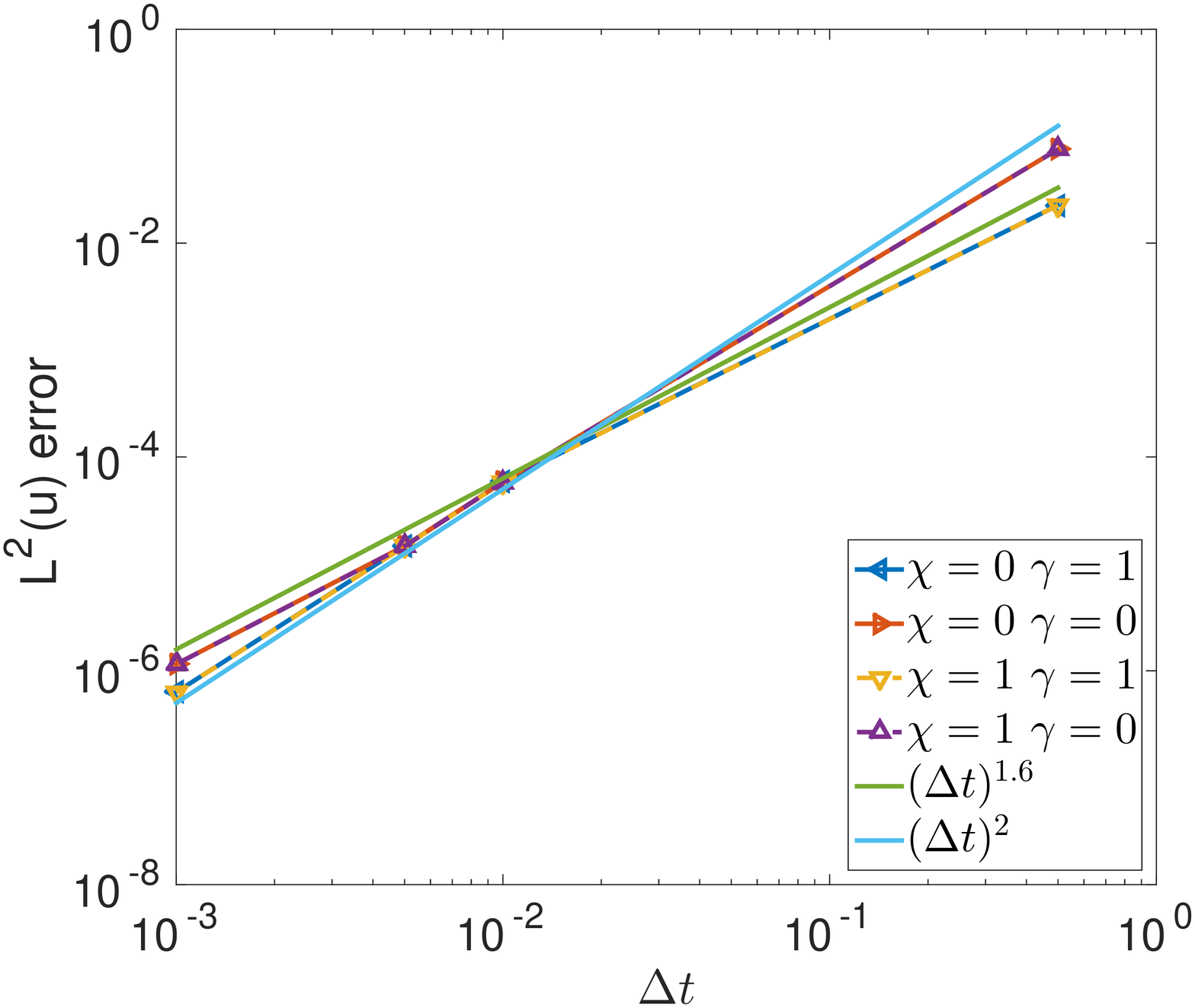}
\caption{$h=2^{-4.5}$ fixed}
  \end{subfigure}\hfill
  \begin{subfigure}[b]{.39\textwidth}
\includegraphics[clip, trim=0cm 0cm 2cm 1cm, width=\textwidth]{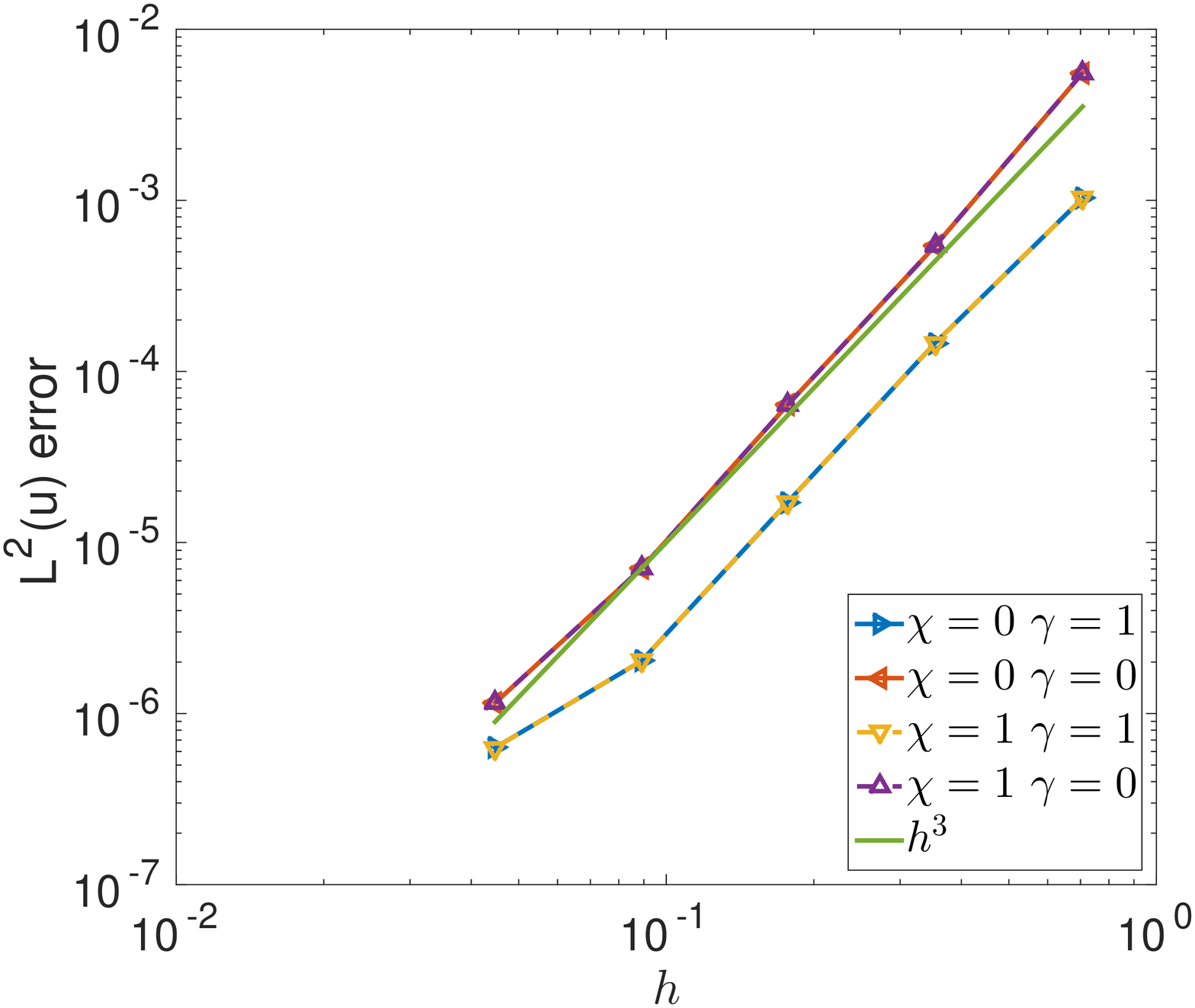}
\caption{$\Delta t=10^{-3}$ fixed}
  \end{subfigure}\hspace*{\fill}
  \caption{$Re=10^2$, errors for the velocity with respect to $L^2(\boldsymbol{u})$}
  \label{fig:Re10+2ul2}
  \end{figure}

\begin{figure}
\centering
\hspace*{\fill}\begin{subfigure}[b]{.39\textwidth}
\includegraphics[clip, trim=0cm 0cm 2cm 1cm, width=\textwidth]{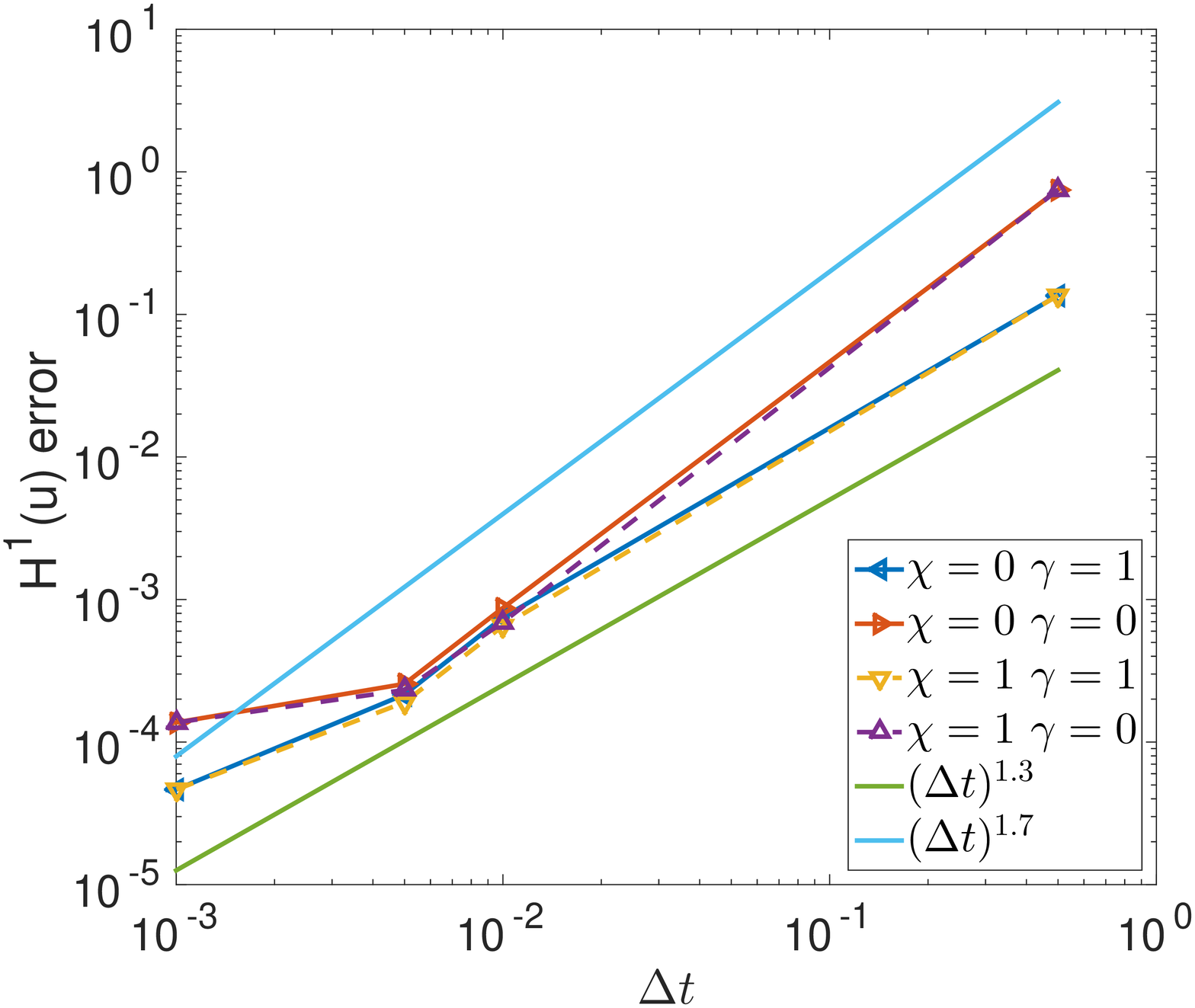}
\caption{$h=2^{-4.5}$ fixed}
  \end{subfigure}\hfill
  \begin{subfigure}[b]{.39\textwidth}
\includegraphics[clip, trim=0cm 0cm 2cm 1cm, width=\textwidth]{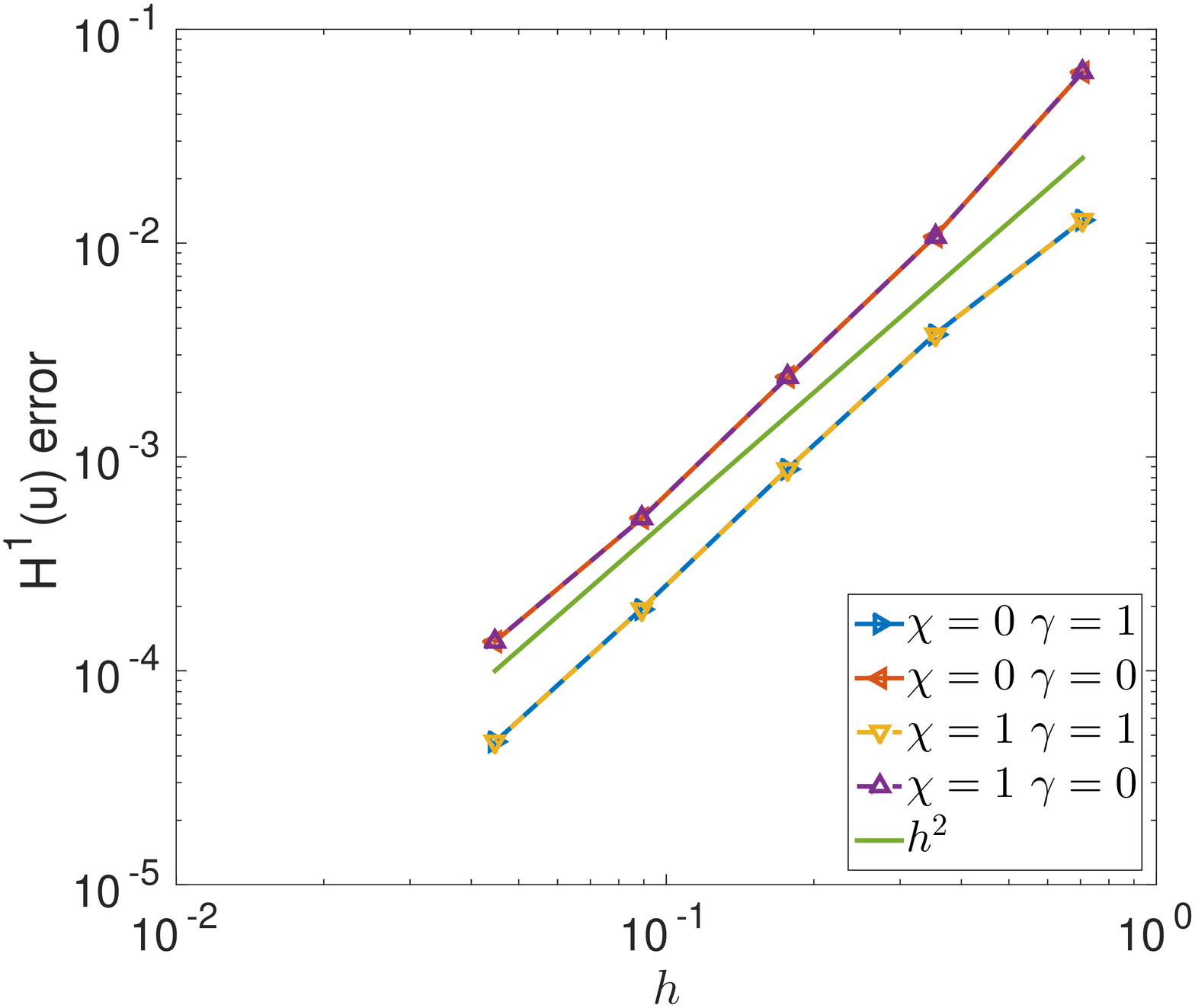}
\caption{$\Delta t=10^{-3}$ fixed}
  \end{subfigure}\hspace*{\fill}
  \caption{$Re=10^2$, errors for the velocity with respect to $H^1(\boldsymbol{u})$}
  \label{fig:Re10+2uh1}
  \end{figure}

\begin{figure}
\centering
\hspace*{\fill}\begin{subfigure}[b]{.39\textwidth}
\includegraphics[clip, trim=0cm 0cm 2cm 1cm, width=\textwidth]{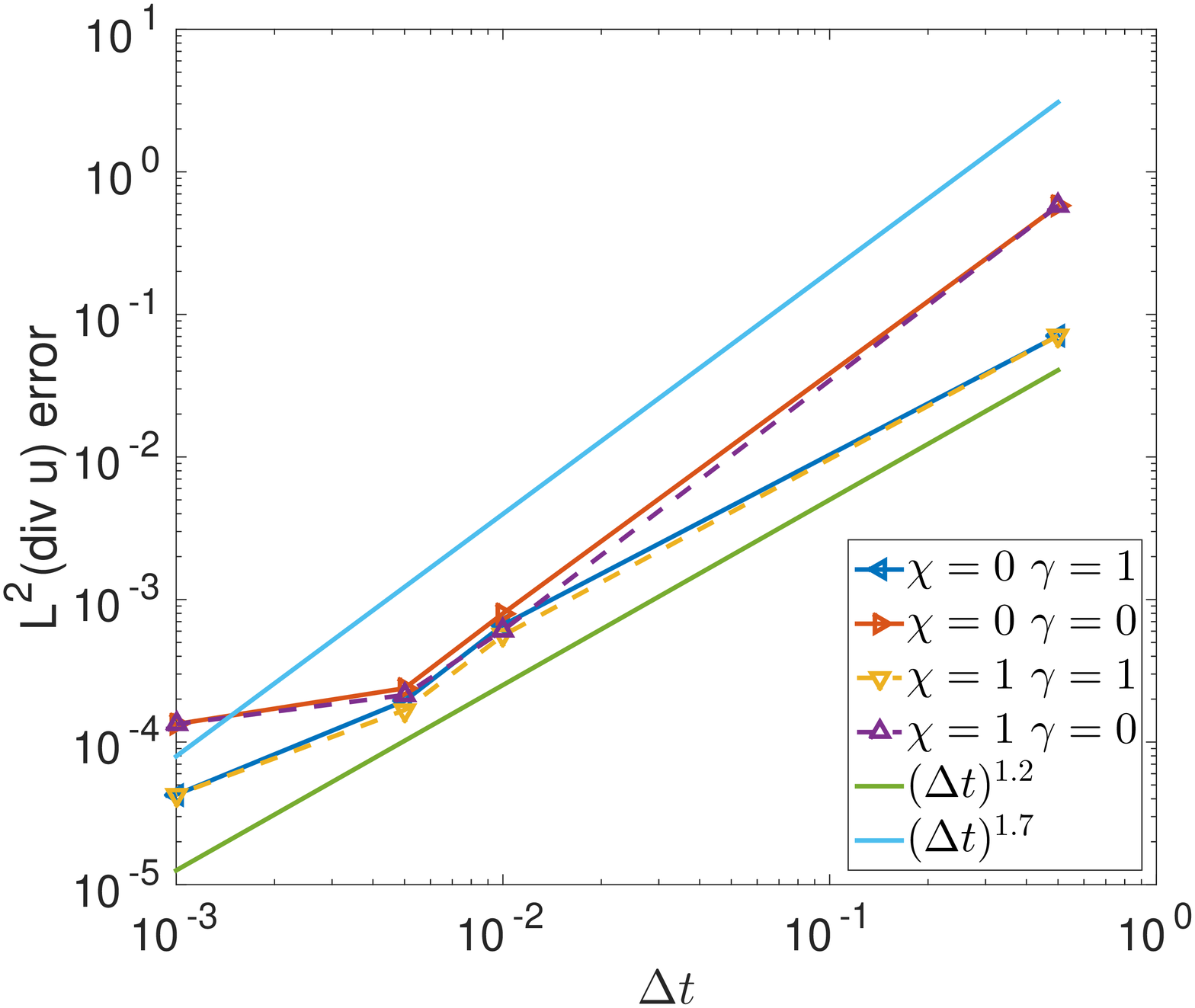}
\caption{$h=2^{-4.5}$ fixed}
  \end{subfigure}\hfill
  \begin{subfigure}[b]{.39\textwidth}
\includegraphics[clip, trim=0cm 0cm 2cm 1cm, width=\textwidth]{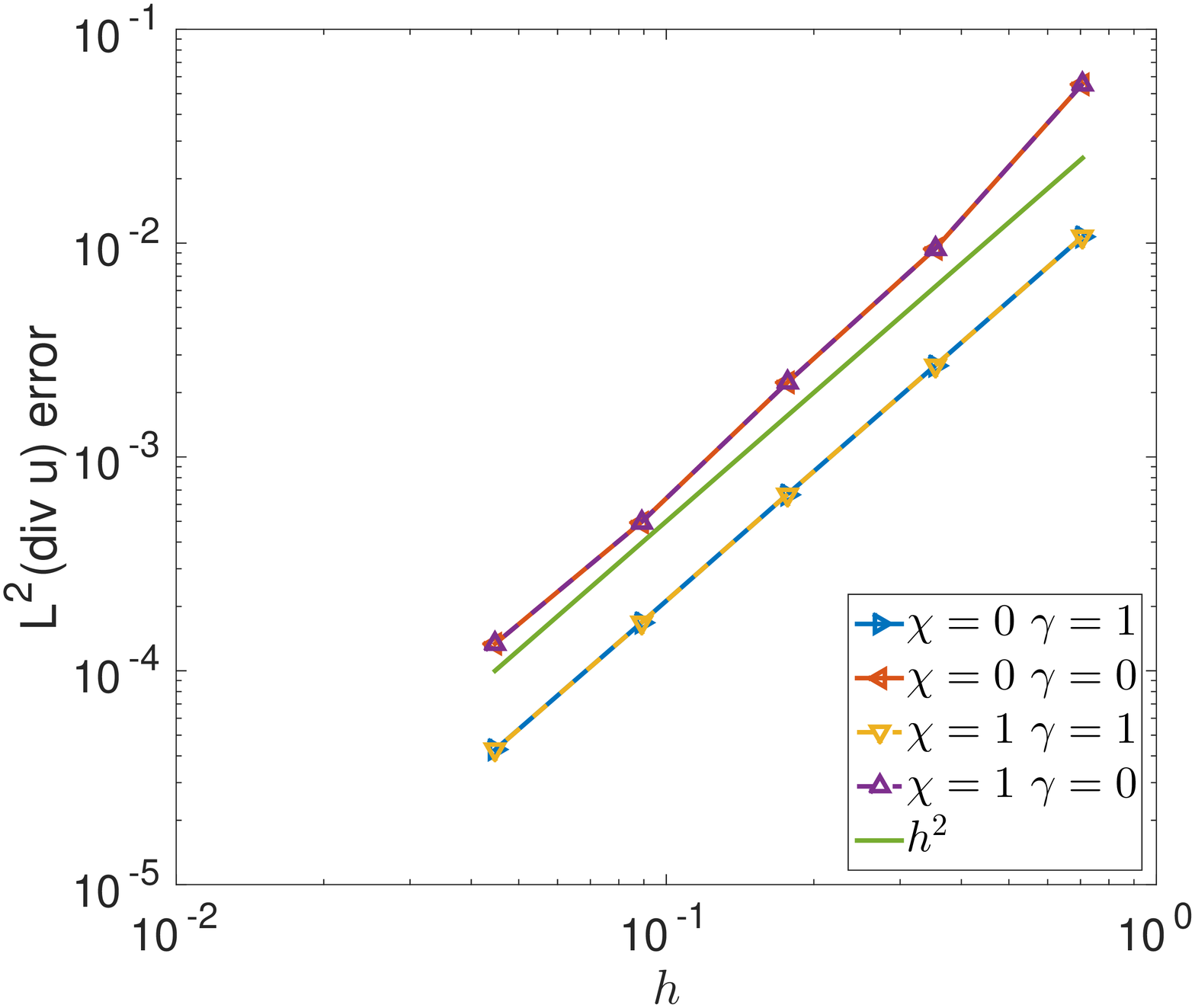}
\caption{$\Delta t=10^{-3}$ fixed}
  \end{subfigure}\hspace*{\fill}
  \caption{$Re=10^2$, errors for the velocity with respect to $L^2(\nabla\cdot\boldsymbol{u})$}
  \label{fig:Re10+2ul2div}
  \end{figure}

\begin{figure}
\centering
\hspace*{\fill}\begin{subfigure}[b]{.39\textwidth}
\includegraphics[clip, trim=0cm 0cm 2cm 1cm, width=\textwidth]{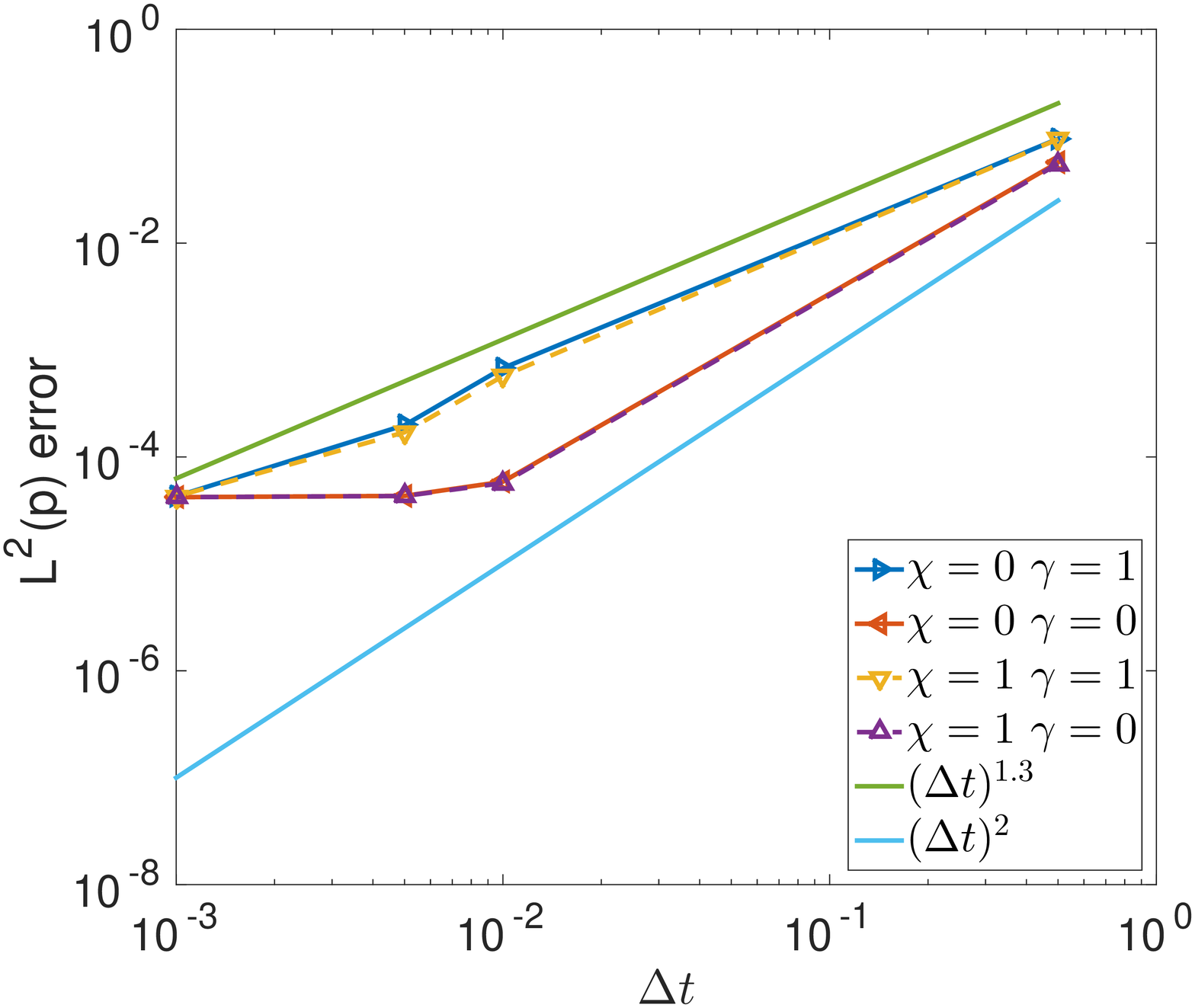}
\caption{$h=2^{-4.5}$ fixed}
  \end{subfigure}\hfill
  \begin{subfigure}[b]{.39\textwidth}
\includegraphics[clip, trim=0cm 0cm 2cm 1cm, width=\textwidth]{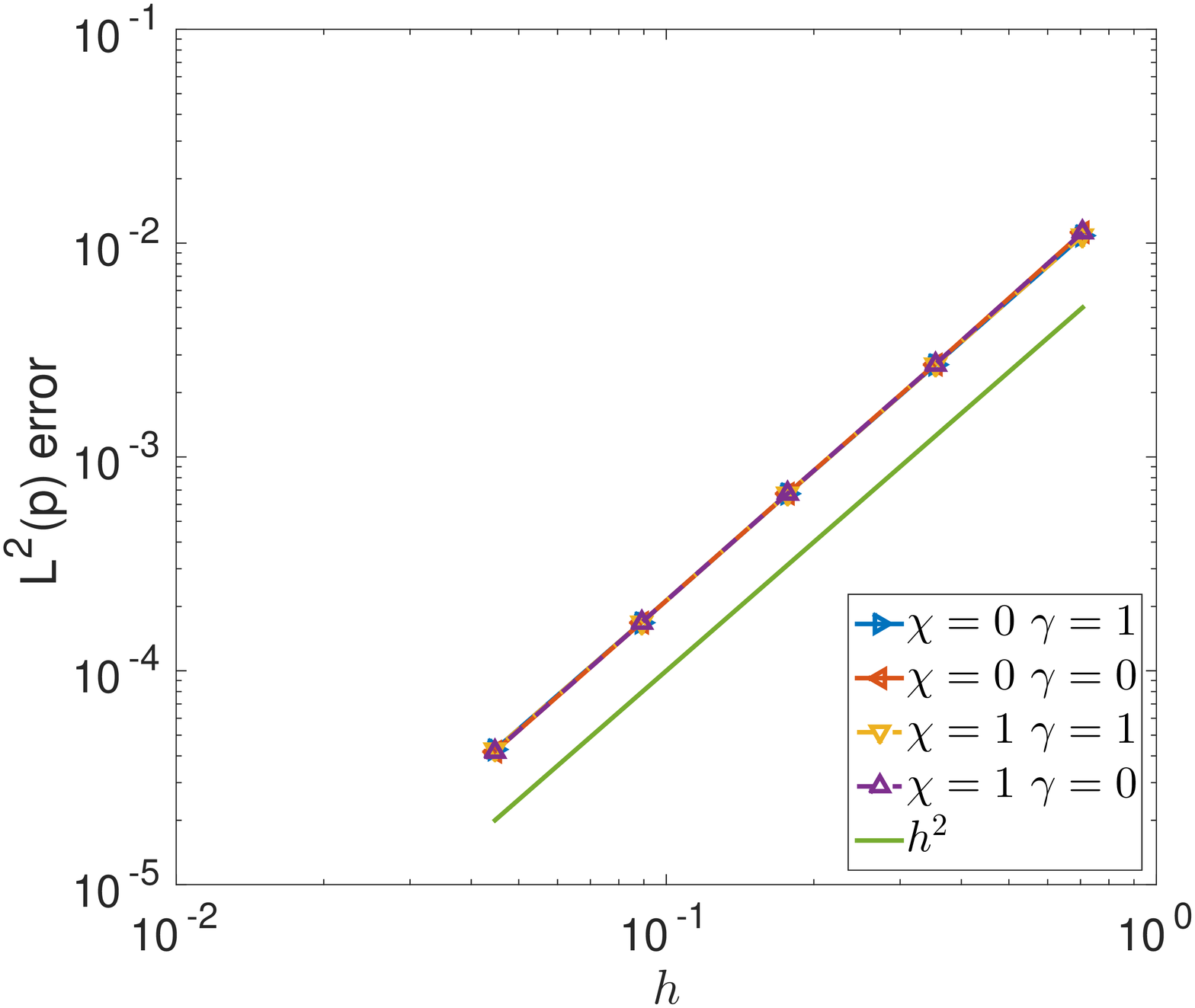}
\caption{$\Delta t=10^{-3}$ fixed}
  \end{subfigure}\hspace*{\fill}
  \caption{$Re=10^2$, errors for the pressure with respect to $L^2(p)$}
  \label{fig:Re10+2pl2}
\end{figure}

\subsection{Taylor-Green Vortex}
For an example in which we expect LPS SU to play a major role, we consider the case of a three-dimensional Taylor-Green vortex (TGV). In particular, we are interested in how good the stabilizations may serve as  implicit subgrid model for isotropic turbulence. 

We consider the flow in a periodic box 
$\Omega=(0, a)^3$ with some $a>0$ that we vary as needed. With $b>0$, the initial values are
\begin{align*}
{\boldsymbol{u}}_0&= b\cdot\begin{pmatrix}
	  \cos\left(\frac{2\pi}{a} x\right)\sin\left(\frac{2\pi}{a}y\right)\sin\left(\frac{2\pi}{a}z\right) \\
	  -\sin\left(\frac{2\pi}{a}x\right)\cos\left(\frac{2\pi}{a}y\right)\sin\left(\frac{2\pi}{a}z\right) \\ 
	  0
	  \end{pmatrix}, \\
p_0 &= b\cdot\frac{1}{16}\left(\cos\left(\frac{4\pi}{a}x\right)+\cos\left(\frac{4\pi}{a}y\right)\right) \left( \cos\left(\frac{4\pi}{a}z\right)+2\right)
\end{align*}
and we choose $a=2\pi,~b=1$ for $\nu=10^{-4}$ and $a=8/\sqrt{3},~b=1/10$ for $\nu=10^{-5}$. 
The time interval is chosen according to $t\in[0, 10/b]$ and we use Taylor-Hood elements for the discretization.
For the coarse space $\mathbb Q_1$ elements are used.

To evaluate the numerical results, we are interested in the energy spectrum at time $T=10/b$ given as 
\begin{align*}
E(k, t) &= \frac{1}{2} \sum_{k-\frac{1}{2}\leq |\boldsymbol{ k}|\leq k+\frac{1}{2}} \hat{\boldsymbol{u}}(\boldsymbol{ k}, t)\cdot \hat{\boldsymbol{u}}(\boldsymbol{k}, t)
\end{align*}
with the Fourier transform $\hat{\boldsymbol{u}}(\boldsymbol{k}, t) =\int_\Omega \boldsymbol{u}(\boldsymbol{x}, t)\exp(-i\boldsymbol{k}\boldsymbol{x}) d\boldsymbol{x}$.
For locally isotropic turbulence we expect in the inertial subrange the behavior (Kolmogorov's $-5/3$-law)
\begin{align*}
E(k, t) \sim \varepsilon^{2/3}k^{-5/3}, 
\end{align*}
where $\varepsilon$ is the turbulent dissipation rate.

In a first attempt, we consider grad-div stabilization alone (cf. Figure~\ref{fig:TGV_energy_GDandSU}(A)). The result clearly shows that the grad-div stabilization does not produce enough 
dissipation as the smallest resolved scales carry too much energy. Additional LPS SU cures this situation considerably but produces too much dissipation (Figure~\ref{fig:TGV_energy_GDandSU}(B)).

\begin{figure}[htp]
\centering
\hspace*{\fill}\begin{subfigure}[b]{.39\textwidth}
\includegraphics[trim= 0cm 0cm 2cm 1cm, clip, width=\textwidth]{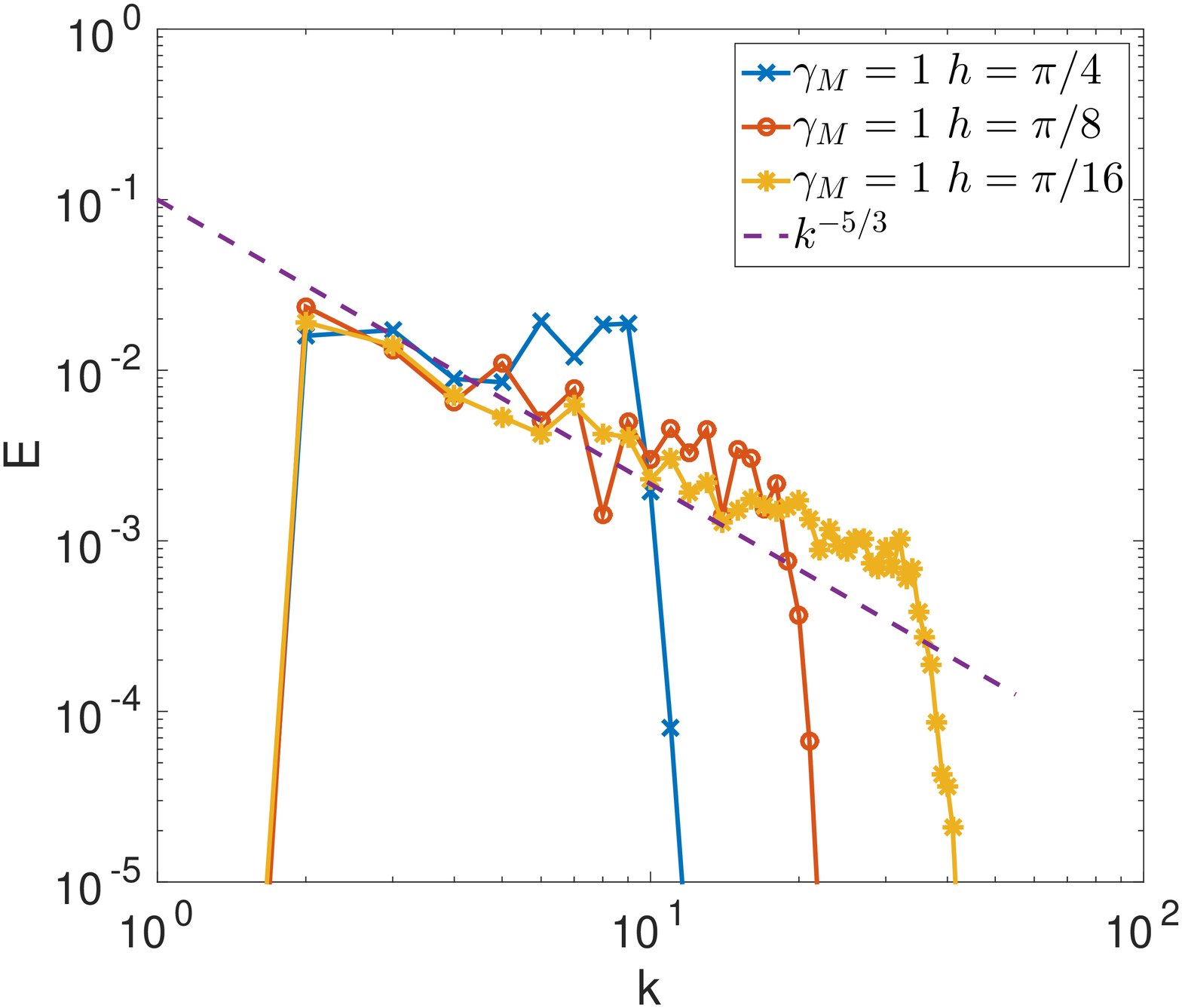}
\caption{$\gamma=1$}
  \end{subfigure}\hfill
  \begin{subfigure}[b]{.39\textwidth}
\includegraphics[trim= 0cm 0cm 2cm 1cm, clip, width=\textwidth]{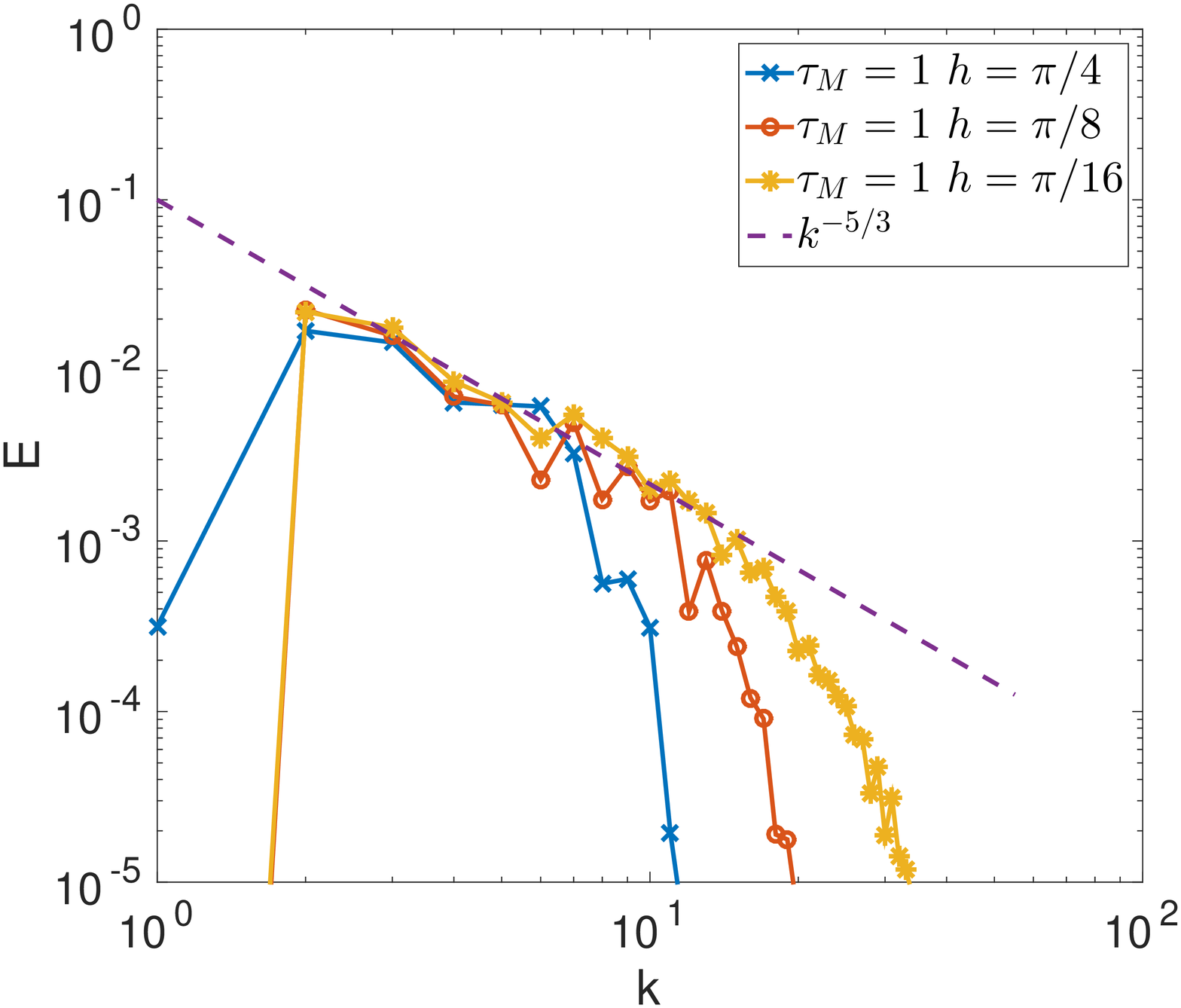}
\caption{$\gamma=1,~\tau_M=1$}
  \end{subfigure}\hspace*{\fill}
\caption{Energy spectra at $t=9$ for different mesh widths; $a=2\pi, b=1$.}
\label{fig:TGV_energy_GDandSU}
\end{figure}

Figure~\ref{fig:TGV_energy_dependenceh}(A) shows that the obtained results are comparable to those of the Smagorinsky model that is also known to bee too dissipative.
Dimensional analysis suggests to choose the parameter according to $\tau_M^n=C h/\|\widetilde{\boldsymbol{u}}_h^n\|_{\infty,M}$.
Interestingly, this choice performs comparably well as simply $\tau_M=1$ (cf. Figure~\ref{fig:TGV_energy_dependenceh}(B)).

In summary, we observe that grad-div stabilization is not sufficient in this example. However, additional LPS SU performs considerably well as implicit subgrid model in this test case.

\begin{figure}[htp]
\centering
\hspace*{\fill}\begin{subfigure}[b]{.395\textwidth}
\includegraphics[trim= 0cm 0cm 2cm 1cm, clip, width=\textwidth]{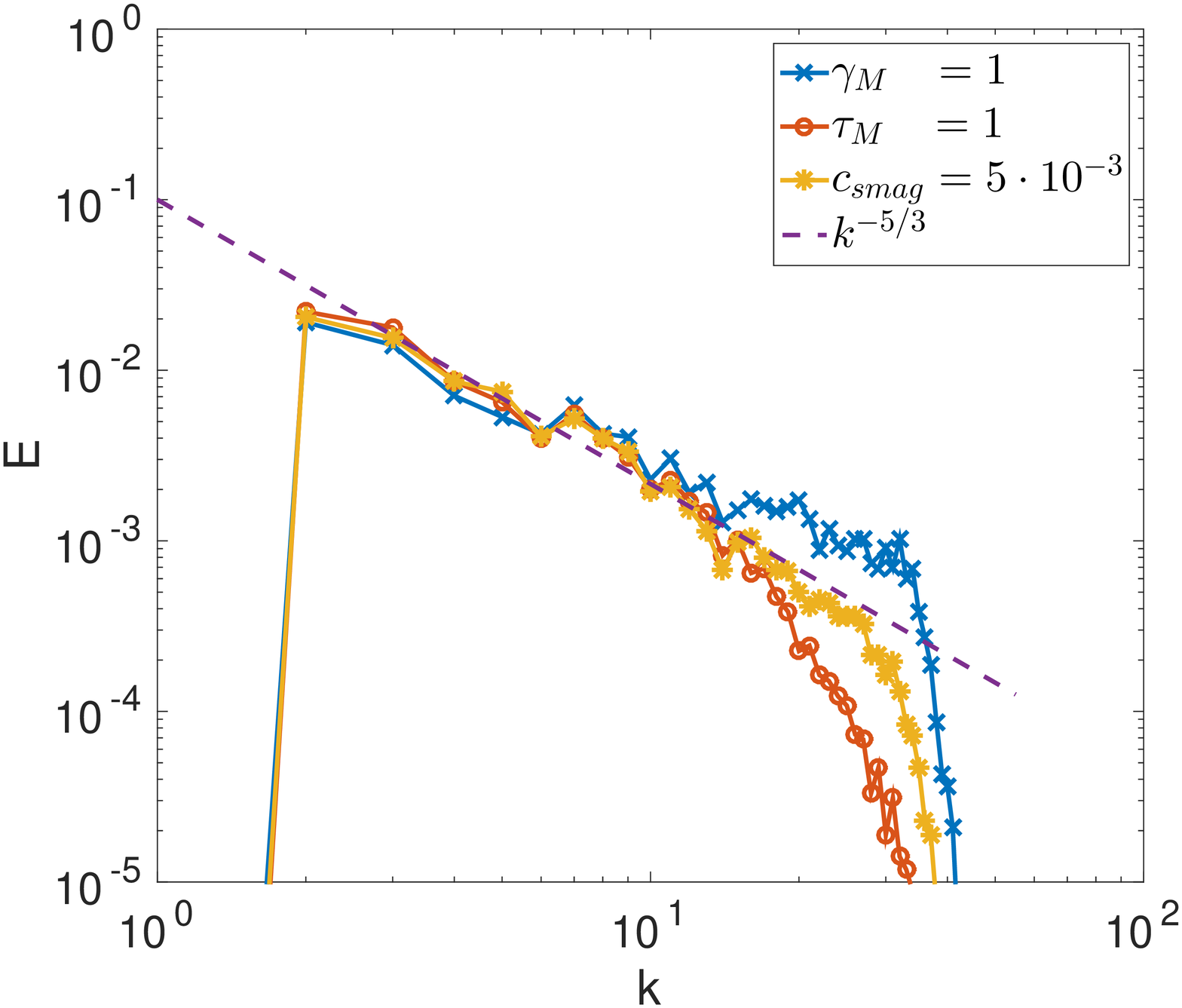}
\caption{$a=2\pi, b=1, h=\pi/16, \nu=10^{-4}$}
  \end{subfigure}\hfill
  \begin{subfigure}[b]{.395\textwidth}
\includegraphics[trim= 0cm 0cm 2cm 1cm, clip, width=\textwidth]{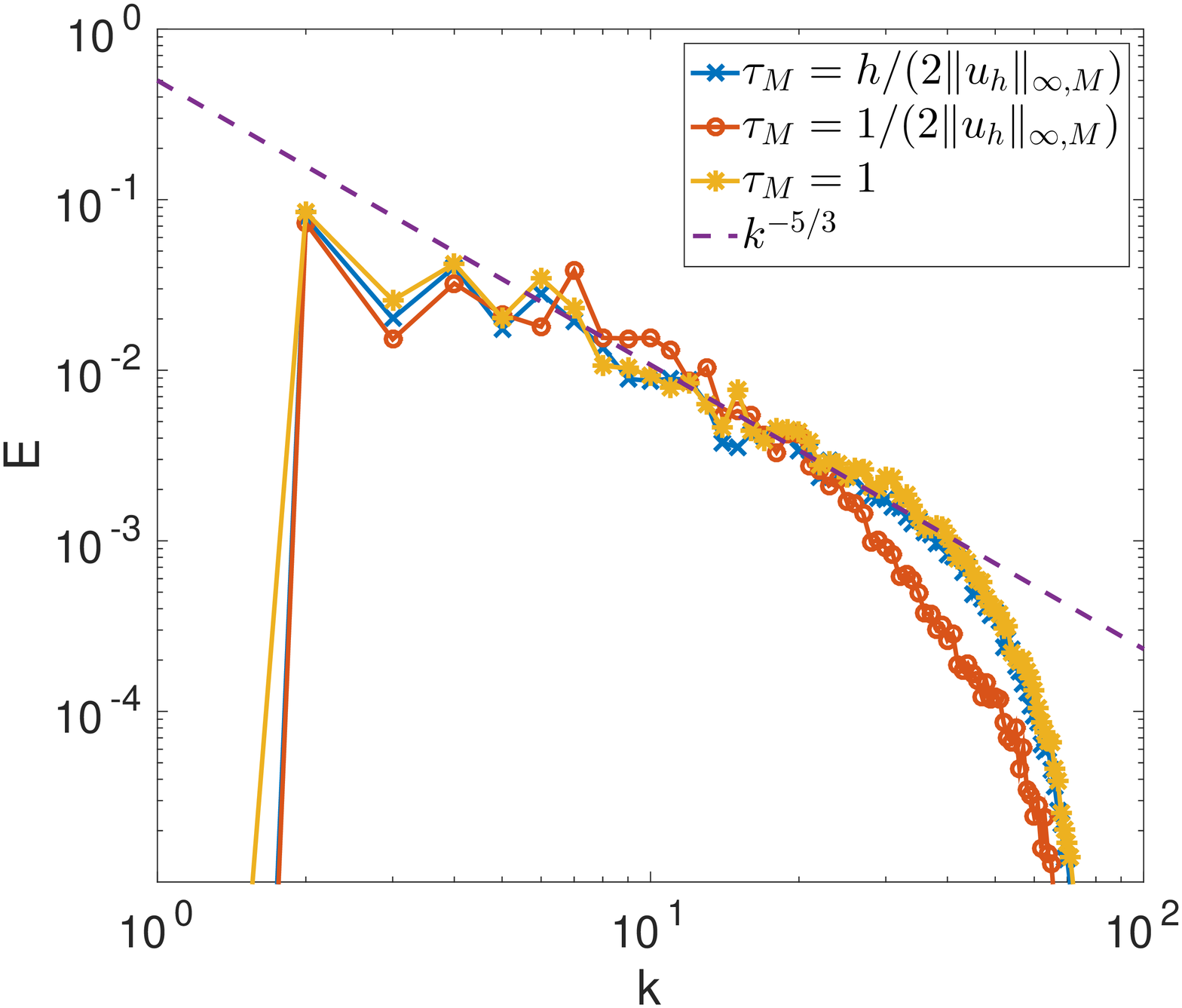}
\caption{$a=8/\sqrt{3}, b=1/10, h=1/(8\sqrt{3}), \nu=10^{-5}$}
  \end{subfigure}\hspace*{\fill}
\caption{Energy spectra for different stabilization models}
\label{fig:TGV_energy_dependenceh}
\end{figure}

\section{Summary}\label{section:discussion}

We considered two different approaches for error estimates of the time-dependent Navier-Stokes problem discretized in space by a stabilized finite element approach and in time by a BDF2-based projection 
algorithm. \\
In the first approach (Section~\ref{sec:full1}), we achieved quasi-robust error 
estimates for all considered velocity error norms. The obtained rates of 
convergence are quasi-optimal apart from the temporal convergence for the 
velocity energy error. A main ingredient for these results is a careful estimate 
of the convective terms. Interestingly, the SU stabilization is not necessary 
for these findings; grad-div stabilization is sufficient for robust error 
estimates without any restrictions on the time step size or the mesh width.\\
The second approach (Section~\ref{sec:temporal-spatial}) aimed to improve the above estimates in terms of the suboptimal rate of convergence for the energy error. Unfortunately, in this approach we were not able to mirror the estimates for the convective terms and hence did not achieve quasi-robust error estimates. 
Additionally, this approach requires some restrictions on the time step size and the mesh width with respect to $\nu$. However, introducing an intermediate time-dependent Stokes problem allowed us to get quasi-optimal error estimates with respect to time in the limiting case $h\to 0$. Opposed to the first approach, we here needed to assume additional regularity on some intermediate solution. The main problem for omitting this intermediate step is the treatment of the nonlinear stabilization for an estimate on in the norm given by the inverse Stokes operator.\\
Section~\ref{sec:optimal} allowed us to combine the results from the two previous sections to get quasi-optimal error estimates for the energy norm and the LPS norm of the velocity as well as for the pressure error in the $L^2(\Omega)$ norm.\\
In combination, we observe tight restrictions on the LPS parameter $\tau_M^n$ if 
we want to control the energy discretization error comparably good to the 
interpolation error. On the other hand, the requirements for the LPS error are 
relatively mild as the results for that norm are improved by the second 
approach.\\
These findings are confirmed by the numerical examples in Section~\ref{sec:numerics}. For an academic example with a suitable chosen forcing term, we compared the impact of a rotational correction as well as the impact of the grad-div stabilization. Interestingly, we observed rates of convergence that are much better as the ones proven in this work and often also better than the ones proven for the rotational scheme in the Stokes case in \cite{guermond2004error}. Furthermore, we clearly saw that the advantage that a rotational formulation gives vanishes for decreasing $\nu$. In these cases the influence of the grad-div stabilization becomes dominant. Basically the velocity error is decreased by the factor $1/\sqrt{\nu}$. On the other hand, for small $\nu$ we see a negative effect on the pressure error if the error due to time discretization is dominant.\\
Due to the fact that the first test case showed no dependence on the LPS parameter, we considered afterwards homogeneous decaying isotropic turbulence in the form of a Taylor-Green vortex.
We observe that we need to add the LPS SU term to the grad-div stabilization to resolve subgrid scales. For this simple problem we achieve with this model satisfying results comparable to the Smagorinsky model.

\appendix

\section{Splittings for the Discretized Time Derivative}
In this work, we need quite often a splitting for terms of the form
\begin{align*}
 \langle 3f(n)-4g(n-1)+g(n-2), f(n)\rangle
\end{align*}
where $\langle\cdot,\cdot\rangle$ denotes a symmetric bilinear form.
Some auxiliary algebraic identities are:
\begin{align*}
 2\langle a, a-b\rangle &= |a|^2+|a-b|^2-|b|^2\\
 2\langle 3a-4b+c, a \rangle &= |a|^2+(4|a|^2-4\langle a, b\rangle +|b|^2)
  \\ &\quad
 +(|a|^2+4|b|^2+|c|^2-4\langle a, b\rangle-4\langle b, c\rangle+2\langle a, c\rangle)\\
 &\quad-|b|^2-(4|b|^2-4\langle b, c\rangle+|c|^2)\\
 &=|a|^2+|2a-b|^2+|a-2b+c|^2-|b|^2-|2b-c|^2
\end{align*}
where $|a|^2$ is an abbreviation for $\langle a, a \rangle$.
Using the abbreviations
\begin{align*}
  a=f(n), \quad b=g(n-1), \quad c=g(n-2), \quad d=g(n), 
 \end{align*}
we obtain for the desired term
 \begin{align}
 \begin{aligned}
 &2\langle 3f(n)-4g(n-1)+g(n-2), f(n)\rangle\\&=2\langle 3a-4b+c, a\rangle 
 =6\langle a- d, a\rangle +2\langle 3d-4b+c, a- d\rangle +2\langle 3d -4b+c, d\rangle \\
&=3|a|^2-3|a-d|^2-3|d|^2+2\langle 3d-4b+c, a- d\rangle
 \\&\quad
+|d|^2+|2d-b|^2+|d-2b+c|^2-|b|^2-|2b-c|^2\\
&=3|a|^2-3|a-d|^2-2|d|^2+2\langle 3d-4b+c, a- d\rangle
 \\&\quad
+|2d-b|^2+|d-2b+c|^2-|b|^2-|2b-c|^2\\
&=3|f(n)|^2-3|f(n)-g(n)|^2-2|g(n)|^2
 \\&\quad
+2\langle 3g(n)-4g(n-1)+g(n-2), f(n)- g(n)\rangle\\
&\quad+|2g(n)-g(n-1)|^2+|\partial_{tt} g(n)|^2
-|g(n-1)|^2-|2g(n-1)-g(n-2)|^2
 \end{aligned}
 \label{eqn:splitting}
\end{align}
where $\delta_t$ is the the propagation operator defined by $\delta_t f(n):=f(n)-f(n-1)$.

\section{Estimates for the Convective Term}\label{appendixconvective}

\begin{lemma}\label{lemma:convectivebound1}
For $d\leq 3$ the convective term
\begin{align*}
 c(\boldsymbol{u};\boldsymbol{v}, \boldsymbol{w}):=\frac12(((\boldsymbol{u}\cdot\nabla)\boldsymbol{v}, \boldsymbol{w})-((\boldsymbol{u}\cdot\nabla)\boldsymbol{v}, \boldsymbol{w}))
\end{align*}
can be  bounded according to
\begin{align}
\begin{aligned}
 c(\boldsymbol{u};\boldsymbol{v}, \boldsymbol{w})&\leq C
 \begin{cases}
 \|\boldsymbol{u}\|_1\|\boldsymbol{v}\|_1\|\boldsymbol{w}\|_1  
 & \forall \boldsymbol{u}, \boldsymbol{v}, \boldsymbol{w} \in H_0^1(\Omega), \\
 \|\boldsymbol{u}\|_2\|\boldsymbol{v}\|_0\|\boldsymbol{u}\|_1   
 & \forall \boldsymbol{u}\in H^2(\Omega)\cap H_0^1(\Omega), \boldsymbol{v}, \boldsymbol{w} \in H_0^1(\Omega), \\
 \|\boldsymbol{u}\|_2\|\boldsymbol{v}\|_1\|\boldsymbol{w}\|_0   
 & \forall \boldsymbol{u}\in H^2(\Omega)\cap H_0^1(\Omega), \boldsymbol{v}, \boldsymbol{w} \in H_0^1(\Omega), \\
 \|\boldsymbol{u}\|_1\|\boldsymbol{v}\|_2\|\boldsymbol{w}\|_0   
 & \forall \boldsymbol{v}\in H^2(\Omega)\cap H_0^1(\Omega), \boldsymbol{u}, \boldsymbol{w} \in H_0^1(\Omega)
 \end{cases}\\
 c(\boldsymbol{u}, \boldsymbol{v}, \boldsymbol{u})&\leq C \|\boldsymbol{u}\|_0^{1/2}\|\boldsymbol{u}\|_1^{3/2}\|\boldsymbol{v}\|_1 \quad \forall \boldsymbol{u}, \boldsymbol{v}\in H_0^1(\Omega).
 \end{aligned}
 \label{eqn:estimate_conv}  
\end{align}
\end{lemma}
\begin{proof}
\cite{temam1995navier} using Sobolev and H\"older inequalities.
\end{proof}

\begin{lemma} \label{lem:convectiveterms}
Consider solutions $(\boldsymbol{u}(t_n), p^n)\in \boldsymbol{V}^{div} \times Q$, $(\widetilde{\boldsymbol{u}}_h^n, p_h^n) \in \boldsymbol{V}\!_h^{\,div} \times Q_h$ of the continuous and fully discretized equations satisfying 
$\boldsymbol{u}(t_n)\in  [W^{1,\infty}(\Omega)]^d$ and $\widetilde{\boldsymbol{u}}_h^n\in  [L^{\infty}(\Omega)]^d$.
We can estimate the difference of the convective terms by means of 
$\boldsymbol{\eta}_u^n=\boldsymbol{u}(t_n)-j_u\boldsymbol{u}(t_n)$, 
$\boldsymbol{e}_u^n=j_u\boldsymbol{u}(t_n)-\boldsymbol{u}_h^n$
and 
$\widetilde{\boldsymbol{e}}_u^n=\boldsymbol{u}(t_n)-\widetilde{\boldsymbol{u}}_h^n$
as
\begin{align*}
&c(\boldsymbol{u}(t_n);\boldsymbol{u}(t_n), \widetilde{\boldsymbol{e}}_u^n) - c(\widetilde{\boldsymbol{u}}_h^n, \widetilde{\boldsymbol{u}}_h^n, \widetilde{\boldsymbol{e}}_u^n) 
\\& 
\le \Bigg[ |\boldsymbol{u}(t_n)|_{W^{1,\infty}} + C h^{2z} |\boldsymbol{u}(t_n)|_{W^{1,\infty}}^2 
+  C\frac{h^{2z}}{\gamma} | \boldsymbol{u}(t_n)|^2_{W^{1,\infty}} 
 \\&\qquad\quad
+ C \frac{h^{2z-2}}{\gamma}\|\boldsymbol{u}(t_n)\|_{\infty}^2 
+ C h^{2z-2} \|\widetilde{\boldsymbol{u}}_h^n\|_{\infty}^2\Bigg] 
\| \widetilde{\boldsymbol{e}}_u^n\|_0^2
\\&\quad
+C h^{-2z}\| \boldsymbol{\eta}_u^n\|_0^2  
+ 3 C h^{2-2z}||| \boldsymbol{\eta}_u^n|||^2_{LPS}  +   3 C h^{2-2z}||| \widetilde{\boldsymbol{e}}_u^n|||^2_{LPS}
\end{align*}
with $C$ independent of $h$, the problem parameters and the solutions and $z\in\{0,1\}$.
\end{lemma}
\begin{proof} 
We remember that the Stokes interpolation operator $j_u\colon \boldsymbol{V}\to \boldsymbol{V}\!_h$ is discretely divergence-preserving.
With the splitting $\boldsymbol{\eta}_u^n+\widetilde{\boldsymbol{e}}_u^n= (\boldsymbol{u}(t_n)-j_u \boldsymbol{u}(t_n)) + (j_u \boldsymbol{u}(t_n)-\widetilde{\boldsymbol{u}}_h^n)$ and integration by parts, we have
\begin{align*}
 &c(\boldsymbol{u}(t_n);\boldsymbol{u}(t_n), \widetilde{\boldsymbol{e}}_u^n) - c(\widetilde{\boldsymbol{u}}_h^n;\widetilde{\boldsymbol{u}}_h^n, \widetilde{\boldsymbol{e}}_u^n)  \\
& = \underbrace{((\boldsymbol{u}(t_n)-\widetilde{\boldsymbol{u}}_h^n) \cdot \nabla \boldsymbol{u}(t_n), \widetilde{\boldsymbol{e}}_u^n)}_{=: T_1} + 
\underbrace{(\widetilde{\boldsymbol{u}}_h^n \cdot \nabla (\boldsymbol{u}(t_n)-j_u\boldsymbol{u}(t_n)), \widetilde{\boldsymbol{e}}_u^n)}_{=: T_2}
\\&\quad
- \frac12 \underbrace{((\nabla \cdot \widetilde{\boldsymbol{u}}_h^n)j_u\boldsymbol{u}(t_n), \widetilde{\boldsymbol{e}}_u^n)}_{=: T_3}.
\end{align*}
Now, we bound each term separately. Using Young's inequality with $C >0$, we calculate:
\begin{align}
\nonumber
T_1  & \le \sum_{M\in\mathcal M_h} \| \nabla \boldsymbol{u}(t_n)\|_{\infty,M} \Big( \| \widetilde{\boldsymbol{e}}_u^n\|_{0,M}^2
+    \| \boldsymbol{\eta}_u^n\|_{0,M}  \| \widetilde{\boldsymbol{e}}_u^n\|_{0,M} \Big) \\
\label{eqn:term-T1}
& \le |\boldsymbol{u}(t_n)|_{W^{1,\infty}} \| \widetilde{\boldsymbol{e}}_u^n\|_0^2
+  \sum_{M\in\mathcal M_h} | \boldsymbol{u}(t_n)|_{W^{1,\infty}(M)} \| \boldsymbol{\eta}_u^n\|_{0,M} \|\widetilde{\boldsymbol{e}}_u^n\|_{0,M} \\
& \le \frac{1}{4C} h^{-2z}\| \boldsymbol{\eta}_u^n \|_0^2 + 
\Big( |\boldsymbol{u}(t_n)|_{W^{1,\infty}} + C  h^{2z}|\boldsymbol{u}(t_n)|_{W^{1,\infty}}^2  \Big)   \| \widetilde{\boldsymbol{e}}_u^n\|_0^2.
\nonumber
\end{align}
For the term $T_2$, we have via integration by parts
\begin{align*}
T_2  &=   (\widetilde{\boldsymbol{u}}_h^n \cdot \nabla \boldsymbol{\eta}_u^n, \widetilde{\boldsymbol{e}}_u^n)
=  - (\widetilde{\boldsymbol{u}}_h^n \cdot \nabla \widetilde{\boldsymbol{e}}_u^n, \boldsymbol{\eta}_u^n) - ((\nabla \cdot \widetilde{\boldsymbol{u}}_h^n)\widetilde{\boldsymbol{e}}_u^n, 
\boldsymbol{\eta}_u^n) =: T_{21}+T_{22}.
\end{align*}
Term $T_{21}$ is the most critical one.
We calculate using the local inverse inequality (Assumption~\ref{assumption-A.2}) and Young's inequality:
\begin{align}
\begin{aligned}
T_{21} & = - (\widetilde{\boldsymbol{u}}_h^n \cdot \nabla \widetilde{\boldsymbol{e}}_u^n, \boldsymbol{\eta}_u^n) 
\le \sum_{M\in\mathcal M_h} \|\widetilde{\boldsymbol{u}}_h^n\|_{\infty,M} \| \nabla \widetilde{\boldsymbol{e}}_u^n\|_{0,M} \| \boldsymbol{\eta}_u^n\|_{0,M} \\ 
&\le C \sum_{M\in\mathcal M_h} \|\widetilde{\boldsymbol{u}}_h^n\|_{\infty,M} \| \widetilde{\boldsymbol{e}}_u^n\|_{0,M}h_M^{-1}\| \boldsymbol{\eta}_u^n\|_{0,M} 
 \\ &
\le  C h^{2z-2}\|\widetilde{\boldsymbol{u}}_h^n\|_{\infty}^2 \| \widetilde{\boldsymbol{e}}_u^n\|_{0}^2  
+ C h^{-2z}\| \boldsymbol{\eta}_u^n\|_0^2.
\label{eqn:term-T21-ohneA5}
\end{aligned}
\end{align}
Using $(\nabla \cdot \boldsymbol{u}(t_n), q)=0$ for all $q\in L^2(\Omega)$ and Young's inequality with $C >0$, we obtain 
\begin{align}
 \nonumber
T_{22} & =  - ((\nabla \cdot \widetilde{\boldsymbol{u}}_h^n)\boldsymbol{\eta}_u^n, \widetilde{\boldsymbol{e}}_u^n)
=  ((\nabla \cdot (\boldsymbol{\eta}_u^n + \widetilde{\boldsymbol{e}}_u^n))\boldsymbol{\eta}_u^n, \widetilde{\boldsymbol{e}}_u^n)
\\&\nonumber
\le \sum_{M\in\mathcal M_h} \| \boldsymbol{\eta}_u^n\|_{\infty,M} \Big( \| \nabla \cdot \widetilde{\boldsymbol{e}}_u^n\|_{0,M} + 
\|\nabla \cdot \boldsymbol{\eta}_u^n \|_{0,M} \Big) \| \widetilde{\boldsymbol{e}}_u^n\|_{0,M} \\
\label{eqn:term-T22}
& \le \sum_{M\in\mathcal M_h} \frac{Ch_M}{\sqrt{\gamma}}|\boldsymbol{u}(t_n)|_{W^{1,\infty}(M)}\sqrt{\gamma}
\Big( \|\nabla \cdot \widetilde{\boldsymbol{e}}_u^n\|_{0,M} + \| \nabla \cdot \boldsymbol{\eta}_u^n\|_{0,M} \Big) \| \widetilde{\boldsymbol{e}}_u^n\|_{0,M} \\
 \nonumber
& \le C h^{2-2z}||| \boldsymbol{\eta}_u^n|||^2_{LPS} + C h^{2-2z}||| \widetilde{\boldsymbol{e}}_u^n|||^2_{LPS}
+  C\frac{h^{2z}}{\gamma} | \boldsymbol{u}(t_n)|^2_{W^{1,\infty}}  \| \widetilde{\boldsymbol{e}}_u^n\|_0^2.
\end{align}
For $T_3$ we have the splitting
\begin{align*}
 T_3  =   ((\nabla \cdot \widetilde{\boldsymbol{u}}_h^n)j_u\boldsymbol{u}(t_n), \widetilde{\boldsymbol{e}}_u^n) 
=   -((\nabla \cdot \widetilde{\boldsymbol{u}}_h^n)\boldsymbol{\eta}_u^n, \widetilde{\boldsymbol{e}}_u^n)
+((\nabla \cdot \widetilde{\boldsymbol{u}}_h^n) \boldsymbol{u}(t_n), \widetilde{\boldsymbol{e}}_u^n) = T_{22} + T_{32}
\end{align*}
and use the same estimate as in (\ref{eqn:term-T22}).\\
Utilizing $(\nabla \cdot \boldsymbol{u}(t_n), q)=0$ for all $q\in L^2(\Omega)$ and Young's inequality we obtain
\begin{align}
\begin{aligned}
|T_{32}| & =  |(\nabla \cdot \widetilde{\boldsymbol{u}}_h^n, \boldsymbol{u} \cdot \widetilde{\boldsymbol{e}}_u^n)|
= | (\nabla \cdot (-\boldsymbol{\eta}_u^n - \widetilde{\boldsymbol{e}}_u^n + \boldsymbol{u}(t_n)), \boldsymbol{u}(t_n) \cdot \widetilde{\boldsymbol{e}}_u^n)|\\
&\le  | (\nabla \cdot \boldsymbol{\eta}_u^n, \boldsymbol{u}(t_n) \cdot \widetilde{\boldsymbol{e}}_u^n) | +|(\nabla \cdot \widetilde{\boldsymbol{e}}_u^n, \boldsymbol{u}(t_n) \cdot \widetilde{\boldsymbol{e}}_u^n)|\\
&\le   \sum_{M\in\mathcal M_h} \Big( \|\boldsymbol{u}(t_n)\|_{\infty,M} \sqrt{\gamma}\|\nabla \cdot \boldsymbol{\eta}_u^n\|_{0,M} \frac{1}{\sqrt{\gamma}} \|\widetilde{\boldsymbol{e}}_u^n\|_{0,M} 
 \\&\qquad\qquad\quad 
+ \|\boldsymbol{u}(t_n)\|_{\infty,M} \sqrt{\gamma}\|\nabla \cdot \widetilde{\boldsymbol{e}}_u^n\|_{0,M}  \frac{1}{\sqrt{\gamma}}\|\widetilde{\boldsymbol{e}}_u^n\|_{0,M} \Big)
\label{eqn:term-T32}
\\&
\le C h^{2-2z}||| \boldsymbol{\eta}_u^n|||_{LPS}^2 
+ C h^{2-2z}||| \widetilde{\boldsymbol{e}}_u^n|||_{LPS}^2 
+ C \frac{h^{2z-2}}{\gamma}\|\boldsymbol{u}(t_n)\|_{\infty}^2  \|\widetilde{\boldsymbol{e}}_u^n\|_0^2.
\end{aligned}
\end{align}
Combining the above bounds (\ref{eqn:term-T1})-(\ref{eqn:term-T32}) yields the claim.
\end{proof}

\begin{lemma}\label{lem:convective2}
 The error with respect to the convective terms can be estimated as
   \begin{align*}
 &c(\boldsymbol{u}(t_n), \boldsymbol{u}(t_n), \widetilde{\boldsymbol{e}}_{u,h}^n)- 
c(\widetilde{\boldsymbol{u}}_h^n, \widetilde{\boldsymbol{u}}_h^n, \widetilde{\boldsymbol{e}}_{u,h}^n) \\
&\leq 
C\|\widetilde{\boldsymbol{e}}_{u,h}^n\|_0^2\left(\frac{\|\widetilde{\boldsymbol{\eta}}_{u,t}^n\|_1^4+\|\widetilde{\boldsymbol{\eta}}_{u,h}^n\|_1^4}{\nu^3}+\frac{\|\boldsymbol{u}(t_n)\|_2^2}{\nu}\right)+\frac{\nu}{4}\|\widetilde{\boldsymbol{e}}_{u,h}^n\|_1^2 
 \\&\quad
+\frac{C}{\nu}(\|\widetilde{\boldsymbol{\eta}}_{u,t}^n\|_1^4+\|\widetilde{\boldsymbol{\eta}}_{u,h}^n\|_1^4+\|\widetilde{\boldsymbol{\eta}}_{u,h}^n\|_0^2\|\boldsymbol{u}(t_n)\|_2^2)
\end{align*}
 with $\widetilde{\boldsymbol{e}}_{u,h}^n$, $\widetilde{\boldsymbol{\eta}}_{u,h}^n$ and $\widetilde{\boldsymbol{\eta}}_{u,t}^n$
 as defined in (\ref{eqn:error_semi2t}), (\ref{eqn:error_semi2h_eta}) and (\ref{eqn:error_semi2h_e}).
\end{lemma}
\begin{proof}
 Due to $\widetilde{\boldsymbol{\eta}}_{u,t}^n+\widetilde{\boldsymbol{\xi}}_{u,h}^n= \boldsymbol{u}(t_n)- \widetilde{\boldsymbol{u}}_h^n$ 
we calculate for the convective term using skew-symmetry and the estimates from Lemma~\ref{lemma:convectivebound1}
\begin{align}
 \begin{aligned}
&c(\boldsymbol{u}(t_n), \boldsymbol{u}(t_n), \widetilde{\boldsymbol{e}}_{u,h}^n)- 
c(\widetilde{\boldsymbol{u}}_h^n, \widetilde{\boldsymbol{u}}_h^n, \widetilde{\boldsymbol{e}}_{u,h}^n) 
\\&
= 
c(\widetilde{\boldsymbol{\xi}}_{u,h}^n, \boldsymbol{u}(t_n), \widetilde{
\boldsymbol{e}}_{u,h}^n) + 
c(\widetilde{\boldsymbol{u}}_h^n, \widetilde{\boldsymbol{\xi}}_{u,h}^n, \widetilde{
\boldsymbol{e}}_{u,h}^n)+
c(\widetilde{\boldsymbol{\eta}}_{u,t}^n, \boldsymbol{u}(t_n), \widetilde{
\boldsymbol{e}}_{u,h}^n) + 
c(\widetilde{\boldsymbol{u}}_h^n, \widetilde{\boldsymbol{\eta}}_{u,t}^n, \widetilde{
\boldsymbol{e}}_{u,h}^n)
\\&
=c(\widetilde{\boldsymbol{\xi}}_{u,h}^n, \boldsymbol{u}(t_n), \widetilde{
\boldsymbol{e}}_{u,h}^n) + 
c(\boldsymbol{u}(t_n)-\widetilde{\boldsymbol{\eta}}_{u,t}^n-\widetilde{\boldsymbol{\xi}}_{u,h}^n, \widetilde{\boldsymbol{\eta}}_{u,h}^n, \widetilde{
\boldsymbol{e}}_{u,h}^n)\\&\quad+
c(\widetilde{\boldsymbol{\eta}}_{u,t}^n, \boldsymbol{u}(t_n), \widetilde{
\boldsymbol{e}}_{u,h}^n) + 
c(\boldsymbol{u}(t_n)-\widetilde{\boldsymbol{\eta}}_{u,t}^n-\widetilde{\boldsymbol{\xi}}_{u,h}^n, \widetilde{\boldsymbol{\eta}}_{u,t}^n, \widetilde{
\boldsymbol{e}}_{u,h}^n)
\\&
=c(\widetilde{\boldsymbol{\eta}}_{u,h}^n+\widetilde{\boldsymbol{\eta}}_{u,t}^n, \boldsymbol{u}(t_n), \widetilde{
\boldsymbol{e}}_{u,h}^n) 
+c(\boldsymbol{u}(t_n), \widetilde{\boldsymbol{\eta}}_{u,h}^n+\widetilde{\boldsymbol{\eta}}_{u,t}^n, \widetilde{
\boldsymbol{e}}_{u,h}^n)
+c(\widetilde{\boldsymbol{e}}_{u,h}^n, \boldsymbol{u}(t_n), \widetilde{
\boldsymbol{e}}_{u,h}^n) 
\\&\quad
- c(\widetilde{\boldsymbol{\eta}}_{u,t}^n+\widetilde{\boldsymbol{\eta}}_{u,h}^n, \widetilde{\boldsymbol{\eta}}_{u,t}^n+\widetilde{\boldsymbol{\eta}}_{u,h}^n, \widetilde{
\boldsymbol{e}}_{u,h}^n)
+c(\widetilde{\boldsymbol{e}}_{u,h}^n, \widetilde{\boldsymbol{\eta}}_{u,h}^n+\widetilde{\boldsymbol{\eta}}_{u,t}^n, \widetilde{
\boldsymbol{e}}_{u,h}^n).
\end{aligned}
\end{align}

These terms can be estimated as
\begin{align*}
 c(\widetilde{\boldsymbol{\eta}}_{u,h}^n+\widetilde{\boldsymbol{\eta}}_{u,t}^n, \boldsymbol{u}(t_n), \widetilde{
\boldsymbol{e}}_{u,h}^n)
&\leq C\|\widetilde{\boldsymbol{\eta}}_{u,h}^n+\widetilde{\boldsymbol{\eta}}_{u,t}^n\|_0\|\boldsymbol{u}(t_n)\|_2\|\widetilde{
\boldsymbol{e}}_{u,h}^n\|_1
\\&
\leq \frac{C}{\nu}\|\widetilde{\boldsymbol{\eta}}_{u,h}^n+\widetilde{\boldsymbol{\eta}}_{u,t}^n\|_0^2\|\boldsymbol{u}(t_n)\|_2^2+\frac{\nu}{32}\|\widetilde{
\boldsymbol{e}}_{u,h}^n\|_1^2
 \\
c(\boldsymbol{u}(t_n), \widetilde{\boldsymbol{\eta}}_{u,h}^n+\widetilde{\boldsymbol{\eta}}_{u,t}^n, \widetilde{
\boldsymbol{e}}_{u,h}^n)
&\leq C\|\widetilde{\boldsymbol{\eta}}_{u,h}^n+\widetilde{\boldsymbol{\eta}}_{u,t}^n\|_0\|\boldsymbol{u}(t_n)\|_2\|\widetilde{
\boldsymbol{e}}_{u,h}^n\|_1 
\\&
\leq \frac{C}{\nu}\|\widetilde{\boldsymbol{\eta}}_{u,h}^n+\widetilde{\boldsymbol{\eta}}_{u,t}^n\|_0^2\|\boldsymbol{u}(t_n)\|_2^2+\frac{\nu}{32}\|\widetilde{
\boldsymbol{e}}_{u,h}^n\|_1^2
 \\
c(\widetilde{\boldsymbol{e}}_{u,h}^n, \boldsymbol{u}(t_n), \widetilde{
\boldsymbol{e}}_{u,h}^n) 
&\leq C\|\widetilde{\boldsymbol{e}}_{u,h}^n\|_0\|\boldsymbol{u}(t_n)\|_2\|\widetilde{
\boldsymbol{e}}_{u,h}^n\|_1
\\&
\leq \frac{C}{\nu}\|\widetilde{\boldsymbol{e}}_{u,h}^n\|_0^2\|\boldsymbol{u}(t_n)\|_2^2+\frac{\nu}{16}\|\widetilde{
\boldsymbol{e}}_{u,h}^n\|_1^2
\\
c(\widetilde{\boldsymbol{\eta}}_{u,t}^n+\widetilde{\boldsymbol{\eta}}_{u,h}^n, \widetilde{\boldsymbol{\eta}}_{u,t}^n+\widetilde{\boldsymbol{\eta}}_{u,h}^n, \widetilde{
\boldsymbol{e}}_{u,h}^n)
&\leq C\|\widetilde{\boldsymbol{\eta}}_{u,t}^n+\widetilde{\boldsymbol{\eta}}_{u,h}^n\|_1^2\|\widetilde{
\boldsymbol{e}}_{u,h}^n\|_1
\\&
\leq \frac{C}{\nu}\|\widetilde{\boldsymbol{\eta}}_{u,t}^n+\widetilde{\boldsymbol{\eta}}_{u,h}^n\|_1^4+\frac{\nu}{16}\|\widetilde{
\boldsymbol{e}}_{u,h}^n\|_1^2
\\
c(\widetilde{\boldsymbol{e}}_{u,h}^n, \widetilde{\boldsymbol{\eta}}_{u,h}^n+\widetilde{\boldsymbol{\eta}}_{u,t}^n, \widetilde{
\boldsymbol{e}}_{u,h}^n)
&\leq C\|\widetilde{\boldsymbol{e}}_{u,h}^n\|_0^{1/2}\|\widetilde{\boldsymbol{\eta}}_{u,h}^n+\widetilde{\boldsymbol{\eta}}_{u,t}^n\|_1\|\widetilde{
\boldsymbol{e}}_{u,h}^n\|_1^{3/2}
\\&
\leq \frac{C}{\nu^3}\|\widetilde{\boldsymbol{e}}_{u,h}^n\|_0^2\|\widetilde{\boldsymbol{\eta}}_{u,h}^n+\widetilde{\boldsymbol{\eta}}_{u,t}^n\|_1^4+\frac{\nu}{16}\|\widetilde{
\boldsymbol{e}}_{u,h}^n\|_1^2.
\end{align*}

For the last estimate we used the inequality
\begin{align*}
 a^{1/2}b^{3/2} &=(\nu^{-3/2}a)^{1/2}(\sqrt{\nu} b)^{3/2}\leq C (((\nu^{-3/2} a)^{1/2})^4+((\sqrt{\nu} 
b)^{3/2})^{4/3})
 \\ &
\leq  C (\nu^{-3} a^2+ \nu b^2).
\end{align*}
In combination, the error with respect to the convective terms is given by
\begin{align*}
 &c(\boldsymbol{u}(t_n), \boldsymbol{u}(t_n), \widetilde{\boldsymbol{e}}_{u,h}^n)- 
c(\widetilde{\boldsymbol{u}}_h^n, \widetilde{\boldsymbol{u}}_h^n, \widetilde{\boldsymbol{e}}_{u,h}^n) \\
&\leq 
C\|\widetilde{\boldsymbol{e}}_{u,h}^n\|_0^2\left(\frac{\|\widetilde{\boldsymbol{\eta}}_{u,t}^n\|_1^4+\|\widetilde{\boldsymbol{\eta}}_{u,h}^n\|_1^4}{\nu^3}+\frac{\|\boldsymbol{u}(t_n)\|_2^2}{\nu}\right)+\frac{\nu}{4}\|\widetilde{\boldsymbol{e}}_{u,h}^n\|_1^2 
 \\&\quad
+\frac{C}{\nu}(\|\widetilde{\boldsymbol{\eta}}_{u,t}^n\|_1^4+\|\widetilde{\boldsymbol{\eta}}_{u,h}^n\|_1^4+\|\widetilde{\boldsymbol{\eta}}_{u,h}^n\|_0^2\|\boldsymbol{u}(t_n)\|_2^2).
\end{align*}
\end{proof}

\section{Discrete Gronwall Lemma}
\begin{lemma}{\label{lem:Gronwall_discrete}}
Let $y^n, h^n, g^n, f^n$ be non-negative sequences satisfying for all $0\le m\le [T/k]$
\begin{align*}
y^m + k \sum_{n=0}^m h^n &\le  B + k\sum_{n=0}^m (g^n y^n+f^n)~~\text{with}~k\sum_{n=0}^{[T/k]} g^n \le M.
\end{align*}
Assume $kg^n<1$ and let $\sigma=\max_{0\le n\le [T/k]} (1-kg^n)^{-1}$. Then for all $0\le m\le [T/k]$ it holds
\begin{align}
y^m + k\sum_{n=1}^m h^n &\le \exp(\sigma M)\left( B+k\sum_{n=0}^m f^n\right).
\end{align}
\end{lemma}
\begin{proof}
A proof of this result can be found in \cite{temam1995navier}, for instance.
\end{proof}

\bibliography{NSE_full}
\bibliographystyle{ieeetr}
\end{document}